\documentclass{article}
\usepackage{microtype}
\usepackage{graphicx}
\usepackage{amssymb}
\usepackage{subfigure}
\usepackage{booktabs} 
\usepackage{lipsum}
\usepackage[utf8]{inputenc} 
\usepackage[T1]{fontenc}    
\usepackage{tgtermes}
\usepackage{microtype}   
\usepackage{framed}   
\usepackage[hidelinks]{hyperref}

\usepackage{url}            
\usepackage{booktabs}       
\usepackage{amsfonts}       
\usepackage{nicefrac}       
\usepackage{microtype}      
\usepackage[english]{babel}
\usepackage{hyperref}
\usepackage{amsmath}
\usepackage{enumitem}
\usepackage{amssymb}
\usepackage{graphicx}
\usepackage[table]{xcolor}
\usepackage{multirow}
\usepackage{bm}
\usepackage{theorem}

\usepackage{mathtools}
\newenvironment{proof}{\paragraph{\it Proof.}}{\hfill$\square$}

\usepackage[title]{appendix}
\usepackage{comment}
\usepackage{amsfonts}
\usepackage{dsfont}
\usepackage{hyperref}
\usepackage{lipsum}
\usepackage{dsfont,subcaption,float}
\usepackage{algorithm}
\usepackage{algorithmic}

\usepackage{thm-restate}
\usepackage[size=small,labelfont=bf]{caption}

\newtheorem{thm}{Theorem}
\newtheorem{lem}{Lemma}
\newtheorem{lemma}{Lemma}[section]
\newtheorem{prop}{Proposition}

\newtheorem{defi}{Definition}
\newtheorem{rem}{Remark}

\newtheorem{assump}{A\!}

\newtheorem{example}{Example}



\def\@#1{\pmb{#1}}

\def\ca#1{\mathcal{#1}}

\usepackage{xcolor}

\usepackage[final]{neurips_2024}

\usepackage[symbol]{footmisc}

\usepackage[utf8]{inputenc} 
\usepackage[T1]{fontenc}    
\usepackage{hyperref}       
\usepackage{url}            
\usepackage{booktabs}       
\usepackage{amsfonts}       
\usepackage{nicefrac}       
\usepackage{microtype}      
\usepackage{xcolor}         

\title{Performative Control for Linear Dynamical Systems}

\author{
  Songfu Cai\thanks{Equal contribution.}\,\,,\,\,\,\, Fei Han\footnotemark[1]\,\,,\,\,\,\, Xuanyu Cao\thanks{Corresponding author.}\\
  Department of Electronic and Computer Engineering\\
  The Hong Kong University of Science and Technology\\
  \texttt{eesfcai@ust.hk, fhanac@connect.ust.hk, eexcao@ust.hk}}

\begin{document}

\maketitle

\begin{abstract}
We introduce the framework of performative control, where the policy chosen by the controller affects the underlying dynamics of the control system. This results in a sequence of policy-dependent system state data with policy-dependent temporal correlations. Following the recent literature on performative prediction \cite{perdomo2020performative}, we introduce the concept of a performatively stable control (PSC) solution. We first propose a sufficient condition for the performative control problem to admit a unique PSC solution with a problem-specific structure of distributional sensitivity propagation and aggregation. We further analyze the impacts of system stability on the existence of the PSC solution. Specifically, for {almost surely strongly stable} policy-dependent dynamics, the PSC solution exists if the sum of the distributional sensitivities is small enough. However, for almost surely unstable policy-dependent dynamics, the existence of the PSC solution will necessitate a temporally backward decaying of the distributional sensitivities. We finally provide a repeated stochastic gradient descent scheme that converges to the PSC solution and analyze its non-asymptotic convergence rate. Numerical results validate our theoretical analysis.
\end{abstract}

\section{Introduction}
Control theory is a fundamental field of study in engineering
and mathematics that centers on coordinating the behaviors of dynamical
systems. It has extensive applications in aerospace, robotics,
manufacturing, economics, natural sciences, etc. It provides a framework
for designing control policies that regulate the system states, enabling
them to evolve in a desired manner over time dynamically. The linear state
space model is a powerful tool for representing dynamical systems,
which employs a set of difference equations to describe the Markovian
system state process with a linear state transition model. {Many of the existing works} on control of linear dynamical systems (LDSs) \cite{recht2019tour,gravell2020learning,zhao2023non,fazel2018global,mania2019certainty,venkataraman2019recovering,zhao2023non,zhou2023efficient} {are developed based on the} key assumption that the system state transition model is static.
However, {in many real world applications, such a static dynamics assumption usually does not hold because} system state transition model can be changed
by control policies. For instance, in the stock market, the investment
policy of well-known investors can impact the actions of the general
public investors, resulting in famed investment strategy-dependent
changes to the dynamics of stock prices. Another example is autonomous vehicle (AV). A deployed AV might change how the
pedestrians and other neighboring cars behave, and the resulting
traffic environment might be quite different from what the designers
of the AV had in mind \cite{nikolaidis2017game}. {Such interplay between decision-making and decision-dependent system dynamics is a pervasive phenomenon in a multitude of domains, including finance, transportation, and public policy, among others.}

Most of the existing works on the control of linear systems with non-static
state transition model are primarily focused on the additive random
perturbations to the system state transition matrix, where policy
optimization methods are employed to find the optimal control policy
that minimizes quadratic costs \cite{hambly2021policy,agarwal2019online,hazan2020nonstochastic,ghai2024online}. This type of problem
also includes the control of linear systems with additive state-dependent
noise \cite{liu2023data,duncan2017stochastic} or action-dependent noise \cite{yaghmaie2022linear,carlos2018analysis}, which
is equivalent to having additive perturbations on the state transition
matrix or the control input gain matrix, respectively. Some other
works on jump/switched linear systems formulate the variation of system
model as stochastic model jumps among multiple linear modes, where
the jumping law is governed by a finite time-homogeneous Markov process. However, in all these works, the changes of the system model are unrelated to the control policy.

\textit{Performative prediction }provides a systematic way to model
the interaction between decision-making and data via decision-dependent data distribution maps \cite{quinonero2022dataset}. {The pioneering work by \cite{perdomo2020performative} has led to a growing body of research dedicated to performative prediction problems.} Most of these studies are focused on establishing the conditions for the existence and uniqueness of the performative stable point and designing learning algorithms
with provable convergence to such a unique performative stable point \cite{li2022state,li2209multi,drusvyatskiy2023stochastic,brown2022performative,mendler2020stochastic,wood2021online,ray2022decision}, or algorithms to find a stationary solution of the performative risk \cite{izzo2021learn,miller2021outside,ray2022decision}.

An open question follows: Can the idea of performative prediction,
which tackles decision-dependent data in the learning/prediction domain, be employed to address the decision-dependent dynamics in the
control domain? Notice that in all the aforementioned performative
prediction works \cite{perdomo2020performative,li2022state,li2209multi,drusvyatskiy2023stochastic,brown2022performative,mendler2020stochastic,wood2021online,ray2022decision,izzo2021learn,miller2021outside,ray2022decision}, the data input to the learning algorithms
are independently generated based on the decision-dependent distributions.
No temporal correlation is considered or exploited between two consecutive sets of input data. However, in the context of control, the situation is different as the changing system dynamics introduce various additional
complexities. This is because the data of the system state at each time
the step will depend on the decision-dependent state transition matrices in all previous time steps. The expected total cost function to be
optimized will depend on all the decision-dependent state transition
matrices across the entire control time horizon. This implies that
we need a framework of \textit{performative control} that is more
general than the framework of performative prediction, and that can
accommodate a sequence of decision-dependent data with temporal correlations,
where the temporal correlations are again decision-dependent. 

In this work, we provide an affirmative answer to the above question. Our idea
hinges on adopting the concept of performative stable solution \cite{perdomo2020performative}, which is a fixed point solution for the
interplay between the controller and the system dynamics that react
to the controller's decisions. The condition for the existence of
such a performative stable solution is obtained by analyzing the
propagation and aggregation of sensitivities associated with the
distributions of policy-dependent system dynamics over the entire
control time horizon. In particular, we follow existing studies \cite{agarwal2019online,hazan2020nonstochastic,ghai2024online} and focus on the disturbance-action policy, which allows the
consideration of general convex control costs instead of only quadratic
control costs. 

To the best of our knowledge, this paper provides the first study
and analysis of performative control of linear systems where the system
dynamics are constantly changing in a control-policy-dependent manner.
We highlight the following key contributions:
\begin{itemize}
	\item We introduce the notion of performative control, where the deployed
	control policy affects the underlying dynamics of the control system.
	We provide sufficient conditions for the existence and uniqueness
	of the performative stable control (PSC) solution. The sufficient
	condition is expressed in terms of a weighted sum of all the distributional
	sensitivities associated with the policy-dependent system state transition
	matrices over the entire control horizon. An interesting finding is
	that the sufficient condition exhibits a structure of sensitivity propagation and aggregation, implying that it is preferable for sensitivities
	to be relatively small in the early stages of the system state evolution.
	
	\item We analyze the impacts of system stability on the existence and uniqueness of the PSC solution. We show that when the policy-dependent dynamics are {almost surely strongly stable}, the PSC solution exists if the sum of all distributional sensitivities is below a certain threshold. On the other hand, when the policy-dependent dynamics are almost surely unstable, the proposed sufficient condition for the existence of the PSC solution will place a necessary requirement on a temporally backwards decaying of the distributional sensitivities.
	
	\item We propose a repeated stochastic gradient descent (RSGD) scheme
	and analyze its convergence towards the PSC solution. We show that
	the scheme is convergent under the same sufficient condition for the existence
	of the PSC solution. With an appropriate step size rule, the expected squared distance between the PSC solution and the iterates decays
	as $\mathcal{O}\left(\frac{1}{N}\right),$ where $N$ is the iteration
	number.
\end{itemize}
Finally, we conduct experiments on policy-dependent stock investment
risk minimization problem. The numerical results validate the effectiveness
of our algorithm and theoretical analysis.

\subsection{Related Work} 

\textbf{Performative Prediction.} The notion of performative prediction is initiated by \cite{perdomo2020performative},
where performative stability is first introduced, and a sufficient and necessary condition is provided so that the performative stable point can be reached via iterative risk minimization algorithms. Since
then, there has been a growing literature analyzing the performative prediction problem and studying the convergence of learning algorithms to the performative stable point \cite{mendler2020stochastic,drusvyatskiy2023stochastic,brown2022performative,li2022state,wood2021online}.There are also some other papers that find algorithms that converge to a stationary solution of the performative risk \cite{izzo2021learn,miller2021outside,ray2022decision}. Our problem poses entirely new challenges because the decision-dependent data of system state and control costs also have decision-dependent temporal correlations. Such a two-layer decision-dependent structure cannot be incorporated into the existing performative prediction framework.

There are some very recent works on \textit{performative reinforcement learning}, where a Markov decision process (MDP) is considered and
the deployed policy not only changes the costs but also the underlying
state transition kernel \cite{mandal2023performative,rank2024performative}. However, these works only considered
the tabular MDP, where the state and action space are finite. The LDS considered in our work may also be viewed
as a Markov decision process with a decision-dependent linear transition
kernel and decision-dependent costs. But both the state and action
space are continuous and there are infinitely many admissible state-action
pairs, which are beyond the scopes of \cite{mandal2023performative,rank2024performative}.

{\textbf{Stateful Performative Prediction.} Note that most of the existing
	performative prediction works \cite{perdomo2020performative,mendler2020stochastic,mendler2020stochastic,drusvyatskiy2023stochastic,li2022state,wood2021online,drusvyatskiy2023stochastic,brown2022performative,li2022state,wood2021online} are non-stateful in the
	sense that, for any deployed policy $\theta$, the data sample $Z$
	{follows} a static distribution $\mathcal{D}\left(\theta\right),$ i.e., $Z\sim\mathcal{D}\left(\theta\right).$ The seminal work \cite{brown2022performative}
	generalizes the stateless framework of \cite{perdomo2020performative} and proposes
	a more general framework of stateful performative prediction via a
	stateful distribution transition map $f\left(\cdot\right)$. Specifically,
	the observed data distribution in \cite{brown2022performative} is time-varying and depends
	on the history of previously deployed polices, i.e., $\mathcal{D}_{t+1}=f\left(\mathcal{D}_{t},\theta_{t}\right)$.
	Consequently, in comparison to the conventional non-stateful performative prediction works
	\cite{perdomo2020performative,mendler2020stochastic,mendler2020stochastic,drusvyatskiy2023stochastic,li2022state,wood2021online,drusvyatskiy2023stochastic,brown2022performative,li2022state,wood2021online}, this framework is capable of encapsulating the phenomenon
	of strategic decision-making with outdated information, and serves
	as a foundation for investigations of the disparate effects of performativity
	(please refer to Examples 1-3 in \cite{brown2022performative} for more details). The interconnections
	and generalizations between our work and the stateful performative prediction will be substantiated
	in Section 2.}

\textbf{Nonstochastic Control.} Another line of relevant works pertains to non-stochastic control, which is initiated by \cite{agarwal2019online}. Various applications of non-stochastic
control can be found in \cite{hambly2021policy,agarwal2019online,hazan2020nonstochastic,ghai2024online}. At the core of non-stochastic
control is the disturbance-action control policy, which chooses the
action as a linear map of the past disturbances \cite{hazan2020nonstochastic}. Such a disturbance-action policy facilitates efficient algorithms for control problems with
arbitrary additive disturbances in the dynamics and arbitrary convex
control costs instead of quadratic costs only. In this paper, we adopt
the disturbance-action control policy for analysis. However, in contrast
to \cite{agarwal2019online,hambly2021policy,agarwal2019online,hazan2020nonstochastic,ghai2024online}, where the system dynamics are static, our analysis involves
policy-dependent nonstationary dynamics.

\section{Motivations}
In this section, we discuss the key motivations for our performative
linear control framework. We also outline the key connections and
generalizations of our proposed framework to the stateful performative prediction framework
of \cite{brown2022performative}. 

\textbf{Connections between Static Stateful Distribution Transition Maps and LDS.} The LDS modeling via control theory in our work is closely aligned with the static stateful distribution maps
within the state performative prediction framework \cite{brown2022performative}. Let us first consider a static
stateful distribution transition map $f\left(\cdot\right)$ with $\mathcal{D}_{t+1}=f\left(\mathcal{D}_{t},\theta_{t}\right),\forall t\geq0.$
To facilitate the analysis, we use a performative data sample $Z_{t}\sim\mathcal{D}_{t}$
to equivalently characterize the performative distribution $\mathcal{D}_{t}.$
If the properties of $f\left(\cdot\right)$ are nice enough, then
from the perspective of random variables, there will exist a corresponding
mapping $\widetilde{f}\left(\cdot\right)$ such that 
\begin{align}
	& Z_{t+1}\overset{\mathrm{d}}{=}\widetilde{f}\left(Z_{t},\theta_{t}\right),Z_{t+1}\sim\mathcal{D}_{t+1},Z_{t}\sim\mathcal{D}_{t},\forall t\geq0,
\end{align}
where $\overset{\mathrm{d}}{=}$ denotes equal in distribution. Linearizing
$\widetilde{f}\left(\cdot\right)$ at some equilibrium point $\left(\overline{Z},\overline{\theta}\right)$
with $\overline{Z}\sim\overline{\mathcal{D}}$ and ignoring the higher
order terms will lead to
\begin{align}
	Z_{t+1} & \overset{\mathrm{d}}{=}\mathbf{A}Z_{t}+\mathbf{B}\theta_{t}+\mathbf{w},\label{eq:linearize}
\end{align}
where $\mathbf{A}=\left[\frac{\partial\widetilde{f}}{\partial d}\right]_{\left(\overline{d},\overline{\theta}\right)}$
and $\mathbf{B}=\left[\frac{\partial\widetilde{f}}{\partial\theta}\right]_{\left(\overline{d},\overline{\theta}\right)}$
are the static Jacobin matrices at the equilibrium point $\left(\overline{d},\overline{\theta}\right)$
and $\mathbf{w}=\widetilde{f}\left(\overline{Z},\overline{\theta}\right)-\mathbf{A}\overline{Z}-\mathbf{B}\overline{\theta}.$ Consider the performative data sample $Z_{t}$, the deployed
policy $\theta_{t}$ and the residual $\mathbf{w}$ in (\ref{eq:linearize})
as the state variable $\mathbf{x}_{t}$, the control input action
$\mathbf{u}_{t}$ and the additive system noise $\mathbf{w}_{t}\sim\mathbf{w}$
of an LDS, respectively. The linearized static stateful distribution
map model in (\ref{eq:linearize}) is thus precisely connected to
a LDS with system dynamics given by
\begin{align}
	& \mathbf{x}_{t+1}=\mathbf{A}\mathbf{x}_{t}+\mathbf{B}\mathbf{u}_{t}+\mathbf{w}_{t}.\label{eq:LDS}
\end{align}

\textbf{Extension to Performative Stateful Distribution Transition
	Maps via Performative LDS.} Further consider a more complicated
case of \textit{performative distribution transition maps} $f_{\theta_{t}}\left(\cdot\right)$
with $\mathcal{D}_{t+1}=f_{\theta_{t}}\left(\mathcal{D}_{t},\theta_{t}\right),\forall t\geq0,$
where the specific form of the distribution transition map $f_{\theta_{t}}$
is time-varying and depends on the delployed policy $\theta_{t}$.
Employing a similar linearization for $\widetilde{f}_{\theta_{t}}\left(\cdot\right)$
and neglecting the higher order terms, we obtain
\begin{align}
	Z_{t+1} & \overset{\mathrm{d}}{=}\mathbf{A}_{\theta_{t}}Z_{t}+\mathbf{B}_{\theta_{t}}\theta_{t}+\mathbf{w}_{\theta_{t}},\label{eq:linearize-1}
\end{align}
where $\mathbf{A}_{\theta_{t}}=\left[\frac{\partial\widetilde{f}_{\theta_{t}}}{\partial d}\right]_{\left(\overline{d},\overline{\theta}\right)}$
and $\mathbf{B}_{\theta_{t}}=\left[\frac{\partial\widetilde{f}_{\theta_{t}}}{\partial\theta}\right]_{\left(\overline{d},\overline{\theta}\right)}$
are the performative Jacobin matrices,
and $\mathbf{w}_t=\widetilde{f}_{\theta_{t}}\left(\overline{Z},\overline{\theta}\right)-\mathbf{A}_{\theta_{t}}\overline{Z}-\mathbf{B}_{\theta_{t}}\overline{\theta}.$
This results in an equivalent performative LDS given by
\begin{align}
	& \mathbf{x}_{t+1}=\mathbf{A}_{\mathbf{u}_{t}}\mathbf{x}_{t}+\mathbf{B}_{\mathbf{u}_{t}}\mathbf{u}_{t}+\mathbf{w}_{\mathbf{u}_{t}}.\label{eq:performative LDS}
\end{align}
In contrast to (\ref{eq:LDS}), where the system dyanmics are static,
the state transition matrix $\mathbf{A}_{\mathbf{u}_{t}}$, control
input gain matrix $\mathbf{B}_{\mathbf{u}_{t}}$ and additive noise
$\mathbf{w}_{\mathbf{u}_{t}}$ in (\ref{eq:performative LDS}) are
all performative and depend on the deployed control policy $\mathbf{u}_{t}$. 

In conclusion, the LDS modeling presented in our paper allows for
the \textit{performative transition maps of performative distributions}, thereby
extending the technical results previously obtained for fixed performative
distribution transition maps in the stateful performative prediction work \cite{brown2022performative}. As a result, our
proposed performative LDS framework has the potential to enhance the
understanding of general performative prediction.

\section{Problem Setup}

We consider the control of a linear dynamic system with per stage cost $c_{t}\left(\mathbf{x}_{t},\mathbf{u}_{t}\right)$.
A control policy $\pi$ is a mapping $\pi:\mathbb{R}^{d_{x}\times1}\rightarrow\mathbb{R}^{d_{u}\times1},$
which maps the system state $\mathbf{x}_{t}$ to the control action
$\mathbf{u}_{t}$, i.e., $\mathbf{u}_{t}=\pi\left(\mathbf{x}_{t}\right)$.
For each control policy $\pi$, we attribute a finite time horizon expected cost defined as
\begin{align}
	C_{T}^{\pi} & =\mathbb{E}_{\mathbf{x}_{0},\left\{ \mathbf{A}_{t}\right\} ,\left\{ \mathbf{w}_{t}\right\} }\left[\sum_{t=0}^{T}c_{t}\left(\mathbf{x}_{t},\mathbf{u}_{t}\right)\right],\label{eq:min formulation}
\end{align}
where $\mathbf{x}_{t+1}=\mathbf{A}_{t}\mathbf{x}_{t}+\mathbf{B}\mathbf{u}_{t}+\mathbf{w}_{t},$
the initial system state $\mathbf{x}_{0}$ follows a general distribution
of $\mathcal{D}_{x_{0}}$ with bounded support $\|\mathbf{x}_0\|\leq x_0$, and $\mathbb{E}_{\mathbf{x}_{0},\left\{ \mathbf{A}_{t}\right\} ,\left\{ \mathbf{w}_{t}\right\} }$
represents the expectation over $\mathbf{x}_{0},$ the entire policy-dependent state transition matrix sequence $\left\{ \mathbf{A}_{t},1\leq t\leq T\right\} $
and the entire disturbance sequence $\left\{ \mathbf{w}_{t},1\leq t< T\right\} $. 

{To the best of our knowledge, this paper is the very first work to
investigate the performative LDS, and we intend to provide a thorough
theoretical investigation and aim at establishing various new theoretical
results. Therefore, we opt to construct our performative LDS theoretical framework
upon the performative state transition matrix $\mathbf{A}_{t}$,
while maintaining the control input gain matrix $\mathbf{B}$ and
additive disturbance $\mathbf{w}_{t}$ as non-performative. This
facilitates us to streamline the theoretical analysis and consolidates
the system design insights.} 

We have the following assumption on the additive disturbances $\left\{ \mathbf{w}_{t},1\leq t\leq T\right\} .$ 
\begin{assump}	\label{Assumption noises wt}
	The additive disturbance per time step $\mathbf{w}_t$ is bounded, i.i.d, and zero-mean with a lower bounded covariance i.e., $\mathbf{w}_t \sim \mathcal{D}_{\mathbf{w}}, \mathbb{E}[\mathbf{w}_t]=\mathbf{0}, \mathbb{E}[\mathbf{w}_t \mathbf{w}_t^{\top}] \succeq \sigma^2 \mathbf{I}$, $\|\mathbf{w}_t\| \leq W, \forall 0 \leq t <T$, and $\mathbb{E}[\mathbf{w}_{t_1} \mathbf{w}_{t_2}^{\top}]=\mathbf{0}, \forall 0 \leq t_1 \neq t_2 < T$.
\end{assump}

\textbf{Disturbance-Action Control Policy.} We work with the following class of disturbance-action control policy throughout this paper, which is commonly used in nonstochastic control \cite{hambly2021policy,agarwal2019online,hazan2020nonstochastic,ghai2024online} to address general convex control cost functions.

\begin{defi}
	(Disturbance-Action Policy). For a disturbance-action control policy,
	the mapping $\pi$, $\forall0\leq t<T$, is uniquely characterized
	by a set of matrices $\left\{ \mathbf{M}^{(1)},\cdots,\mathbf{M}^{(H)}\right\} $.
	At every time step $t$, such a disturbance-action control policy
	assigns a control action $\mathbf{u}_{t}^{\left(\mathbf{M}\right)}$
	in the form of
	\begin{align}
		& \mathbf{u}_{t}^{\left(\mathbf{M}\right)}=\pi\left(\mathbf{x}_{t}\right)=-\mathbf{K}\mathbf{x}_{t}+\sum_{i=1}^{H}\mathbf{M}^{(i)}\mathbf{w}_{t-i}=-\mathbf{K}\mathbf{x}_{t}+\mathbf{M}\left[\mathbf{w}\right]_{t-1}^{H},\label{eq:DAP u_t^M}
	\end{align}
	where $\mathbf{M}=\left[\mathbf{M}^{(1)}, \mathbf{M}^{(2)}, \cdots, \mathbf{M}^{(H)}\right]$ belongs to a convex set $\mathbb{M}$ with bounded support $\|\mathbf{M}\|_F \leq M$, $[\mathbf{w}]_{t-1}^H$ is short for $\left[\mathbf{w}_{t-1}^{\top}, \cdots \mathbf{w}_{t-1-H}^{\top}\right]^{\top}$, $H < T$ is a constant and $\mathbf{w}_i=\mathbf{0}$ for all $i<0$.
\end{defi}

In the disturbance-action policy (\ref{eq:DAP u_t^M}), we adopt a
class of linear controller $\mathbf{K}$ defined as follows.
\begin{defi}
	(Strongly Stabilizing Linear Controller \cite{agarwal2019online,hambly2021policy,hazan2020nonstochastic,ghai2024online}). Given $\mathbf{A}$ and $\mathbf{B}$, a linear controller $\mathbf{K}$ is $(\kappa, \gamma)$ almost surely strongly stable for real numbers $\kappa \geq 1, \gamma<1$, if $\|\mathbf{K}\| \leq \kappa$, and there exists matrices $\mathbf{Q}$ and $\mathbf{L}$ such that $\widetilde{\mathbf{A}} := \mathbf{A}-\mathbf{B K}:=\mathbf{Q} \mathbf{L} \mathbf{Q}^{-1}$, with $\|\mathbf{L}\| \leq 1-\gamma$ and $\|\mathbf{Q}\|,\left\|\mathbf{Q}^{-1}\right\| \leq \kappa$.
\end{defi}


\medskip{}

It is worth noting that the disturbance-action policy is only parameterized
by the matrix $\mathbf{M}.$ Whereas the state feedback gain $\mathbf{K}$,
which is a fixed matrix, is not part of the parameterization of the
policy. As pointed out in \cite{agarwal2019logarithmic}, a typical choice of the parameter $H$ is $H = \gamma^{-1} \log(T\kappa^2)$. With an appropriate choice of the policy $\mathbf{M}$, the control action $\mathbf{u}_{t}^{(\mathbf{M})}$
in (\ref{eq:DAP u_t^M}) is capable of approximating any linear state
feedback control policy in terms of the total cost suffered with a finite time horizon of $H$ \cite{hambly2021policy,agarwal2019online,hazan2020nonstochastic,ghai2024online}.

\textbf{Policy-dependent Dynamics.} Without loss of generality, at
any time step $t$, the impact of the disturbance-action control policy
$\mathbf{u}_{t}^{\left(\mathbf{M}\right)}$ to the dynamics of the
linear system is modeled as a policy-dependent additive perturbation
$\mathbf{\Delta}_{t}$ to a common state transition matrix $\mathbf{A}.$ 
\begin{assump}\label{A2}
	(Policy-dependent State Transition Matrix). The disturbance-action
	policy-dependent state transition matrix $\mathbf{A}_{t}$ takes
	the form of 
	\begin{align}\label{trans_matrix}
		& \mathbf{A}_{t}=\mathbf{A}+\mathbf{\Delta}_{t},\mathbf{\Delta}_{t}\sim\mathcal{D}_{t}\left(\mathbf{M}\right),\forall0\leq t< T,
	\end{align}
	where $\mathbf{A}$ is the mean value of $\mathbf{A}_{t}$,
	and $\mathbf{\Delta}_{t}$ is the policy-dependent state transition
	perturbation with zero mean and bounded support, i.e., $\text{\ensuremath{\mathbb{E}\left[\mathbf{\Delta}_{t}\right]}=\ensuremath{\mathbf{0}}}$
	and there exists a bounded constant $\xi_{t}$ such that $\left\Vert \mathbf{\Delta}_{t}\right\Vert \leq\xi_{t}$,
	$\forall0\leq t<T$. For different time steps $t_{1}$ and $t_{2}$,
	$\mathbf{\Delta}_{t_{1}}$ and $\mathbf{\Delta}_{t_{2}}$ are mutually
	independent. Besides, $\left\{ \mathbf{\Delta}_{t},0\leq t<T\right\} $
	and $\left\{ \mathbf{w}_{t},0\leq t<T\right\} $ are mutually
	independent.
\end{assump}

{\begin{rem}
	 (Non-zero Mean $\mathbf{\Delta}_t$). If the mean of the disturbance is non-zero and time-varying, i.e.,  $\mathbb{E}\left[\boldsymbol{\Delta}_t\right]=\Theta_t$, we will have a new equivalent  $\mathbf{A}_t^{\prime}=\mathbf{A}+\Theta_t$ with zero mean disturbances. We only need to choose a new linear controller $\mathbf{K}_t^{\prime}$ such that $\mathbf{A}_t^{\prime}-\mathbf{B} \mathbf{K}_t^{\prime}$ is $\left(\kappa_t, \gamma_t\right)$-strongly stabilizing. As a result, without loss of generality, we assume $\mathbf{\Delta}_t$ is zero mean throughout this paper.
\end{rem}}

{\begin{rem}
	(Policy-dependent Control Input Gain Matrix $\mathbf{B}$). Note that under the DAP \eqref{eq:DAP u_t^M},  the disturbance-action is given by $\mathbf{B M}[\mathbf{w}]_{t-1}^H$, which is linear in policy $\mathbf{M}$. Such kind of linear disturbance-action control policy has various nice theoretical performance guarantees as substantiated in \cite{hambly2021policy,agarwal2019online,hazan2020nonstochastic,ghai2024online}. On the other hand, if $\mathbf{B}$ is also performative, we will have a generalized disturbance-action $\mathbf{B}_t(\mathbf{M}) \mathbf{M}[\mathbf{w}]_{t-1}^H$, which can be possibly nonlinear in policy $\mathbf{M}$. Such kind of generalized nonlinear disturbance-action control policy has received very few research attention, and it can serve as a very interesting future research direction.
\end{rem}}


We make the following sensitivity
assumption on the distributions $\left\{ \mathcal{D}_{t}\left(\mathbf{M}\right),0\leq t<T\right\} .$
\begin{assump}($\varepsilon$-Sensitivity). \label{epsilon-Sensitivity}For
	any $t=0,1,\cdots,T-1$, there exists a constant $\varepsilon_{t}>0$
	such that
	\begin{align}
		& \mathcal{W}^{1}\left(\mathcal{D}_{t}\left(\mathbf{M}\right),\mathcal{D}_{t}\left(\mathbf{M}^{\prime}\right)\right)\leq\varepsilon_{t}\left\Vert \mathbf{M}-\mathbf{M}^{\prime}\right\Vert _{F},\ \forall\mathbf{M},\mathbf{M}^{\prime}\in\mathcal{\mathbb{M}},
	\end{align}
	where $\mathcal{W}^{1}\left(\mathcal{D},\mathcal{D}^{\prime}\right)$
	denotes the Wasserstein-1 distance between the distributions $\mathcal{D}$
	and $\mathcal{D}^{\prime}$.
\end{assump}

Assumption A\ref{epsilon-Sensitivity} imposes a regularity requirement
on the distributions $\left\{ \mathcal{D}_{t}\left(\mathbf{M}\right),0\leq t<T\right\} .$
Intuitively, if the disturbance-action control policies are made according
to similar policy parameterizations $\mathbf{M}$, then the resulting
distributions of the policy-dependent state transition perturbations
should also be similar.

\textbf{System State Evolutions.} Under the disturbance-action policy
(\ref{eq:DAP u_t^M}), the following lemma shows that the system state
$\mathbf{x}_{t},\forall1\leq t\leq T,$ can be uniquely determined
by $\mathbf{x}_{0}$, $\widetilde{\mathbf{A}}$, $\mathbf{M}$, $\left\{ \mathbf{\Delta}_{t},0\leq t<T\right\} $
and $\left\{ \mathbf{w}_{t},-H\leq t< T\right\} $ .
\begin{lem}
	\label{Lemma: x_t representation M}Given a disturbance-action policy
	$\mathbf{M}$, the system state $\mathbf{x}_{t}$, $\forall1\leq t\leq T$,
	can be represented as 
	\begin{align}
		\mathbf{x}_{t}=\mathbf{x}_{t}^{\left(\mathbf{M}\right)}= & \prod_{i=0}^{t-1}\left(\widetilde{\mathbf{A}}+\mathbf{\Delta}_{i}\right)\mathbf{x}_{0}+\sum_{i=0}^{t-1}\prod_{j=i+1}^{t-1}\left(\mathbf{1}_{j<t}\left(\widetilde{\mathbf{A}}+\mathbf{\Delta}_{j}\right)+\mathbf{1}_{j=t}\mathbf{I}\right)\mathbf{B}\mathbf{M}\left[\mathbf{w}\right]_{i-1}^{H}\label{eq:x_t representation}\\
		& +\sum_{i=0}^{t-1}\prod_{j=i+1}^{t-1}\left(\mathbf{1}_{j<t}\left(\widetilde{\mathbf{A}}+\mathbf{\Delta}_{j}\right)+\mathbf{1}_{j=t}\mathbf{I}\right)\mathbf{w}_{i}.\nonumber 
	\end{align}
	Besides, let the $\mathbf{K}$ in DAP \eqref{eq:DAP u_t^M} be a strongly stabilizing linear controller, the norm of $\mathbf{x}_{t}$ is upper bounded as $\left\Vert \mathbf{x}_{t}\right\Vert \leq x_{0}\kappa^{2}\alpha_{t}+\kappa^{2}W\left(\left\Vert \mathbf{B}\right\Vert HM+1\right)\beta_{t},$
	where $\alpha_{t}=\prod_{i=0}^{t-1}\left(1-\gamma+\kappa^{2}\xi_{i}\right)$
	and $\beta_{t}=\sum_{i=0}^{t-1}\prod_{j=i+1}^{t-1}\left(\mathbf{1}_{j<t}\left(1-\gamma+\kappa^{2}\xi_{j}\right)+\mathbf{1}_{j=t}\right).$ 
\end{lem}


The \textit{performative optimal control (POC)} problem can therefore
be formulated as:
\begin{align}
	\min_{\mathbf{M}\in \mathbb{M}} \quad C_{T}^{\mathbf{M}}=\mathbb{E}_{\mathbf{x}_{0},\left\{\text{\ensuremath{\mathbf{\Delta}_{t}\sim\mathcal{D}_{t}\left(\mathbf{M}\right)}}\right\} ,\left\{ \mathbf{w}_{t}\right\} }\left[\sum_{t=0}^{T}c_{t}\left(\mathbf{x}_{t}^{\left(\mathbf{M}\right)},\mathbf{u}_{t}^{\left(\mathbf{M}\right)}\right)\right].\label{eq: PP optimal formulation}
\end{align}

The \textit{POC} problem (\ref{eq: PP optimal formulation}) comprises of a stochastic objective function with policy-dependent distributions. Due to non-convexity, the performative optimal solution $\mathbf{M}^{PO}$ to (\ref{eq: PP optimal formulation}) is usually difficult to obtain. Alternatively, in this paper, we are interested in the \textit{performative stable control (PSC) }solution:
	\begin{align}
		\mathbf{M}^{PS}=\Phi(\mathbf{M}^{PS}) :=\arg\min_{\mathbf{M}\in \mathbb{M}}\mathbb{E}_{\mathbf{x}_{0},\left\{ \mathbf{\Delta}_{t}\sim\mathcal{D}_{t}\left(\mathbf{M}^{PS}\right)\right\} ,\left\{ \mathbf{w}_{t}\right\} }\left[\sum_{t=0}^{T}c_{t}\left(\mathbf{x}_{t}^{\left(\mathbf{M}\right)},\mathbf{u}_{t}^{\left(\mathbf{M}\right)}\right)\right].\label{eq: performative stable OPT}
	\end{align}
 	Notice that $\mathbf{M}^{P S}$ is defined to be a fixed point of the map $\Phi$.
	Compared to the \textit{POC} (\ref{eq: PP optimal formulation}),
	the distribution of $\mathbf{\Delta}_{t}$ in \textit{PSC} (\ref{eq: performative stable OPT}) changes from $\mathbf{\Delta}_{t}\sim\mathcal{D}_{t}\left(\mathbf{M}\right)$
	to $\mathbf{\Delta}_{t}\sim\mathcal{D}_{t}\left(\mathbf{M}^{PS}\right),\forall0\leq t<T.$ The existence and uniqueness of $\mathbf{M}^{PS}$ will be dicsussed in Lemma \ref{Lemma: Existence Uniqueness M^PS}.
	
	\textbf{Comparison to Existing Works.} Most of the existing performative
prediction works \cite{perdomo2020performative,li2022state,li2209multi,drusvyatskiy2023stochastic,brown2022performative,mendler2020stochastic,wood2021online,ray2022decision,izzo2021learn,miller2021outside,ray2022decision} consider the cost in the form of $l\left(\mathbf{\theta};Z\right),$
where $\theta$ is the decision variable. Then the relationship between data samples $Z$ and decision $\theta$ is parameterized
by a fixed distribution $Z\sim\mathcal{D}\left(\theta\right)$ with
fixed sensitivity $\varepsilon$. However, in our work, the relationship
between system state data samples $\mathbf{x}_{t}$ and policy $\mathbf{M}$
is characterized by a time-varying distribution $\mathbf{x}_{t}\sim f_{t}\left(\mathcal{D}_{0}\left(\mathbf{M}\right),\cdots,\mathcal{D}_{t-1}\left(\mathbf{M}\right)\right)$
with a time-varying sequence of joint sensitivities $\left\{ \varepsilon_{0},\cdots,\varepsilon_{t}\right\} ,$
$\forall0\leq t< T$. The total cost $\sum_{t=0}^{T}c_{t}$ depends on all the policy-dependent distributions $\left\{ \mathcal{D}_{0}\left(\mathbf{M}\right),\cdots,\mathcal{D}_{T-1}\left(\mathbf{M}\right)\right\} $
with a collection of sensitivities $\left\{ \varepsilon_{0},\cdots,\varepsilon_{T-1}\right\} $.
These key differences lead to a more complicated analysis in our work.

\begin{algorithm}[t]
	\caption{Repeated Stochastic Gradient Descent (RSGD)}
	\label{algo:algo_stochastic}
	{\bf Input:}  Step sizes $\left\{ \eta_{n},0\leq n\leq N\right\} $,
	parameters $\mathbf{K},H$. Define $\mathcal{\mathbb{M}}=\left\{ \mathbf{M}:\left\Vert \mathbf{M}\right\Vert \leq M\right\} .$
	{Initialize $\mathbf{M}_{0}\in\mathcal{\mathbb{M}}$ arbitrarily.}
	\begin{algorithmic}[1]
		\FOR{$n=0,\cdots,N,$}
		\STATE {Initialize $\nabla J_{T}=\mathbf{0}$.}
		\FOR{$t=0,\cdots,T-1,$ }
		\STATE {Use control $\mathbf{u}_{t}=-\mathbf{K}\mathbf{x}_{t}+\mathbf{M}_{n}\left[\mathbf{w}\right]_{t-1}^{H}.$}
		\STATE {Observe $\mathbf{A}_{t},$ $\mathbf{x}_{t+1}$; compute noise
			$\mathbf{w}_{t}=\mathbf{x}_{t+1}-\mathbf{A}_{t}\mathbf{x}_{t}-\mathbf{B}\mathbf{u}_{t}.$}
		\STATE {Compute the gradient $\nabla_{\mathbf{M}_{n}}c_{t}\left(\mathbf{x}_{t},\mathbf{u}_{t}\right)$ and update $\nabla J_{T}\leftarrow\nabla J_{T}+\nabla_{\mathbf{M}_{n}}c_{t}\left(\mathbf{x}_{t},\mathbf{u}_{t}\right).$}
		\ENDFOR
		\STATE {Update $\mathbf{M}_{n+1}\leftarrow\mathrm{Proj}_{\mathbb{M}}\{\mathbf{M}_{n}-\eta_{n}\nabla J_{T}\}$.}
		\ENDFOR
	\end{algorithmic}
\end{algorithm}
\textbf{RSGD Scheme.} We propose a repeated stochastic gradient descent
(RSGD) scheme in Algorithm \ref{algo:algo_stochastic} to find a PSC solution to (\ref{eq: performative stable OPT}). {The metric of the projection in line 9 of Algorithm 1 is the matrix Frobenius norm. Specifically, $\operatorname{Proj}_{\mathbb{M}}\left\{\mathbf{M}\right\}=\arg \min _{\mathbf{M}' \in \mathbb{M}}\left\|\mathbf{M}-\mathbf{M}'\right\|_{F}^2$. Note that such a projection is computationally tractable because the Frobenius norm square minimization is a convex optimization problem.} The RSGD scheme first computes the stochastic gradient of the total
cost w.r.t. policy $\sum_{t=0}^{T}\nabla_{\mathbf{M}_{n}}c_{t}\left(\mathbf{x}_{t},\mathbf{u}_{t}\right)$
and then perform stochastic gradient descent on $\mathbf{M}$. The
detailed steps for computation of the stochastic gradient $\nabla_{\mathbf{M}_{n}}c_{t}\left(\mathbf{x}_{t},\mathbf{u}_{t}\right)$ are provided in Appendix B.

\section{Main Results}
This section investigates the existence of a PSC solution $\mathbf{M}^{P S}$ to \eqref{eq: performative stable OPT} and the convergence of the RSGD scheme to $\mathbf{M}^{P S}$. We require the following assumptions on the per stage cost $c_t$ \cite{agarwal2019online,hambly2021policy,agarwal2019online,hazan2020nonstochastic,ghai2024online}.
\begin{assump}\label{A4}
	(Strongly Convex). The per stage cost function
	$c_{t}\left(\mathbf{x},\mathbf{u}\right)$ is $\mu$-strongly convex
	such that 
	\begin{align*}
		c_{t}\left(\mathbf{x}_{1},\mathbf{u}_{1}\right)&\geq c_{t}\left(\mathbf{x}_{2},\mathbf{u}_{2}\right)+\nabla_{\mathbf{x}}^{\top}c_{t}\left(\mathbf{x}_{2},\mathbf{u}_{2}\right)\left(\mathbf{x}_{1}-\mathbf{x}_{2}\right)+\nabla_{\mathbf{u}}^{\top}c_{t}\left(\mathbf{x}_{2},\mathbf{u}_{2}\right)\left(\mathbf{u}_{1}-\mathbf{u}_{2}\right)\\
		&+\frac{\text{\ensuremath{\mu}}}{2}\left(\left\Vert \mathbf{x}_{1}-\mathbf{x}_{2}\right\Vert ^{2}+\left\Vert \mathbf{u}_{1}-\mathbf{u}_{2}\right\Vert ^{2}\right),\forall\mathbf{x}_{1},\mathbf{x}_{2}\in\mathbb{R}^{d_{x}},\mathbf{u}_{1},\mathbf{u}_{2}\in\mathbb{R}^{d_{u}}.
	\end{align*}
\end{assump}

\begin{assump}\label{A5}
	(Smoothness). The per stage cost function $c_{t}\left(\mathbf{x},\mathbf{u}\right)$
	is $\varsigma$ smooth such that 
	\begin{align*}
		&\left\Vert \nabla_{\mathbf{x}}c_{t}\left(\mathbf{x}_{1},\mathbf{u}_{1}\right)-\nabla_{\mathbf{x}}c_{t}\left(\mathbf{x}_{2},\mathbf{u}_{1}\right)\right\Vert +\left\Vert \nabla_{\mathbf{u}}c_{t}\left(\mathbf{x}_{1},\mathbf{u}_{1}\right)-\nabla_{\mathbf{u}}c_{t}\left(\mathbf{x}_{1},\mathbf{u}_{2}\right)\right\Vert\\
		&\leq\varsigma\left(\left\Vert \mathbf{x}_{1}-\mathbf{x}_{2}\right\Vert +\left\Vert \mathbf{u}_{1}-\mathbf{u}_{2}\right\Vert \right),\forall\mathbf{x}_{1},\mathbf{x}_{2}\in\mathbb{R}^{d_{x}},\mathbf{u}_{1},\mathbf{u}_{2}\in\mathbb{R}^{d_{u}}.
	\end{align*}
\end{assump}

\begin{assump}\label{A6}
	(Boundedness). There exists a positive constant
	$G$ such that $$\left\Vert \nabla_{\mathbf{x}}c_{t}\left(\mathbf{x},\mathbf{u}\right)\right\Vert ,\left\Vert \nabla_{\mathbf{u}}c_{t}\left(\mathbf{x},\mathbf{u}\right)\right\Vert \leq GD,\forall\left\Vert \mathbf{x}\right\Vert ,\left\Vert \mathbf{u}\right\Vert \leq D.$$
\end{assump}

\medskip{}

The above Assumptions {A\ref{A4}-A\ref{A6}} on the per stage cost
function $c_{t}\left(\mathbf{x},\mathbf{u}\right)$ are quite standard,
which hold for broad classes of costs such as the quadratic costs. 

To facilitate our discussions, we define an expected total cost   with distribution shift as
	\begin{align*}
		& C_{T}\left(\mathbf{M};\mathbf{M}^{\prime}\right)\coloneqq\mathbb{E}_{\mathbf{x}_{0},\left\{ \mathbf{\Delta}_{t}\sim\mathcal{D}_{t}\left(\mathbf{M}^{\prime}\right)\right\} ,\left\{ \mathbf{w}_{t}\right\} }\left[\sum_{t=0}^{T}c_{t}\left(\mathbf{x}_{t}^{\left(\mathbf{M}\right)},\mathbf{u}_{t}^{\left(\mathbf{M}\right)}\right)\right],
	\end{align*}
	where, under policy $\mathbf{M}$, the distribution of policy-dependent
	perturbation is changed from $\mathbf{\Delta}_{t}\sim\mathcal{D}_{t}\left(\mathbf{M}\right)$
	to $\mathbf{\Delta}_{t}\sim\mathcal{D}_{t}\left(\mathbf{M}^{\prime}\right)$,
	$\forall0\leq t<T.$ For the rest of this paper, unless otherwisespecified, $\nabla C_{T}\left(\mathbf{M};\mathbf{M}^{\prime}\right)$ denote the gradients taken w.r.t. the first argument $\mathbf{M}$.

Our main results rely on the strong convexity of the {expected total
cost $C_{T}\left(\mathbf{M};\mathbf{M}^{\prime}\right)$ with respect
to its first argument $\mathbf{M}$. However, the strong convexity of the
per stage cost function $c_{t}\left(\mathbf{x},\mathbf{u}\right)$
over the state-action space in Assumption A\ref{A4} does not by itself imply
the strong convexity of the  {expected} total cost $C_{T}\left(\mathbf{M};\mathbf{M}^{\prime}\right)$
over the space of policies $\mathbf{M}$. This is becasue the policy
$\mathbf{M}$, which maps from a space of dimensionality $H\times d_{x}\times d_{u}$
to that of $d_{x}+d_{u}$, is not necessarily full column-rank. Our
next lemma, which forms the core of our analysis, shows that this
is not the case using the inherent stochastic nature of the policy-dependent
dynamics.
\medskip{}
\begin{lem}\label{s_cvx_t}
	Under A\ref{Assumption noises wt}-A\ref{A6}, fix any $\mathbf{M}^{\prime}\in\mathcal{\mathbb{M}},$
	the  {expected} total cost $C_{T}\left(\mathbf{M};\mathbf{M}^{\prime}\right)$
	is $\widetilde{\mu}$-strongly convex in its first argument $\mathbf{M}$
	such that $\forall \mathbf{M}_1,\mathbf{M}_2 \in \mathbb{M}$,
	\begin{align}\label{scvx_eq}
	 C_{T}\left(\mathbf{M}_{1};\mathbf{M}^{\prime}\right)\geq C_{T}\left(\mathbf{M}_{2};\mathbf{M}^{\prime}\right)+\mathrm{Tr}\left(\left(\nabla C_{T}\left(\mathbf{M}_{2};\mathbf{M}^{\prime}\right)\right)^{\top}\left(\mathbf{M}_{1}-\mathbf{M}_{2}\right)\right)+\frac{\widetilde{\mu}}{2}\left\Vert \mathbf{M}_{1}-\mathbf{M}_{2}\right\Vert _{F}^{2},\nonumber\\
	\end{align}
	where $\widetilde{\mu}=\min\left\{ \frac{\left(T-H+1\right)\mu\sigma^2}{2},\frac{\left(T-H+1\right)\mu\sigma^{2}\gamma^{2}}{64\kappa^{10}}\right\} .$
\end{lem}
\medskip{}
For a concise presentation of the smoothness of
 {expected} total control cost, we next define a collection of constants
as follows:
\begin{align*}
	& c_{1}\coloneqq d_{x}\varsigma H^{\frac{3}{2}}W\left(1+\left(\kappa^{2}+\kappa^{3}\right)\left\Vert \mathbf{B}\right\Vert \right)\left(\kappa^{2}+\kappa^{3}\right)\left(1-\gamma\right)^{-1},\\
	&c_{2}\coloneqq d_{x}H^{\frac{3}{2}}WG\left(1-\gamma\right)^{-1}\left(\kappa^{4}+\kappa^{5}\right)\left\Vert \mathbf{B}\right\Vert , c_{3}\coloneqq\left(HM\left\Vert \mathbf{B}\right\Vert +1\right)W,c_{4}\coloneqq HW\left(1-\gamma\right)c_{1},\\
	&c_{5}\coloneqq HW\left(\kappa^{2}+\kappa^{3}\right)^{-1}\left(1-\gamma\right)c_{1}.
\end{align*}
The smoothness of the  {expected} total cost $C_{T}\left(\mathbf{M};\mathbf{M}^{\prime}\right)$
is summarized below. 
\begin{lem}\label{smooth}
	Under A\ref{Assumption noises wt}-A\ref{A6}, the  {expected} total cost $C_{T}$ is smooth in the sense
	that, for any $\mathbf{M},\mathbf{M}^{\prime},\mathbf{M}_{1},\mathbf{M}_{2}\in\mathcal{\mathbb{M}}$
	and $\forall1\leq t\leq T,$ the following inequality holds
	\begin{align}
		 &\left\Vert \nabla C_{T}\left(\mathbf{M}_{1};\mathbf{M}\right)-\nabla C_{T}\left(\mathbf{M}_{2};\mathbf{M}^{\prime}\right)\right\Vert _{F}\nonumber\\
		 &\leq\sum_{t=1}^{T}\lambda_{t}\left\Vert \mathbf{M}_{1}-\mathbf{M}_{2}\right\Vert _{F}
		+\sum_{t=0}^{T-1}\left(\varepsilon_{t}\sum_{i=t+1}^{T}\nu_{i}\right)\left\Vert \mathbf{M}-\mathbf{M}^{\prime}\right\Vert _{F},\label{eq:average grad diff}
	\end{align}
	where $\lambda_{t}=c_{1}\left(c_{4}\beta_{t}+c_{5}\right),\nu_{t}=\left(c_{1}+c_{2}\beta_{t}\right)\left(x_{0}\alpha_{t}+c_{3}\beta_{t}\right),\forall1\leq t\leq T,$
	and we recall that $\alpha_{t}=\prod_{i=0}^{t-1}\left(1-\gamma+\kappa^{2}\xi_{i}\right)$
	and $\beta_{t}=\sum_{i=0}^{t-1}\prod_{j=i+1}^{t-1}\left(\mathbf{1}_{j<t}\left(1-\gamma+\kappa^{2}\xi_{j}\right)+\mathbf{1}_{j=t}\right)$
	from Lemma \ref{Lemma: x_t representation M}, which characterize the growth of the norm of system
	state $\left\Vert \mathbf{x}_{t}\right\Vert .$ 
\end{lem}

\textbf{Existence and Uniqueness of $\mathbf{M}^{PS}$.} Our first
main result establishes a sufficient condition for the existence and
uniqueness of the performative stable policy $\mathbf{M}^{PS}$ that
solves the \textit{PSC} problem \eqref{eq: performative stable OPT}.
\begin{lem}
	\label{Lemma: Existence Uniqueness M^PS} Under A\ref{Assumption noises wt}-A\ref{A6}, consider the fixed-point iteration
	\begin{align}
		\mathbf{M}_{n+1}= & \Phi\left(\mathbf{M}_{n}\right),\forall n\geq0,\label{eq:RRM}
	\end{align}
	where the map $\Phi$ is defined in \eqref{eq: performative stable OPT}.
	If the following condition is satisfied
	\begin{align}
		& \sum_{t=0}^{T-1}\left(\varepsilon_{t}\sum_{i=t+1}^{T}\nu_{i}\right)<\widetilde{\mu},\label{eq:suff condition M^PS}
	\end{align}
	then iterates $\mathbf{M}_{n}$ converge to a unique performatively
	stable point $\mathbf{M}^{PS}$ at a linear rate, i.e., 
	\begin{align*}
		\left\|\mathbf{M}_n-\mathbf{M}^{P S}\right\|_F \leq \rho \text { for } n \geq\left(1-\frac{\sum_{t=0}^{T-1}\left(\varepsilon_t \sum_{i=t+1}^T \nu_i\right)}{\widetilde{\mu}}\right)^{-1} \log \left(\frac{1}{\rho}\left\|\mathbf{M}_0-\mathbf{M}^{P S}\right\|_F\right).
	\end{align*}
	
\end{lem}

\medskip{}


The sufficient condition \eqref{eq:suff condition M^PS} delivers a fact that that the existence and uniqueness of $\mathbf{M}^{P S}$ is jointly determined by all the sensitivities $\left\{\varepsilon_t, 0 \leq t<T\right\}$ in the temporal domain. The sensitivity $\varepsilon_t$ at the $t$-th time step is propagated starting from time step $t+1$ to the last time step $T$ with a sequence of weights $\left\{\nu_{t+1}, \cdots, \nu_T\right\}$. The aggregated impact of all the policy-dependent disturbances $\left\{ \mathbf{\Delta}_{t},0\leq t<T\right\} $ is captured by total sum in L.H.S. of \eqref{eq:suff condition M^PS}. This is very different from the existing performative prediction works \cite{perdomo2020performative,mendler2020stochastic,drusvyatskiy2023stochastic,brown2022performative,li2022state,wood2021online,izzo2021learn,miller2021outside,ray2022decision}, where only one distribution $\ca{D}$ and one sensitivity $\varepsilon$ are involved.


The sufficient condition (\ref{eq:suff condition M^PS}) also implies
that it is preferable for the initial policy-dependent disturbance
$\mathbf{\Delta}_{0}$ to be small because it propagates and aggregates for the longest time steps of $T$.

\medskip{}

\textbf{Impacts of System Stability. }The sufficient condition (\ref{eq:suff condition M^PS})
also reveals the impacts of system stability on the existence and
uniqueness of the performative stable solution $\mathbf{M}^{PS},$
which are summarized below.\medskip{}

{\begin{prop}
	(Almost Surely Strongly Stable Case) \label{Proposition 1: Almost Sure Stable} Let A\ref{Assumption noises wt}-A\ref{A6} hold. {If}
	policy-dependent state transition matrix $\mathbf{A}_{t}$ is $\left(\kappa, \gamma-\kappa^2 \xi_t\right)$-strongly stable for real numbers $\kappa \geq 1, \gamma- \kappa^2 \xi_t<1, \forall 0\leq t<T$, almost surely. Let $\zeta=\max \left\{1-\gamma+\kappa^2 \xi_t, \forall 0 \leq t<T\right\}$, then $\mathbf{M}^{PS}$ exists and is unique if $ \sum_{t=0}^{T-1}\varepsilon_{t}<\phi{\left(1-\frac{H}{T}\right)}\overline{\mu},$
where $\overline{\mu}=\min\left\{ \frac{\mu\sigma^2}{2},\frac{\mu\sigma^{2}\gamma^{2}}{64\kappa^{10}}\right\} $
and $\phi$ is some positive constant. 
\end{prop}}

Proposition \ref{Proposition 1: Almost Sure Stable} points out that
when the policy-dependent dynamics are {almost surely strongly stable}, we only
need to make sure that the sum of all the sensitivities $\left\{ \varepsilon_{t},0\leq t<T\right\} $
is below a certain threshold. 
\begin{prop}
	\label{Proposition 2: Almost Sure Unstable} {(Almost Surely Unstable Case)} Let A\ref{Assumption noises wt}-A\ref{A6} hold. {If} the policy-dependent
	state transition matrix $\mathbf{A}_{t}$ is almost surely unstable,
	i.e., there exists a positive constant $\widetilde{\zeta}>1$ such
	that $\widetilde{\zeta}\leq\left\Vert \mathbf{A}_{t}\right\Vert \leq1-\gamma+\kappa^{2}\xi_{t},\forall0\leq t<T.$
	In this case, {to guarantee the sufficient condition \eqref{eq:suff condition M^PS} can be satifised, we must have } $ \varepsilon_{t}<\frac{\overline{\phi}{\left(T-H+1\right)\overline{\mu}}}{\widetilde{\zeta}^{T-t}-1},\forall0\leq t<T,$
	where $\overline{\phi}$ is some positive constant.
\end{prop}
For unstable policy-dependent dynamics, condition \eqref{eq:suff condition M^PS} will impose a necessary requirement on the temporally backwards decaying of the sensitivities. Particularly, more restrictive requirements are placed on the the sensitivies in the early time steps, i.e., smallt, which should decay exponentially fast w.r.t. the control time horzion $T$.

For more general applications, where $\mathbf{A}_{t}$ can be
either stable or unstable for different time steps, the weighted sum requirement of the sensitivities {$\{\varepsilon_t, 0\leq t<T\}$} in \eqref{eq:suff condition M^PS} is sufficient to guarantee the existence and uniqueness of the performative stable solution $\mathbf{M}^{PS}$.

\textbf{Convergence of RSGD Scheme.} Our next theorem establishes the convergence rate of the proposed RSGD algorithm. 
\begin{center}
	\fbox{\begin{minipage}{.975\linewidth}
				\begin{thm}	\label{Thm: Convergence RSGD}
					{Choose two positive constants $\phi_{1}>0$ and $\phi_{2}\geq1$
						such that the following two conditions are satisfied simualtneously
						\begin{align}
							&\frac{\phi_{1}}{\phi_{2}}  \leq\mathrm{min}\left\{ \frac{\widetilde{\mu}-\sum_{t=0}^{T-1}\left(\varepsilon_{t}\sum_{i=t+1}^{T}\nu_{i}\right)}{2\left(\sum_{t=1}^{T}\lambda_{t}+\sum_{t=0}^{T-1}\left(\varepsilon_{t}\sum_{i=t+1}^{T}\nu_{i}\right)\right)^{2}},\frac{1}{\widetilde{\mu}-\sum_{t=0}^{T-1}\left(\varepsilon_{t}\sum_{i=t+1}^{T}\nu_{i}\right)}\right\},\label{17}\\
							&\frac{\phi_{1}}{1+\frac{1}{\phi_{2}}}  \geq\frac{2}{\widetilde{\mu}-\sum_{t=0}^{T-1}\left(\varepsilon_{t}\sum_{i=t+1}^{T}\nu_{i}\right)}\label{18}.
						\end{align}
					}{Consider a sequence of non-negative step sizes $\left\{ \eta_{n}=\frac{\phi_{1}}{n+\phi_{2}},n\geq0\right\} $
					. Then, the iterates generated by RSGD admit the following bound for
					any $N\geq1$:
					\begin{align}
						\mathbb{E}\left[\left\Vert \mathbf{M}_{N}-\mathbf{M}^{PS}\right\Vert _{F}^{2}\right]\leq\mathrm{e}^{-\sum_{n=1}^{N}\frac{\phi_{1}}{n}\left(\widetilde{\mu}-\sum_{t=0}^{T-1}\left(\varepsilon_{t}\sum_{i=t+1}^{T}\nu_{i}\right)\right)}\mathbb{E}\left[\left\Vert \mathbf{M}_{0}-\mathbf{M}^{PS}\right\Vert _{F}^{2}\right] & +\frac{\phi_{3}}{N},\label{eq: RSGD convergence}
					\end{align}
					where $\vartheta_{t}=\kappa^{3}G\left(\left(HW+\kappa^{2}\right)\kappa\left\Vert \mathbf{B}\right\Vert \beta_{t}+1\right)\left(x_{0}\alpha_{t}+c_{3}\beta_{t}\right)+GHWM\left(\kappa^{3}\beta_{t}+1\right),\forall1\leq t\leq T,$
					and $\phi_{3}=\frac{4\phi_{1}T\sum_{t=1}^{T}\vartheta_{t}^{2}}{\widetilde{\mu}-\sum_{t=0}^{T-1}\left(\varepsilon_{t}\sum_{i=t+1}^{T}\nu_{i}\right)}$
					is a positive constant.}
				\end{thm}
		\end{minipage}} 
\end{center}

\medskip{}


{Note that, under the sufficient condition \eqref{eq:suff condition M^PS} in Lemma \ref{Lemma: Existence Uniqueness M^PS}, there always exists a pair of $(\phi_1, \phi_2)$ that satisfies \eqref{17} and \eqref{18} simultaneously by letting $\phi_2$ be sufficiently large.} The first term on the R. H. S. of (\ref{eq: RSGD convergence}) decays {at the rate of  $\ca{O}(\mathrm{e}^{-\sum_{n=1}^{N}\frac{\phi_{1}}{n}\left(\widetilde{\mu}-\sum_{t=0}^{T-1}\left(\varepsilon_{t}\sum_{i=t+1}^{T}\nu_{i}\right)\right)})$}
and is scaled by the initial error $\mathbb{E}\left[\left\Vert \mathbf{M}_{0}-\mathbf{M}^{PS}\right\Vert _{F}^{2}\right].$
The second term is a fluctuation term that only depends on the variance
of the stochastic gradient, which
decays {at the rate  $\mathcal{O}\left(1/{N}\right)$. For more general types of step sizes and the associated nonasymptotic convergence rate analysis, please refer to Lemma \ref{Thm: Convergence RSGD appendix_lem} in the appendix \ref{G.1}.}


\begin{figure}[t] 
	\centering
	\includegraphics[width=.31\linewidth]{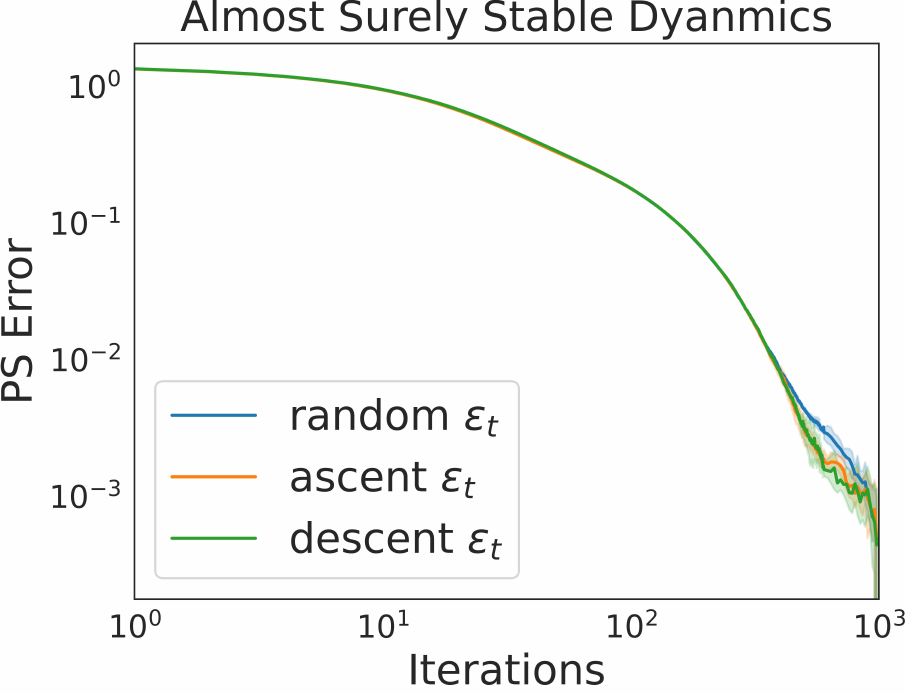}
	~~
	\includegraphics[width=.31\linewidth]{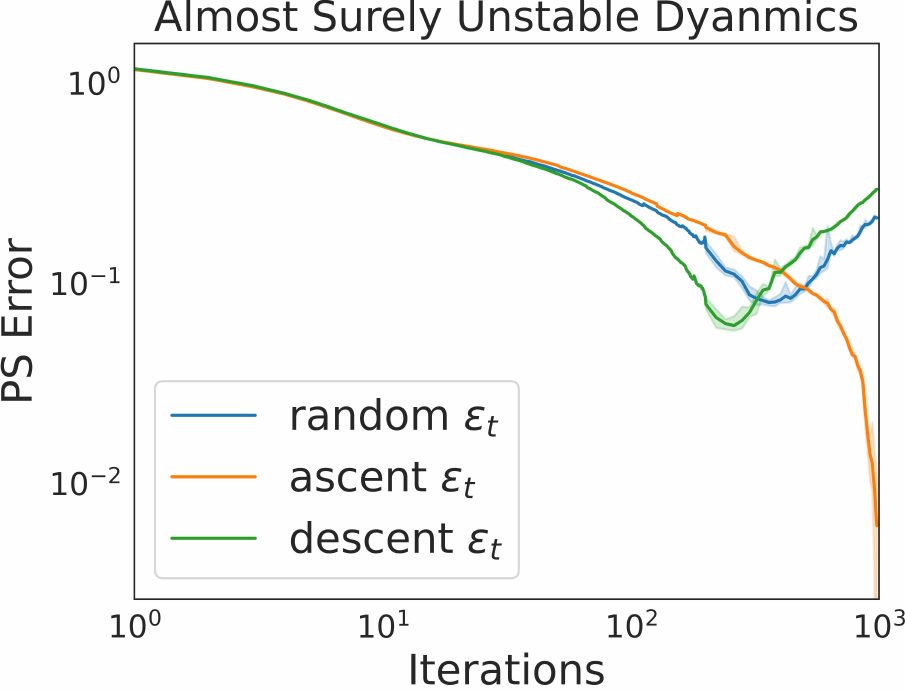}
	~~
	\includegraphics[width=.31\linewidth]{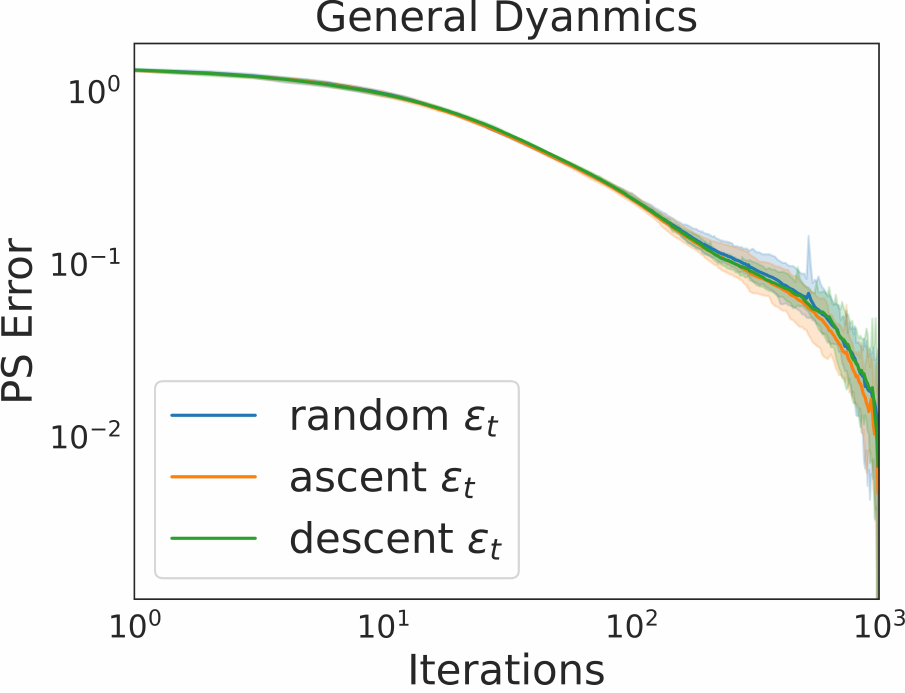}
	~~
	\caption{PS Error $\left\Vert \mathbf{M}_{N}-\mathbf{M}^{PS}\right\Vert _{F}^{2}$}
	\label{fig:gap}
\end{figure}

\begin{figure}[t] 
	\centering
	\includegraphics[width=.3\linewidth]{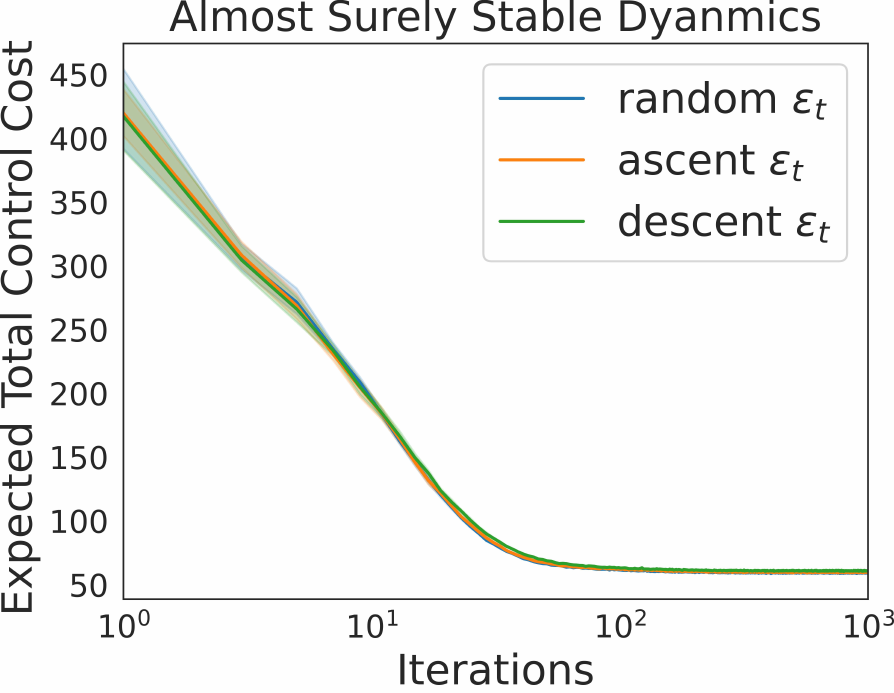}
	~~
	\includegraphics[width=.3\linewidth]{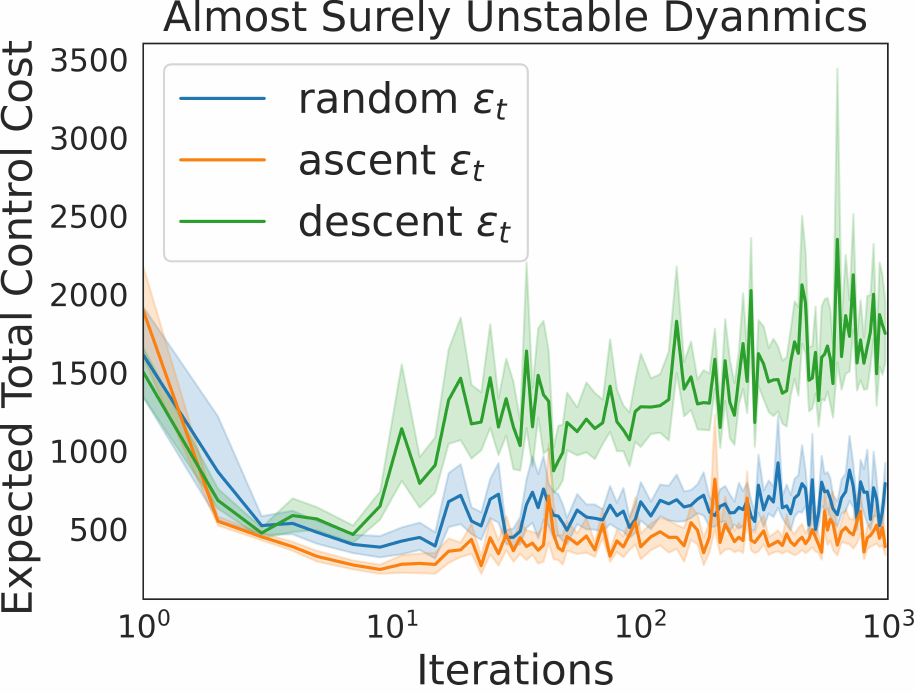}
	~~
	\includegraphics[width=.3\linewidth]{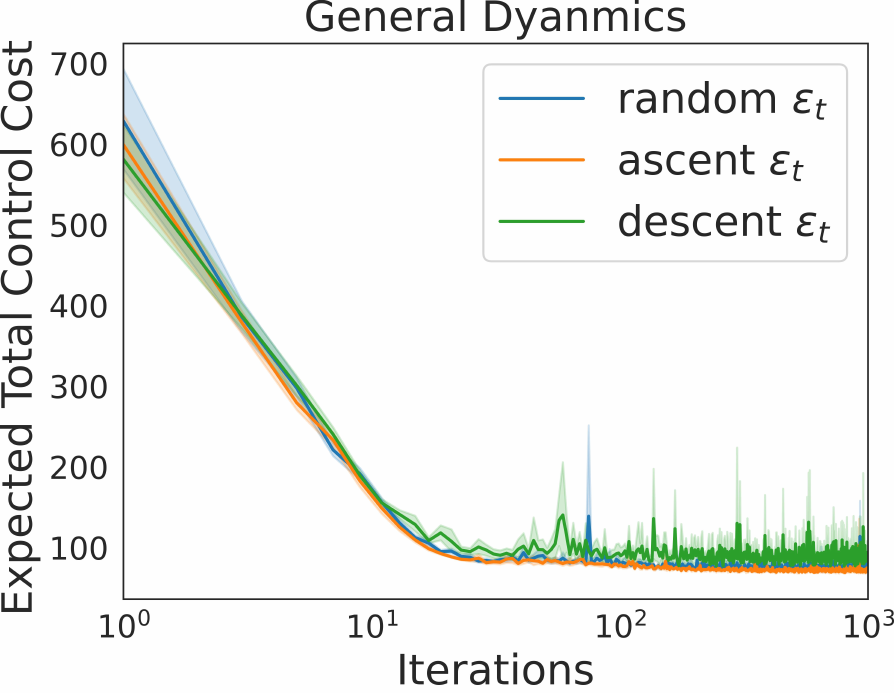}
	~~
	\caption{ {Expected Total Control Cost $C_{T}\left(\mathbf{M}_{N};\mathbf{M}_{N}\right)$}}
	\label{fig:cost}
\end{figure}

\section{Numerical Experiments}
We consider an application of stock investment risk minimization problem
to verify our algorithm and theoretical results. Consider an investor
trading a total number of 10 stocks over a period of $T=60$ trading
days. The detailed system setups are described in Example 1 in Appendix A of the supplementary. We compare the {performative error (PS error)} $\left\Vert \mathbf{M}_{N}-\mathbf{M}^{PS}\right\Vert _{F}^{2}$ and the  {expected} total cost $C_{T}\left(\mathbf{M}_{N};\mathbf{M}_{N}\right)$
against the iteration number $N$, respectively. {We consider a fixed distributional sensitivity value set $\boldsymbol{\varepsilon}=\left\{\varepsilon_{i},i=0,1,\ldots,T-1\right\}$.} We assign three different patterns of the sensitivity sequence as
 {$\varepsilon_{\mathrm{d}}=\mathrm{descend}\left(\boldsymbol{\varepsilon}\right)$,}
 {$\varepsilon_{\mathrm{a}}=\mathrm{ascend}\left(\boldsymbol{\varepsilon}\right)$}
and  {$\varepsilon_{\mathrm{r}}=\mathrm{random}\left(\boldsymbol{\varepsilon}\right)$}
in descending, ascending, and random order, respectively, over the time steps. We first observe from Figure \ref{fig:gap} (left) and Figure \ref{fig:cost} (left) that when the policy-dependent
system dynamics  {$\mathbf{A}_{t},\forall 0\leq t<T,$} are {almost surely strongly stable}, the gap $\left\Vert \mathbf{M}_{N}-\mathbf{M}^{PS}\right\Vert _{F}^{2}$
of three different patterns $\varepsilon_{\mathrm{d}},$ $\varepsilon_{\mathrm{a}}$
and $\varepsilon_{\mathrm{r}}$ all decay at $\mathcal{O}\left(\frac{1}{N}\right)$
as $N\rightarrow\infty$, and the  {expected} total control cost also
converges. These coincide exactly with Proposition \ref{Proposition 1: Almost Sure Stable}
and Theorem \ref{Thm: Convergence RSGD}, where the existence of $\mathbf{M}^{PS}$
only requires that the sum of the distributional sensitivities $\sum_{t=0}^{T-1}\varepsilon_{t}$
is small enough, and the temporal order of each $\varepsilon_{t}$ is negligible. We next observe from Figure \ref{fig:gap} (middle) and Figure \ref{fig:cost} (middle),  that when the
 {$\mathbf{A}_{t},\forall 0\leq t<T,$} are almost  {surely} unstable,
the iterates $\mathbf{M}_{n}$ of both $\varepsilon_{\mathrm{d}}$ and
$\varepsilon_{\mathrm{r}}$ diverge. The  {expected} total control costs
associated with $\varepsilon_{\mathrm{d}}$ and $\varepsilon_{\mathrm{r}}$
are significantly larger than that of $\varepsilon_{\mathrm{a}}.$
This is because large initial distributional sensitivities in $\varepsilon_{\mathrm{d}}$ and
$\varepsilon_{\mathrm{r}}$ rule out the existence of $\mathbf{M}^{PS},$
which matches Proposition \ref{Proposition 2: Almost Sure Unstable}.
For the general case of $\mathbf{A}_{t}$ in Figure \ref{fig:gap} (right) and Figure \ref{fig:cost} (right),  all three cases converge due to the relatively mild sufficient condition (\ref{eq:suff condition M^PS}).

\section*{Conclusion}
In this work, we have introduced the framework of performative control
and studied the conditions under which a PSC policy exists. We
have analyzed the impact of system stability on the existence of the
PSC policy, and  {proposed a condition} on the sum of the distributional sensitivities and the temporally backwards decaying
of sensitivities for {almost surely strongly stable} and almost surely unstable systems, respectively. We have also proposed an RSGD algorithm that
converges to the PSC policy in a mean-square sense. The extension of our current results to general control policies and general control costs [cf. A\ref{A4}], will be explored in future work.

\newpage 
\bibliography{main}
\bibliographystyle{plain}

\medskip
\newpage 
\appendix
\section{Application Example and Experiment Detail.}
In this section, we elaborate a concrete example of stock market  risk minimization to justify the policy-dependent state transition model in \eqref{trans_matrix}. We also conduct numerical experiment based on this example to demonstrate the efficacy of our developed theory.
\begin{example}
	(Stock Market Risk Minimization). We describe a risk minimization
	problem for stock market investment to illustrate the application
	of (\ref{eq: PP optimal formulation}). Consider an investor trading
	a total number of $L$ stocks over a period of $T$ trading days. The observed market price $s_t^{(l)}$ of the $n$-th stock at the $t$-th day follows a stochastic volatility model of $\log s_{t+1}^{(l)}=\log s_t^{(l)}+\frac{r-\frac{1}{2}\left(v_t^{(l)}\right)^2}{T}+\frac{v_t^{(l)}}{\sqrt{T}}, s_1^{(l)}>0, \forall 1 \leq t<T, 1 \leq l \leq L$, where $r>0$ is the riskless interest rate per day, $v_t^{(l)}$ and $s_1^{(l)}$ are the unobservable independent random volatility process and the constant initial stock price associated with the $l$-th stock, respectively \cite{pham2009numerical,reno2006nonparametric}. Let $q_{t}^{\left(l\right)}$ and $q_{1}^{\left(l\right)}$ be the
	the total return at the $t$-th day before the market opens and the
	initial investment associated with the $l$-th stock, respectively.
	The investor maintains a portfolio for each $q_{t}^{\left(l\right)},$
	which is parameterized by a row vector $\mathbf{m}^{\left(l\right)}=\left[m^{l,1},\cdots,m^{l,M}\right]$
	with $m^{l,i}\in\left[0,1\right],\forall1\leq i\leq L,$ being the
	weight of allocation and $\sum_{i=1}^{L}m^{l,i}=1$. Specifically, at the $t$-th day when the market opens, the investor immediately allocates $q_t^{(l)}$ proportionally to buy the $N$ stocks according to the weight vector $\mathbf{m}^{(l)}$. The total systemic
	risk associated with $q_{t}^{l}$ during the period between the $t$-th
	day right before market opening and the $\left(t+1\right)$-th day
	right before market opening is $\sum_{i=1}^{L}\left(m^{l,i}w_{t}^{\left(i\right)}h\right)+1\cdot w_{t+1}^{\left(l\right)}\left(24-h\right),$
	where $w_{t}^{i}$ is the per hour risk associated with the $i$-th
	stock in the $t$-th day and $h$ is the total number of trading hours
	per day. is $\sum_{i=1}^L m^{l, i} e\left(\left(r-\frac{1}{2}\left(v_t^{(i)}\right)^2\right) \frac{1}{T}+v_t^{(i)} \frac{1}{\sqrt{T}}\right) q_t^l$. The investor then spends all this amount to buy back the $l$-th stock and then the $t$-th trading day ends. Each stock in the portfolio suffers from a proportionate random i.i.d. systemic risk $w_t^i$ per hour, i.e., holding the $i$-th stock with a ratio of $m^{l, i}$ in a portfolio for a total number of $h$ trading hours in the $t$-th day will incur a systemic risk of $m^{l, i} w_t^{(i)} h$. Denote the mean-shifted return as $\mathbf{r}_{t}=\mathbf{q}_{t}-\mathbb{E}\left[\mathbf{q}_{t}\right]$
	and let $\mathbf{x}_{t}=\left[\begin{array}{c}
		\mathbf{r}_{t}\\
		\mathbf{q}_{t}
	\end{array}\right]$. It follows that the evolution of $\mathbf{x}_{t}$ can be characterized
	by the canonical form of
	\begin{align}
		& \mathbf{x}_{t+1}=\left(\text{\ensuremath{\mathbf{A}}}+\mathbf{\Delta}_{t}\right)\mathbf{x}_{t}+\mathbf{u}_{t}^{\left(\mathbf{M}\right)}+\widetilde{\mathbf{w}}_{t},\label{eq: Example 1-dynamics}\\
		& \mathbf{u}_{t}^{\left(\mathbf{M}\right)}=\mathbf{M}\widetilde{\mathbf{w}}_{t-1}\ \ with\ \ \mathbf{B}=\mathbf{I},\mathbf{K}=\mathbf{0},\forall1\leq t<T.\label{eq: Example 1 u_t (M)}
	\end{align}
	where $\mathbf{M}=\frac{h}{24-h}\left[\mathbf{m}^{\left(1\right)};\cdots;\mathbf{m}^{\left(L\right)}\right],$
	$\mathbf{w}_{t-1}=\left(24-h\right)\left[w_{t}^{\left(1\right)},\cdots,w_{t}^{\left(L\right)}\right]^{T},$
	$\widetilde{\mathbf{w}}_{t}=\left[\begin{array}{c}
		\mathbf{w}_{t}-\mathbb{E}\left[\mathbf{w}_{t}\right]\\
		\mathbf{w}_{t}
	\end{array}\right],$ the state transition matrix $\mathbf{A}=\mathbf{I}_{2L\times2L}$
	and the policy-dependent state transition perturbation $\mathbf{\Delta}_{t}$
	is given by (\ref{eq:Example 1-A}) and (\ref{eq:Example 1-V(M)}).
	\begin{align}
		& \mathbf{\Delta}_{t}=\left[\begin{array}{cc}
			\mathbb{E}\left[\mathbf{V}_{t}^{\left(\mathbf{M}\right)}\right]-\mathbf{I}_{L\times L} & \mathbf{V}_{t}^{\left(\mathbf{M}\right)}-\mathbb{E}\left[\mathbf{V}_{t}^{\left(\mathbf{M}\right)}\right]\\
			\mathbf{0}_{L\times L} & \mathbf{V}_{t}^{\left(\mathbf{M}\right)}-\mathbf{I}_{L\times L}
		\end{array}\right],\label{eq:Example 1-A}\\
		& \mathbf{V}_{t}^{\left(\mathbf{M}\right)}=\mathrm{diag}\left(\left[\sum_{i=1}^{L}m^{1,i}e^{\left(\frac{r-\frac{1}{2}\left(v_{t}^{\left(i\right)}\right)^{2}}{T}+\frac{v_{t}^{\left(i\right)}}{\sqrt{T}}\right)},\cdots,\sum_{i=1}^{L}m^{L,i}e^{\left(\frac{r-\frac{1}{2}\left(v_{t}^{\left(i\right)}\right)^{2}}{T}+\frac{v_{t}^{\left(i\right)}}{\sqrt{T}}\right)}\right]\right).\label{eq:Example 1-V(M)}
	\end{align}
	The investment risk minimization problem can thus be casted as 
	\begin{align*}
		\min_{\mathbf{M}\in \mathbb{M}}\ \  & C_{T}^{\mathbf{M}}=\mathbb{E}_{\left\{ \mathbf{\Delta}_{t}\right\} ,\left\{ \widetilde{\mathbf{w}}_{t}\right\} }\left[\sum_{t=1}^{T}\left(\left\Vert \left[\mathbf{I}_{L},\mathbf{0}_{L\times L}\right]\cdot\mathbf{x}_{t}\right\Vert ^{2}+\left\Vert \left[\mathbf{I}_{L},\mathbf{0}_{L\times L}\right]\cdot\text{\ensuremath{\mathbf{u}_{t}^{\left(\mathbf{M}\right)}}}\right\Vert ^{2}\right)\right],\ \ \ \ s.t.\ \ (\ref{eq: Example 1-dynamics})-(\ref{eq:Example 1-V(M)}).
	\end{align*}
	Tackling (\ref{eq: PP optimal formulation}) leads to an optimized
	stock investment policy $\mathbf{M}$ which takes the effects of policy-dependent
	random perturbations to the dynamics of stock returns into account.
\end{example}

	In the simulation, we set the number of stocks $L=10$ and the number of trading days $T=60$. The initial policy $\mathbf{M}_0$ is randomly chosen within the feasible set $\mathbb{M}=\{\mathbf{M}:\sum_{i=1}^{10}m^{l,i}=1, \forall 1\leq l\leq 10\}$. The entries of noise term ${\mathbf{w}}_t$ are independently and uniformly drawn from the interval $[0,1]$. For each $1\leq i\leq10$ and each $0\leq t<60$, we first obtain $\tilde{v}_t^{(i)}$ by sampling from a Gaussian distribution, i.e., $\tilde{v}_t^{(i)}\sim\ca{N}(\log(\varepsilon_t),0.2)$, where $\varepsilon_t$ denotes the sensitivity at $t$-th trading day. The $\tilde{v}_t^{(i)}$ is then projected to the interval $[-0.6,0.6]$ to obtain $v_t^{(i)}$ in (\ref{eq:Example 1-V(M)}). We let the total number iterations $N=1000$, and the stepsize $\eta_{n}$ in Algorithm \ref{algo:algo_stochastic} is set to be $0.01$, $\forall 0\leq n \leq 1000$.
	
	 The sensitivities are shown as follows:

\textbf{Sensitivities with Ascending Sequence.} $\varepsilon_{\mathrm{a}}$=[1.25797477e-07 5.03189910e-07 1.13217730e-06 2.01275964e-06
3.14493693e-06 4.52870919e-06 6.16407639e-06 8.05103855e-06
1.01895957e-05 1.25797477e-05 1.52214948e-05 1.81148367e-05
2.12597737e-05 2.46563056e-05 2.83044324e-05 3.22041542e-05
3.63554710e-05 4.07583827e-05 4.54128893e-05 5.03189910e-05
4.66004175e-03 5.35796616e-03 6.12231163e-03 6.95609731e-03
7.86234234e-03 8.84406585e-03 9.90428699e-03 1.10460249e-02
1.22722987e-02 1.35861276e-02 1.49905306e-02 1.64885270e-02
1.80831358e-02 1.97773762e-02 2.15742674e-02 2.34768284e-02
2.54880785e-02 2.76110367e-02 2.98487222e-02 3.22041542e-02
4.33504397e-02 4.66004175e-02 5.00089002e-02 5.35796616e-02
5.73164756e-02 6.12231163e-02 6.53033575e-02 6.95609731e-02
7.39997371e-02 7.86234234e-02 8.34358059e-02 8.84406585e-02
9.36417552e-02 9.90428699e-02 1.04647776e-01 1.10460249e-01
1.16484061e-01 1.22722987e-01 1.29180801e-01 1.35861276e-01]

\textbf{Sensitivities with Descending Sequence.} $\varepsilon_{\mathrm{d}}$=[1.35861276e-01 1.29180801e-01 1.22722987e-01 1.16484061e-01
1.10460249e-01 1.04647776e-01 9.90428699e-02 9.36417552e-02
8.84406585e-02 8.34358059e-02 7.86234234e-02 7.39997371e-02
6.95609731e-02 6.53033575e-02 6.12231163e-02 5.73164756e-02
5.35796616e-02 5.00089002e-02 4.66004175e-02 4.33504397e-02
3.22041542e-02 2.98487222e-02 2.76110367e-02 2.54880785e-02
2.34768284e-02 2.15742674e-02 1.97773762e-02 1.80831358e-02
1.64885270e-02 1.49905306e-02 1.35861276e-02 1.22722987e-02
1.10460249e-02 9.90428699e-03 8.84406585e-03 7.86234234e-03
6.95609731e-03 6.12231163e-03 5.35796616e-03 4.66004175e-03
5.03189910e-05 4.54128893e-05 4.07583827e-05 3.63554710e-05
3.22041542e-05 2.83044324e-05 2.46563056e-05 2.12597737e-05
1.81148367e-05 1.52214948e-05 1.25797477e-05 1.01895957e-05
8.05103855e-06 6.16407639e-06 4.52870919e-06 3.14493693e-06
2.01275964e-06 1.13217730e-06 5.03189910e-07 1.25797477e-07]

\textbf{Sensitivities with Random Sequence.} $\varepsilon_{\mathrm{r}}$=[2.54880785e-02 3.63554710e-05 2.34768284e-02 1.64885270e-02
1.25797477e-07 9.36417552e-02 1.10460249e-02 6.16407639e-06
2.76110367e-02 5.00089002e-02 1.16484061e-01 1.04647776e-01
2.15742674e-02 2.46563056e-05 3.14493693e-06 4.66004175e-03
4.54128893e-05 3.22041542e-02 1.22722987e-02 4.33504397e-02
8.84406585e-03 6.95609731e-03 5.03189910e-05 1.01895957e-05
7.39997371e-02 1.97773762e-02 7.86234234e-02 1.35861276e-02
8.84406585e-02 1.22722987e-01 4.07583827e-05 8.05103855e-06
2.98487222e-02 9.90428699e-03 6.95609731e-02 8.34358059e-02
6.12231163e-03 1.81148367e-05 1.52214948e-05 2.12597737e-05
6.53033575e-02 5.03189910e-07 5.35796616e-03 9.90428699e-02
7.86234234e-03 1.10460249e-01 4.66004175e-02 2.83044324e-05
3.22041542e-05 2.01275964e-06 1.13217730e-06 6.12231163e-02
5.35796616e-02 1.49905306e-02 1.25797477e-05 1.80831358e-02
4.52870919e-06 1.35861276e-01 1.29180801e-01 5.73164756e-02]

The non-ordered lists of sensitivity values in epsilon $\varepsilon_{\mathrm{a}}$,$\varepsilon_{\mathrm{d}}$, and $\varepsilon_{\mathrm{r}}$ are the same. The state trajectory $\mathbf{x}_t$ and control action $\mathbf{u}_t$ are generated according to \eqref{eq: Example 1-dynamics}- \eqref{eq:Example 1-V(M)} based on the above system parameter configurations.

\section{Auxiliariy Lemmas for Per Stage Cost Function }\label{app:auxiliraries}
In this part, we provide several lemmas to characterize the gradient
properties and the smoothness of the per stage cost $c_{t}$. We first
summarize a collection of notations. We denote $\nabla_{\mathbf{x}}c_{t}\left(\mathbf{x}_{t}^{\left(\mathbf{M}\right)},\mathbf{u}_{t}^{\left(\mathbf{M}\right)}\right)$
and $\nabla_{\mathbf{u}}c_{t}\left(\mathbf{x}_{t}^{\left(\mathbf{M}\right)},\mathbf{u}_{t}^{\left(\mathbf{M}\right)}\right)$
as the gradient of $c_{t}$ w.r.t. the system state $\mathbf{x}_{t}^{\left(\mathbf{M}\right)}$
and control action $\mathbf{u}_{t}^{\left(\mathbf{M}\right)},$ respectively.
We denote $\nabla_{\mathbf{M}}c_{t}\left(\mathbf{x}_{t}^{\left(\mathbf{M}\right)},\mathbf{u}_{t}^{\left(\mathbf{M}\right)}\right)$
as the gradient of $c_{t}$ w.r.t. the variable $\mathbf{M}$ in $\mathbf{x}_{t}^{\left(\mathbf{M}\right)}$
and $\mathbf{u}_{t}^{\left(\mathbf{M}\right)}$ under any given realization
of the policy-dependent state transition perturbations $\left\{ \mathbf{\Delta}_{i},0\leq i<t\right\} .$ 

Based on the chain rule and the relationship between $\mathbf{M}$
and the pair $\left(\mathbf{x}_{t},\mathbf{u}_{t}\right)$ in  {\eqref{eq:DAP u_t^M}
	and \eqref{eq:x_t representation}}, the expression of $\nabla_{\mathbf{M}}c_{t}\left(\mathbf{x}_{t}^{\left(\mathbf{M}\right)},\mathbf{u}_{t}^{\left(\mathbf{M}\right)}\right)$
is given by
\begin{align}
	& \nabla_{\mathbf{M}}c_{t}\left(\mathbf{x}_{t}^{\left(\mathbf{M}\right)},\mathbf{u}_{t}^{\left(\mathbf{M}\right)}\right)\label{eq: grad M ct}\\
	&=\sum_{i=0}^{t-1} \left(\prod_{j=i+1}^{t-1}\left(\mathbf{1}_{j<t}\left(\widetilde{\mathbf{A}}+\mathbf{\Delta}_{j}\right)+\mathbf{1}_{j=t}\mathbf{I}\right)\mathbf{B}\right)^{\top}\nabla_{\mathbf{x}}c_{t}\left(\mathbf{x}_{t}^{\left(\mathbf{M}\right)},\mathbf{u}_{t}^{\left(\mathbf{M}\right)}\right)\left(\left[\mathbf{w}\right]_{i-1}^{H}\right)^{\top} \nonumber\\
	& -\sum_{i=0}^{t-1}\left\{ \left(\mathbf{K}\prod_{j=i+1}^{t-1}\left(\mathbf{1}_{j<t}\left(\widetilde{\mathbf{A}}+\mathbf{\Delta}_{j}\right)+\mathbf{1}_{j=t}\mathbf{I}\right)\mathbf{B}\right)^{\top}\nabla_{\mathbf{u}}c_{t}\left(\mathbf{x}_{t}^{\left(\mathbf{M}\right)},\mathbf{u}_{t}^{\left(\mathbf{M}\right)}\right)\left(\left[\mathbf{w}\right]_{i-1}^{H}\right)^{\top}\right\} \nonumber \\
	& +\nabla_{\mathbf{u}}c_{t}\left(\mathbf{x}_{t}^{\left(\mathbf{M}\right)},\mathbf{u}_{t}^{\left(\mathbf{M}\right)}\right)\left(\left[\mathbf{w}\right]_{t-1}^{H}\right)^{\top}.\nonumber
\end{align}
For the ease of notation, without ambiguity we also occasionally use
the notation $c_{t}\left(\mathbf{M};\left\{ \mathbf{\Delta}_{i}\right\} _{0\leq i<t}\right)$
to denote $c_{t}\left(\mathbf{x}_{t}^{\left(\mathbf{M}\right)},\mathbf{u}_{t}^{\left(\mathbf{M}\right)}\right)$
under a sequence of policy-dependent state transition perturbations
$\left\{ \mathbf{\Delta}_{i},0\leq i<t\right\} .$ In this case, $\nabla_{\mathbf{M}}c_{t}\left(\mathbf{M};\left\{ \mathbf{\Delta}_{i}\right\} _{0\leq i<t}\right)$
denote the gradient taken w.r.t. the first argument $\mathbf{M}.$ 

Moreover, the total cost function can be represented as 
\begin{align}
	& J_{T}\left(\mathbf{M};\left\{ \mathbf{\Delta}_{t}\right\} _{0\leq t<T}\right)=\sum_{t=0}^{T}c_{t}\left(\mathbf{M};\left\{ \mathbf{\Delta}_{i}\right\} _{0\leq i<t}\right).
\end{align}
It is clear that
\begin{align}
	& C_{T}\left(\mathbf{M};\mathbf{M}^{\prime}\right)=\mathbb{E}_{\mathbf{x}_{0},\left\{ \mathbf{\Delta}_{t}\sim\mathcal{D}_{t}\left(\mathbf{M}^{\prime}\right)\right\} _{0\leq i<T},\left\{ \mathbf{w}_{i}\right\} _{0\leq i<T}}J_{T}\left(\mathbf{M};\left\{ \mathbf{\Delta}_{t}\right\} _{0\leq t<T}\right).
\end{align}
Unless otherwise specified, $\nabla J_{T}\left(\mathbf{M};\left\{ \mathbf{\Delta}_{t}\right\} _{0\leq t<T}\right)$
and $\nabla C_{T}\left(\mathbf{M};\mathbf{M}^{\prime}\right)$ denotes
the gradient taken w.r.t. the first argument $\mathbf{M}$.

\textbf{Convexity.} Denote the per stage expected cost 
\begin{align*}
	& f_{t}\left(\mathbf{M};\mathbf{M}_{1}\right)=\mathbb{E}_{\mathbf{x}_{0},\left\{ \mathbf{\Delta}_{i}\sim\mathcal{D}_{i}\left(\mathbf{M}_{1}\right)\right\} _{0\leq i<t},\left\{ \mathbf{w}_{i}\right\} _{0\leq i<t}}\left[c_{t}\left(\mathbf{M};\left\{ \mathbf{\Delta}_{i}\right\} _{0\leq i<t}\right)\right],\forall1\leq t\leq T,
\end{align*}
where the distribution of policy-dependent perturbation is changed
from $\mathbf{\Delta}_{i}\sim\mathcal{D}_{i}\left(\mathbf{M}\right)$
to $\mathbf{\Delta}_{i}\sim\mathcal{D}_{i}\left(\mathbf{M}_{1}\right)$,
$\forall0\leq i<t.$ 

We have the following lemma characterizing the convexity of $f_{t}\left(\mathbf{M};\mathbf{M}_{1}\right),\forall1\leq t\leq T.$
\begin{lemma}\label{p_stage_a_cost}
	The per stage expected cost $f_{t}\left(\mathbf{M};\mathbf{M}^{\prime}\right)$
	is a convex function of $\mathbf{M}$, $\forall\mathbf{M},\mathbf{M}^{\prime}\ensuremath{\in\mathbb{M}},$
	$\forall1\leq t\leq T$. 
\end{lemma}

\begin{proof}
	We prove this lemma following the convex property of the function
	$c_{t}\left(\mathbf{x},\mathbf{u}\right)$ in  {Assumption A \ref{A4}}. Let $0<\theta<1$ and $\mathbf{M}_{1},\text{\ensuremath{\mathbf{M}}}_{2},\mathbf{M}^{\prime}\in\mathbb{M}$.
	Consder the weighted policy $\theta\mathbf{M}_{1}+\left(1-\theta\right)\mathbf{M}_{2}$,
	we have 
	\begin{align}
		& \mathbf{x}_{t}^{\left(\theta\mathbf{M}_{1}+\left(1-\theta\right)\mathbf{M}_{2}\right)}=\prod_{i=0}^{t-1}\left(\widetilde{\mathbf{A}}+\mathbf{\Delta}_{i}\right)\mathbf{x}_{0}+\sum_{i=0}^{t-1}\prod_{j=i+1}^{t-1}\left(\mathbf{1}_{j<t}\left(\widetilde{\mathbf{A}}+\mathbf{\Delta}_{j}\right)+\mathbf{1}_{j=t}\mathbf{I}\right)\mathbf{B}\\
		& \cdot\left(\theta\mathbf{M}_{1}+\left(1-\theta\right)\mathbf{M}_{2}\right)\mathbf{M}\left[\mathbf{w}\right]_{i-1}^{H}+\sum_{i=0}^{t-1}\prod_{j=i+1}^{t-1}\left(\mathbf{1}_{j<t}\left(\widetilde{\mathbf{A}}+\mathbf{\Delta}_{j}\right)+\mathbf{1}_{j=t}\mathbf{I}\right)\mathbf{w}_{i}\nonumber\\
		& =\theta\left(\prod_{i=0}^{t-1}\left(\widetilde{\mathbf{A}}+\mathbf{\Delta}_{i}\right)\mathbf{x}_{0}+\sum_{i=0}^{t-1}\prod_{j=i+1}^{t-1}\left(\mathbf{1}_{j<t}\left(\widetilde{\mathbf{A}}+\mathbf{\Delta}_{j}\right)+\mathbf{1}_{j=t}\mathbf{I}\right)\mathbf{B}\mathbf{M}_{1}\left[\mathbf{w}\right]_{i-1}^{H}\right.\nonumber\\
		& \left.+\sum_{i=0}^{t-1}\prod_{j=i+1}^{t-1}\left(\mathbf{1}_{j<t}\left(\widetilde{\mathbf{A}}+\mathbf{\Delta}_{j}\right)+\mathbf{1}_{j=t}\mathbf{I}\right)\mathbf{w}_{i}\right)+\left(1-\theta\right)\nonumber\\
		& \cdot\left(\prod_{i=0}^{t-1}\left(\widetilde{\mathbf{A}}+\mathbf{\Delta}_{i}\right)\mathbf{x}_{0}+\sum_{i=0}^{t-1}\prod_{j=i+1}^{t-1}\left(\mathbf{1}_{j<t}\left(\widetilde{\mathbf{A}}+\mathbf{\Delta}_{j}\right)+\mathbf{1}_{j=t}\mathbf{I}\right)\mathbf{B}\mathbf{M}_{2}\left[\mathbf{w}\right]_{i-1}^{H}\right.\nonumber\\
		& \left.+\sum_{i=0}^{t-1}\prod_{j=i+1}^{t-1}\left(\mathbf{1}_{j<t}\left(\widetilde{\mathbf{A}}+\mathbf{\Delta}_{j}\right)+\mathbf{1}_{j=t}\mathbf{I}\right)\mathbf{w}_{i}\right)\nonumber\\
		& =\theta\mathbf{x}_{t}^{\mathbf{M}_{1}}+\left(1-\theta\right)\mathbf{x}_{t}^{\mathbf{M}_{2}}.\nonumber
	\end{align}
	As a result, 
	\begin{align}
		& f_{t}\left(\theta\mathbf{M}_{1}+\left(1-\theta\right)\mathbf{M}_{2};\mathbf{M}^{\prime}\right)\label{eq:convex f_t}\\
		&=\mathbb{E}_{\mathbf{x}_{0},\left\{ \mathbf{\Delta}_{i}\sim\mathcal{D}_{i}\left(\mathbf{M}^{\prime}\right)\right\} _{0\leq i<t},\left\{ \mathbf{w}_{i}\right\} _{0\leq i<t}}\left[c_{t}\left(\mathbf{x}_{t}^{\left(\mathbf{M}_{1}+\left(1-\theta\right)\mathbf{M}_{2}\right)},\mathbf{u}_{t}^{\left(\mathbf{M}_{1}+\left(1-\theta\right)\mathbf{M}_{2}\right)}\right)\right]\nonumber\\
		& =\mathbb{E}_{\mathbf{x}_{0},\left\{ \mathbf{\Delta}_{i}\sim\mathcal{D}_{i}\left(\mathbf{M}^{\prime}\right)\right\} _{0\leq i<t},\left\{ \mathbf{w}_{i}\right\} _{0\leq i<t}}\left[c_{t}\left(\theta\mathbf{x}_{t}^{\mathbf{M}_{1}}+\left(1-\theta\right)\mathbf{x}_{t}^{\mathbf{M}_{2}},\theta\mathbf{u}_{t}^{\mathbf{M}_{1}}+\left(1-\theta\right)\mathbf{u}_{t}^{\mathbf{M}_{2}}\right)\right]\nonumber \\
		& \overset{\left(\ref{eq:convex f_t}.a\right)}{\leq}\mathbb{E}_{\mathbf{x}_{0},\left\{ \mathbf{\Delta}_{i}\sim\mathcal{D}_{i}\left(\mathbf{M}^{\prime}\right)\right\} _{0\leq i<t},\left\{ \mathbf{w}_{i}\right\} _{0\leq i<t}}\left[\theta c_{t}\left(\mathbf{x}_{t}^{\mathbf{M}_{1}},\mathbf{u}_{t}^{\mathbf{M}_{1}}\right)+\left(1-\theta\right)c_{t}\left(\mathbf{x}_{t}^{\mathbf{M}_{2}},\mathbf{u}_{t}^{\mathbf{M}_{2}}\right)\right]\nonumber \\
		& =\theta f_{t}\left(\mathbf{M}_{1};\mathbf{M}^{\prime}\right)+\left(1-\theta\right)f_{t}\left(\mathbf{M}_{2};\mathbf{M}^{\prime}\right),\nonumber 
	\end{align}
	where inequality ($\ref{eq:convex f_t}.a$) holds becasue of the convexity
	of $c_{t}\left(\mathbf{x},\text{\ensuremath{\mathbf{u}}}\right)$.
\end{proof}

We next have a key lemma characterizing the strong convexity of $f_{t}\left(\mathbf{M};\mathbf{M}^{\prime}\right),\forall H\leq t\leq T.$
\begin{lemma}\label{s_cvx}
	The per stage expected cost $f_{t}\left(\mathbf{M};\mathbf{M}^{\prime}\right)$
	is a  {$\min\left\{ \frac{\mu}{2},\frac{\mu\sigma^{2}\gamma^{2}}{64\kappa^{10}}\right\} $}-strongly
	convex function of $\mathbf{M},$ $\forall\mathbf{M},\mathbf{M}^{\prime}\ensuremath{\in\mathbb{M}},$
	$\forall H\leq t\leq T$. 
\end{lemma}

\begin{proof}
	We prove this lemma following the strong convexity of the function
	$c_{t}\left(\mathbf{x},\mathbf{u}\right)$ in  {Assumption A \ref{A4}}. Let $\mathbf{M}_{1},\text{\ensuremath{\mathbf{M}}}_{2},\mathbf{M}^{\prime}\in\mathbb{M}$.
	Consider a time step $t$ such that $H\leq t\leq T.$ For any given
	realizations of $\left\{ \mathbf{\Delta}_{i}\right\} _{0\leq i<t}$,
	where $\mathbf{\Delta}_{i}\sim\mathcal{D}_{i}\left(\mathbf{M}^{\prime}\right),\forall0\leq i<t,$
	we have 
	
	\begin{align}
		& c_{t}\left(\mathbf{x}_{t}^{\left(\mathbf{M}_{1}\right)},\mathbf{u}_{t}^{\left(\mathbf{M}_{1}\right)}\right)-c_{t}\left(\mathbf{x}_{t}^{\left(\mathbf{M}_{2}\right)},\mathbf{u}_{t}^{\left(\mathbf{M}_{2}\right)}\right)\label{eq:ct(M)-ct(M')}\\
		& \overset{\left(\ref{eq:ct(M)-ct(M')}.a\right)}{\geq}\left[\begin{array}{c}
			\nabla_{\mathbf{x}}c_{t}\left(\mathbf{x}_{t}^{\left(\mathbf{M}_{2}\right)},\mathbf{u}_{t}^{\left(\mathbf{M}_{2}\right)}\right)\\
			\nabla_{\mathbf{u}}c_{t}\left(\mathbf{x}_{t}^{\left(\mathbf{M}_{2}\right)},\mathbf{u}_{t}^{\left(\mathbf{M}_{2}\right)}\right)
		\end{array}\right]^{T}\left[\begin{array}{c}
			\mathbf{x}_{t}^{\left(\mathbf{M}_{1}\right)}-\mathbf{x}_{t}^{\left(\mathbf{M}_{2}\right)}\\
			\mathbf{u}_{t}^{\left(\mathbf{M}_{1}\right)}-\mathbf{u}_{t}^{\left(\mathbf{M}_{2}\right)}
		\end{array}\right]+\frac{\mu}{2}\left\Vert \begin{array}{c}
			\mathbf{x}_{t}^{\left(\mathbf{M}_{1}\right)}-\mathbf{x}_{t}^{\left(\mathbf{M}_{2}\right)}\\
			\mathbf{u}_{t}^{\left(\mathbf{M}_{1}\right)}-\mathbf{u}_{t}^{\left(\mathbf{M}_{2}\right)}
		\end{array}\right\Vert ^{2},\nonumber 
	\end{align}
	where inequality $\left(\ref{eq:ct(M)-ct(M')}.a\right)$ holds because
	of the strong convexity assumption of $c_{t}$ in Assumption  {A\ref{A4}.}
	
	Based on the representations of the pair $\left(\mathbf{x}_{t},\mathbf{u}_{t}\right)$
	in  {\eqref{eq:DAP u_t^M} and \eqref{eq:x_t representation}}, it follows that
	\begin{align}
		\mathbf{x}_{t}^{\left(\mathbf{M}_{1}\right)}-\mathbf{x}_{t}^{\left(\mathbf{M}_{2}\right)}= & \sum_{i=0}^{t-1}\prod_{j=i+1}^{t-1}\left(\mathbf{1}_{j<t}\left(\widetilde{\mathbf{A}}+\mathbf{\Delta}_{j}\right)+\mathbf{1}_{j=t}\mathbf{I}\right)\mathbf{B}\left(\mathbf{M}_{1}-\mathbf{M}_{2}\right)\left[\mathbf{w}\right]_{i-1}^{H}\label{eq:x(M,W)}\\
		\mathbf{u}_{t}^{\left(\mathbf{M}_{1}\right)}-\mathbf{u}_{t}^{\left(\mathbf{M}_{2}\right)}= & -\mathbf{K}\sum_{i=0}^{t-1}\prod_{j=i+1}^{t-1}\left(\mathbf{1}_{j<t}\left(\widetilde{\mathbf{A}}+\mathbf{\Delta}_{j}\right)+\mathbf{1}_{j=t}\mathbf{I}\right)\mathbf{B}\left(\mathbf{M}_{1}-\mathbf{M}_{2}\right)\left[\mathbf{w}\right]_{i-1}^{H}\label{eq:u(M,W)}\\
		& +\left(\mathbf{M}_{1}-\mathbf{M}_{2}\right)\left[\mathbf{w}\right]_{t-1}^{H}.\nonumber 
	\end{align}
	Therefore, we obtain 
	\begin{align}
		& \left[\begin{array}{c}
			\nabla_{\mathbf{x}}c_{t}\left(\mathbf{x}_{t}^{\left(\mathbf{M}_{2}\right)},\mathbf{u}_{t}^{\left(\mathbf{M}_{2}\right)}\right)\\
			\nabla_{\mathbf{u}}c_{t}\left(\mathbf{x}_{t}^{\left(\mathbf{M}_{2}\right)},\mathbf{u}_{t}^{\left(\mathbf{M}_{2}\right)}\right)
		\end{array}\right]^{T}\left[\begin{array}{c}
			\mathbf{x}_{t}^{\left(\mathbf{M}_{1}\right)}-\mathbf{x}_{t}^{\left(\mathbf{M}_{2}\right)}\\
			\mathbf{u}_{t}^{\left(\mathbf{M}_{1}\right)}-\mathbf{u}_{t}^{\left(\mathbf{M}_{2}\right)}
		\end{array}\right]\label{eq: gradient term 1}\\
		& =\sum_{i=0}^{t-1}\nabla_{\mathbf{x}}^{\top}c_{t}\left(\mathbf{x}_{t}^{\left(\mathbf{M}_{2}\right)},\mathbf{u}_{t}^{\left(\mathbf{M}_{2}\right)}\right)\left(\prod_{j=i+1}^{t-1}\left(\mathbf{1}_{j<t}\left(\widetilde{\mathbf{A}}+\mathbf{\Delta}_{j}\right)+\mathbf{1}_{j=t}\mathbf{I}\right)\mathbf{B}\left(\mathbf{M}_{1}-\mathbf{M}_{2}\right)\left[\mathbf{w}\right]_{i-1}^{H}\right)\nonumber \\
		& +\sum_{i=0}^{t-1}\nabla_{\mathbf{u}}^{\top}c_{t}\left(\mathbf{x}_{t}^{\left(\mathbf{M}_{2}\right)},\mathbf{u}_{t}^{\left(\mathbf{M}_{2}\right)}\right)\left(-\mathbf{K}\sum_{i=0}^{t-1}\prod_{j=i+1}^{t-1}\left(\mathbf{1}_{j<t}\left(\widetilde{\mathbf{A}}+\mathbf{\Delta}_{j}\right)+\mathbf{1}_{j=t}\mathbf{I}\right)\mathbf{B}\left(\mathbf{M}_{1}-\mathbf{M}_{2}\right)\left[\mathbf{w}\right]_{i-1}^{H}\right.\nonumber \\
		& \left.+\left(\mathbf{M}_{1}-\mathbf{M}_{2}\right)\left[\mathbf{w}\right]_{t-1}^{H}\right)\nonumber \\
		& =\text{\ensuremath{\mathrm{Tr}}}\left(\left(\sum_{i=0}^{t-1}\left[\mathbf{w}\right]_{i-1}^{H}\nabla_{\mathbf{x}}^{\top}c_{t}\left(\mathbf{x}_{t}^{\left(\mathbf{M}_{2}\right)},\mathbf{u}_{t}^{\left(\mathbf{M}_{2}\right)}\right)\prod_{j=i+1}^{t-1}\left(\mathbf{1}_{j<t}\left(\widetilde{\mathbf{A}}+\mathbf{\Delta}_{j}\right)+\mathbf{1}_{j=t}\mathbf{I}\right)\mathbf{B}\right)\left(\mathbf{M}_{1}-\mathbf{M}_{2}\right)\right)\nonumber \\
		& +\mathrm{Tr}\left(\left(-\sum_{i=0}^{t-1}\left[\mathbf{w}\right]_{i-1}^{H}\nabla_{\mathbf{u}}^{\top}c_{t}\left(\mathbf{x}_{t}^{\left(\mathbf{M}_{2}\right)},\mathbf{u}_{t}^{\left(\mathbf{M}_{2}\right)}\right)\mathbf{K}\prod_{j=i+1}^{t-1}\left(\mathbf{1}_{j<t}\left(\widetilde{\mathbf{A}}+\mathbf{\Delta}_{j}\right)+\mathbf{1}_{j=t}\mathbf{I}\right)\mathbf{B}\right.\right.\nonumber \\
		& \left.+\left.\left[\mathbf{w}\right]_{t-1}^{H}\nabla_{\mathbf{u}}^{\top}c_{t}\left(\mathbf{x}_{t}^{\left(\mathbf{M}_{2}\right)},\mathbf{u}_{t}^{\left(\mathbf{M}_{2}\right)}\right)\right)\left(\mathbf{M}_{1}-\mathbf{M}_{2}\right)\right)\nonumber \\
		& =\mathrm{Tr}\left(\nabla_{\mathbf{M}_{2}}^{\top}c_{t}\left(\mathbf{x}_{t}^{\left(\mathbf{M}_{2}\right)},\mathbf{u}_{t}^{\left(\mathbf{M}_{2}\right)}\right)\left(\mathbf{M}_{1}-\mathbf{M}_{2}\right)\right).\nonumber
	\end{align}
	
	Substitute (\ref{eq: gradient term 1}) back into (\ref{eq:ct(M)-ct(M')}),
	we have 
	\begin{align}
		c_{t}\left(\mathbf{x}_{t}^{\left(\mathbf{M}_{1}\right)},\mathbf{u}_{t}^{\left(\mathbf{M}_{1}\right)}\right)-c_{t}\left(\mathbf{x}_{t}^{\left(\mathbf{M}_{2}\right)},\mathbf{u}_{t}^{\left(\mathbf{M}_{2}\right)}\right)\geq & \mathrm{Tr}\left(\nabla_{\mathbf{M}_{2}}^{\top}c_{t}\left(\mathbf{x}_{t}^{\left(\mathbf{M}_{2}\right)},\mathbf{u}_{t}^{\left(\mathbf{M}_{2}\right)}\right)\left(\mathbf{M}_{1}-\mathbf{M}_{2}\right)\right)\label{eq:ct-ct fine grained}\\
		& +\frac{\mu}{2}\left(\left\Vert \mathbf{x}_{t}^{\left(\mathbf{M}_{1}\right)}-\mathbf{x}_{t}^{\left(\mathbf{M}_{2}\right)}\right\Vert ^{2}+\left\Vert \mathbf{u}_{t}^{\left(\mathbf{M}_{1}\right)}-\mathbf{u}_{t}^{\left(\mathbf{M}_{2}\right)}\right\Vert ^{2}\right).\nonumber 
	\end{align}
	
	Taking full expectation on both sides of (\ref{eq:ct-ct fine grained}),
	it follows 
	\begin{align}
		& f_{t}\left(\mathbf{M}_{1};\mathbf{M}^{\prime}\right)-f_{t}\left(\mathbf{M}_{1};\mathbf{M}^{\prime}\right)\geq\mathrm{Tr}\left(\nabla^{\top}f_{t}\left(\mathbf{M}_{2};\mathbf{M}^{\prime}\right)\cdot\left(\mathbf{M}_{1}-\mathbf{M}_{2}\right)\right)\label{eq: scvx ft - ft}\\
		& +\frac{\mu}{2}\mathbb{E}_{\mathbf{x}_{0},\left\{ \mathbf{\Delta}_{i}\sim\mathcal{D}_{i}\left(\mathbf{M}^{\prime}\right)\right\} _{0\leq i<t},\left\{ \mathbf{w}_{i}\right\} _{0\leq i<t}}\left[\left\Vert \mathbf{x}_{t}^{\left(\mathbf{M}_{1}\right)}-\mathbf{x}_{t}^{\left(\mathbf{M}_{2}\right)}\right\Vert ^{2}+\left\Vert \mathbf{u}_{t}^{\left(\mathbf{M}_{1}\right)}-\mathbf{u}_{t}^{\left(\mathbf{M}_{2}\right)}\right\Vert ^{2}\right].\nonumber 
	\end{align}
	
	We next analyze the last term in (\ref{eq: scvx ft - ft}). Consider
	the conditional expectation on $\left\{ \mathbf{w}_{i}\right\} _{0\leq i<t}$
	as follows
	\begin{align}
		& \mathbb{E}_{\mathbf{x}_{0},\left\{ \mathbf{\Delta}_{i}\sim\mathcal{D}_{i}\left(\mathbf{M}^{\prime}\right)\right\} _{0\leq i<t},\left\{ \mathbf{w}_{i}\right\} _{0\leq i<t}}\left[\left\Vert \mathbf{x}_{t}^{\left(\mathbf{M}_{1}\right)}-\mathbf{x}_{t}^{\left(\mathbf{M}_{2}\right)}\right\Vert ^{2}+\left\Vert \mathbf{u}_{t}^{\left(\mathbf{M}_{1}\right)}-\mathbf{u}_{t}^{\left(\mathbf{M}_{2}\right)}\right\Vert ^{2}\right]\label{eq:scvx ft - ft -1}\\
		& =\mathbb{E}_{\left\{ \mathbf{w}_{i}\right\} _{0\leq i<t}}\left[\mathbb{E}_{\mathbf{x}_{0},\left\{ \mathbf{\Delta}_{i}\sim\mathcal{D}_{i}\left(\mathbf{M}^{\prime}\right)\right\} _{0\leq i<t}}\left[\left.\left\Vert \mathbf{x}_{t}^{\left(\mathbf{M}_{1}\right)}-\mathbf{x}_{t}^{\left(\mathbf{M}_{2}\right)}\right\Vert ^{2}+\left\Vert \mathbf{u}_{t}^{\left(\mathbf{M}_{1}\right)}-\mathbf{u}_{t}^{\left(\mathbf{M}_{2}\right)}\right\Vert ^{2}\right|\left\{ \mathbf{w}_{i}\right\} _{0\leq i<t}\right]\right]\nonumber \\
		& \geq\mathbb{E}_{\left\{ \mathbf{w}_{i}\right\} _{0\leq i<t}}\left[\left\Vert \mathbb{E}_{\mathbf{x}_{0},\left\{ \mathbf{\Delta}_{i}\sim\mathcal{D}_{i}\left(\mathbf{M}^{\prime}\right)\right\} _{0\leq i<t}}\left[\left.\mathbf{x}_{t}^{\left(\mathbf{M}_{1}\right)}-\mathbf{x}_{t}^{\left(\mathbf{M}_{2}\right)}\right|\left\{ \mathbf{w}_{i}\right\} _{0\leq i<t}\right]\right\Vert ^{2}\right.\nonumber \\
		& +\left.\left\Vert \mathbb{E}_{\mathbf{x}_{0},\left\{ \mathbf{\Delta}_{i}\sim\mathcal{D}_{i}\left(\mathbf{M}^{\prime}\right)\right\} _{0\leq i<t}}\left[\left.\mathbf{u}_{t}^{\left(\mathbf{M}_{1}\right)}-\mathbf{u}_{t}^{\left(\mathbf{M}_{2}\right)}\right|\left\{ \mathbf{w}_{i}\right\} _{0\leq i<t}\right]\right\Vert ^{2}\right]\nonumber \\
		& \overset{\left(\ref{eq:scvx ft - ft -1}.a\right)}{=}\mathbb{E}_{\left\{ \mathbf{w}_{i}\right\} _{0\leq i<t}}\left[\left\Vert \sum_{i=0}^{t-1}\widetilde{\mathbf{A}}^{i}\mathbf{B}\left(\mathbf{M}_{1}-\mathbf{M}_{2}\right)\left[\mathbf{w}\right]_{i-1}^{H}\right\Vert ^{2}\right.\nonumber \\
		& +\left.\left\Vert \mathbf{K}\sum_{i=0}^{t-1}\widetilde{\mathbf{A}}^{i}\mathbf{B}\left(\mathbf{M}_{1}-\mathbf{M}_{2}\right)\left[\mathbf{w}\right]_{i-1}^{H}-\left(\mathbf{M}_{1}-\mathbf{M}_{2}\right)\left[\mathbf{w}\right]_{t-1}^{H}\right\Vert ^{2}\right],\nonumber 
	\end{align}
	where $\left(\ref{eq:scvx ft - ft -1}.a\right)$ holds becasue $\text{\ensuremath{\mathbb{E}\left[\mathbf{\Delta}_{t}\right]}=\ensuremath{\mathbf{0}}},\forall0\leq t<T.$ 
	
	We next consider each term on R.H.S. of $\left(\ref{eq:scvx ft - ft -1}.a\right).$
	We notice that 
	\begin{align}
		& \mathrm{vec}\left(\sum_{i=0}^{t-1}\widetilde{\mathbf{A}}^{i}\mathbf{B}\left(\mathbf{M}_{1}-\mathbf{M}_{2}\right)\left[\mathbf{w}\right]_{i-1}^{H}\right)=\left(\sum_{i=0}^{t-1}\left(\left[\mathbf{w}\right]_{i-1}^{H}\right)^{\top}\otimes\left(\widetilde{\mathbf{A}}^{i}\mathbf{B}\right)\right)\mathrm{vec}\left(\mathbf{M}_{1}-\mathbf{M}_{2}\right)
	\end{align}
	Therefore, 
	\begin{align}
		& \mathbb{E}_{\left\{ \mathbf{w}_{i}\right\} _{0\leq i<t}}\left[\left\Vert \sum_{i=0}^{t-1}\widetilde{\mathbf{A}}^{i}\mathbf{B}\left(\mathbf{M}_{1}-\mathbf{M}_{2}\right)\left[\mathbf{w}\right]_{i-1}^{H}\right\Vert ^{2}\right]\label{eq:scvx final term 1}\\
		& =\delta_{\mathbf{M}}^{\top}\mathbb{E}_{\left\{ \mathbf{w}_{i}\right\} _{0\leq i<t}}\left[\left(\sum_{i=0}^{t-1}\left(\left[\mathbf{w}\right]_{i-1}^{H}\right)^{\top}\otimes\left(\widetilde{\mathbf{A}}^{i}\mathbf{B}\right)\right)^{\top}\left(\sum_{i=0}^{t-1}\left(\left[\mathbf{w}\right]_{i-1}^{H}\right)^{\top}\otimes\left(\widetilde{\mathbf{A}}^{i}\mathbf{B}\right)\right)\right]\delta_{\mathbf{M}}\nonumber \\
		& =\delta_{\mathbf{M}}^{\top}\left(\left(\sum_{i_{1}=1}^{t}\sum_{i_{2}=1}^{t}\left(\Lambda_{i_{1}-i_{2}}\otimes\left(\widetilde{\mathbf{A}}^{i_{1}-1}\mathbf{B}\right)^{\top}\widetilde{\mathbf{A}}^{i_{2}-1}\mathbf{B}\right)\right)\otimes\mathbb{E}\left[\mathbf{w}_{t}\mathbf{w}_{t}^{\top}\right]\right)\delta_{\mathbf{M}}\nonumber \\
		& =\delta_{\mathbf{M}}^{\top}\left(\left(\left(\mathbf{I}_{H}\otimes\mathbf{B}^{\top}\right)\Psi\left(\mathbf{I}_{H}\otimes\mathbf{B}\right)\right)\otimes\mathbb{E}\left[\mathbf{w}_{t}\mathbf{w}_{t}^{\top}\right]\right)\delta_{\mathbf{M}},\nonumber 
	\end{align}
	where $\delta_{\mathbf{M}}=\mathrm{vec}\left(\mathbf{M}_{1}-\mathbf{M}_{2}\right)$,
	$\Lambda_{t}\in\mathbb{R}^{H\times H}$ with $\left[\Lambda_{l}\right]_{i,j}=1$
	if and only if $i-j=l$ and 0 otherwise and $\Psi=\sum_{i_{1}=1}^{t}\sum_{i_{2}=1}^{t}\left(\Lambda_{i_{1}-i_{2}}\otimes\left(\widetilde{\mathbf{A}}^{i_{1}-1}\mathbf{B}\right)^{\top}\widetilde{\mathbf{A}}^{i_{2}-1}\mathbf{B}\right).$
	
	Using similar approach, we can obtain 
	\begin{align}
		& \mathbb{E}_{\left\{ \mathbf{w}_{i}\right\} _{0\leq i<t}}\left[\left\Vert \mathbf{K}\sum_{i=0}^{t-1}\widetilde{\mathbf{A}}^{i}\mathbf{B}\left(\mathbf{M}_{1}-\mathbf{M}_{2}\right)\left[\mathbf{w}\right]_{i-1}^{H}-\left(\mathbf{M}_{1}-\mathbf{M}_{2}\right)\left[\mathbf{w}\right]_{t-1}^{H}\right\Vert ^{2}\right]\label{eq:scvx final term 2}\\
		& =\delta_{\mathbf{M}}^{\top}\left(\left(\left(\mathbf{I}_{H}\otimes\mathbf{B}^{\top}\right)\widetilde{\Psi}\left(\mathbf{I}_{H}\otimes\mathbf{B}^{\top}\right)-\Theta\left(\mathbf{I}_{H}\otimes\mathbf{B}\right)-\left(\mathbf{I}_{H}\otimes\mathbf{B}^{\top}\right)\Theta^{\top}+\mathbf{I}_{Hd_{u}}\right)\otimes\mathbb{E}\left[\mathbf{w}_{t}\mathbf{w}_{t}^{\top}\right]\right)\delta_{\mathbf{M}},\nonumber 
	\end{align}
	where $\widetilde{\Psi}=\sum_{i_{1}=1}^{t}\sum_{i_{2}=1}^{t}\left(\Lambda_{i_{1}-i_{2}}\otimes\left(\mathbf{K}\widetilde{\mathbf{A}}^{i_{1}-1}\mathbf{B}\right)^{\top}\mathbf{K}\widetilde{\mathbf{A}}^{i_{2}-1}\mathbf{B}\right)$
	and $\Theta=\left(\mathbf{I}_{H}\otimes\mathbf{K}\right)\otimes\left(\sum_{i=1}^{t}\Lambda_{-i}\otimes\left(\widetilde{\mathbf{A}}^{i-1}\mathbf{B}\right)\right)$
	
	Substitue (\ref{eq:scvx final term 2}) and (\ref{eq:scvx final term 1})
	back into (\ref{eq:scvx ft - ft -1}), it follows that
	\begin{align}
		& \mathbb{E}_{\mathbf{x}_{0},\left\{ \mathbf{\Delta}_{i}\sim\mathcal{D}_{i}\left(\mathbf{M}^{\prime}\right)\right\} _{0\leq i<t},\left\{ \mathbf{w}_{i}\right\} _{0\leq i<t}}\left[\left\Vert \mathbf{x}_{t}^{\left(\mathbf{M}_{1}\right)}-\mathbf{x}_{t}^{\left(\mathbf{M}_{2}\right)}\right\Vert ^{2}+\left\Vert \mathbf{u}_{t}^{\left(\mathbf{M}_{1}\right)}-\mathbf{u}_{t}^{\left(\mathbf{M}_{2}\right)}\right\Vert ^{2}\right]\label{eq:scvx final term}\\
		& =\delta_{\mathbf{M}}^{\top}\left(\Omega\otimes\mathbb{E}\left[\mathbf{w}_{t}\mathbf{w}_{t}^{\top}\right]\right)\delta_{\mathbf{M}}.\nonumber 
	\end{align}
	where 
	\begin{align}
		\Omega=&\left(\mathbf{I}_{H}\otimes\mathbf{B}^{\top}\right)\Psi\left(\mathbf{I}_{H}\otimes\mathbf{B}\right)+\left(\mathbf{I}_{H}\otimes\mathbf{B}^{\top}\right)\widetilde{\Psi}\left(\mathbf{I}_{H}\otimes\mathbf{B}^{\top}\right)\\
		&-\Theta\left(\mathbf{I}_{H}\otimes\mathbf{B}\right)-\left(\mathbf{I}_{H}\otimes\mathbf{B}^{\top}\right)\Theta^{\top}+\mathbf{I}_{Hd_{u}}.\nonumber
	\end{align}
	
	Based on Lemma F.1 and F.2 in \cite{agarwal2019logarithmic}, we have, $\forall H\leq t\leq T,$
	$\Psi\geq\frac{1}{4\kappa^{4}}\mathbf{I}_{Hd_{x}}$ and $\left\Vert \Theta\right\Vert \leq\gamma^{-1}\kappa^{3}$.
	Therefore, if $\left\Vert \mathbf{I}_{H}\otimes\mathbf{B}\right\Vert \geq\frac{\gamma}{4\kappa^{3}}$,
	then 
	\begin{align}
		& \Omega\geq\frac{1}{4\kappa^{4}}\mathbf{I}_{Hd_{x}}\left(\frac{\gamma}{4\kappa^{3}}\right)^{2}=\frac{\gamma^{2}}{64\kappa^{10}}\mathbf{I}_{Hd_{x}}.\label{eq: Omega term 1}
	\end{align}
	Otherwise, if $\left\Vert \mathbf{I}_{H}\otimes\mathbf{B}\right\Vert <\frac{\gamma}{4\kappa^{3}}$,
	then 
	\begin{align}
		& \Omega\geq\mathbf{I}_{Hd_{u}}-\Theta\left(\mathbf{I}_{H}\otimes\mathbf{B}\right)-\left(\mathbf{I}_{H}\otimes\mathbf{B}^{\top}\right)\Theta^{\top},\\
		& \left\Vert \Omega\right\Vert \geq1-\gamma^{-1}\kappa^{3}\frac{\gamma}{4\kappa^{3}}-\gamma^{-1}\kappa^{3}\frac{\gamma}{4\kappa^{3}}=\frac{1}{2}.\label{eq: Omega term 2}
	\end{align}
	
	Combine (\ref{eq: Omega term 1}) and (\ref{eq: Omega term 2}), we
	have $\left\Vert \Omega\right\Vert \geq\mathrm{min}\left\{ \frac{1}{2},\frac{\gamma^{2}}{64\kappa^{10}}\right\} $.
	As a result, 
	\begin{align}
		& \mathbb{E}_{\mathbf{x}_{0},\left\{ \mathbf{\Delta}_{i}\sim\mathcal{D}_{i}\left(\mathbf{M}^{\prime}\right)\right\} _{0\leq i<t},\left\{ \mathbf{w}_{i}\right\} _{0\leq i<t}}\left[\left\Vert \mathbf{x}_{t}^{\left(\mathbf{M}_{1}\right)}-\mathbf{x}_{t}^{\left(\mathbf{M}_{2}\right)}\right\Vert ^{2}+\left\Vert \mathbf{u}_{t}^{\left(\mathbf{M}_{1}\right)}-\mathbf{u}_{t}^{\left(\mathbf{M}_{2}\right)}\right\Vert ^{2}\right]\label{eq:scvx final term-2}\\
		& =\delta_{\mathbf{M}}^{\top}\left(\Omega\otimes\mathbb{E}\left[\mathbf{w}_{t}\mathbf{w}_{t}^{\top}\right]\right)\delta_{\mathbf{M}}\geq\delta_{\mathbf{M}}\left\Vert \Omega\right\Vert \sigma^{2}\geq\mathrm{min}\left\{ \frac{\sigma^{2}}{2},\frac{\gamma^{2}\sigma^{2}}{64\kappa^{10}}\right\} \left\Vert \mathbf{M}_{1}-\mathbf{M}_{2}\right\Vert _{F}^{2}.\nonumber 
	\end{align}
	
	Substitue (\ref{eq:scvx final term-2}) back into (\ref{eq: scvx ft - ft}),
	it follows that 
	\begin{align}
		& f_{t}\left(\mathbf{M}_{1};\mathbf{M}^{\prime}\right)-f_{t}\left(\mathbf{M}_{1};\mathbf{M}^{\prime}\right)\geq\mathrm{Tr}\left(\nabla^{\top}f_{t}\left(\mathbf{M}_{2};\mathbf{M}^{\prime}\right)\cdot\left(\mathbf{M}_{1}-\mathbf{M}_{2}\right)\right)\label{eq: scvx ft final}\\
		& +\frac{\mathrm{min}\left\{ \frac{\mu\sigma^{2}}{2},\frac{\mu\gamma^{2}\sigma^{2}}{64\kappa^{10}}\right\} }{2}\left\Vert \mathbf{M}_{1}-\mathbf{M}_{2}\right\Vert _{F}^{2}.\nonumber 
	\end{align}
	Therefore, Lemma  {\ref{s_cvx}} is proved.
\end{proof}

\textbf{Properties of Gradient.} The following lemma provides upper
bounds for the norm of gradients $\left\Vert \nabla_{\mathbf{x}}c_{t}\left(\mathbf{x}_{t}^{\left(\mathbf{M}\right)},\mathbf{u}_{t}^{\left(\mathbf{M}\right)}\right)\right\Vert $
and $\left\Vert \nabla_{\mathbf{u}}c_{t}\left(\mathbf{x}_{t}^{\left(\mathbf{M}\right)},\mathbf{u}_{t}^{\left(\mathbf{M}\right)}\right)\right\Vert .$
\begin{lemma}\label{Lemma: gradient ct bound} The gradients of
	the per stage cost $c_{t}\left(\mathbf{x}_{t}^{\left(\mathbf{M}\right)},\mathbf{u}_{t}^{\left(\mathbf{M}\right)}\right),\forall1\leq t\leq T,$
	satifies
	\begin{align}
		& \left\Vert \nabla_{\mathbf{x}}c_{t}\left(\mathbf{x}_{t}^{\left(\mathbf{M}\right)},\mathbf{u}_{t}^{\left(\mathbf{M}\right)}\right)\right\Vert \leq x_{0}\kappa^{2}G\alpha_{t}+\kappa^{2}GW\left(\left\Vert \mathbf{B}\right\Vert HM+1\right)\beta_{t},\label{eq: grad_x ct bound}\\
		& \left\Vert \nabla_{\mathbf{u}}c_{t}\left(\mathbf{x}_{t}^{\left(\mathbf{M}\right)},\mathbf{u}_{t}^{\left(\mathbf{M}\right)}\right)\right\Vert \leq x_{0}\kappa^{3}G\alpha_{t}+\kappa^{3}GW\left(\left\Vert \mathbf{B}\right\Vert HM+1\right)\beta_{t}+GMHW,\label{eq:grad_u ct bound}
	\end{align}
	where $\alpha_{t}=\prod_{i=0}^{t-1}\left(1-\gamma+\kappa^{2}\xi_{i}\right)$
	and $\beta_{t}=\sum_{i=0}^{t-1}\prod_{j=i+1}^{t-1}\left(\mathbf{1}_{j<t}\left(1-\gamma+\kappa^{2}\xi_{j}\right)+\mathbf{1}_{j=t}\right).$
\end{lemma}

\begin{proof}
	We first provide upper bounds for $\left\Vert \mathbf{x}_{t}\right\Vert $
	and $\left\Vert \mathbf{u}_{t}\right\Vert .$ From the evolution of
	$\mathbf{x}_{t}$ in \eqref{eq:x_t representation}, we obtain
	\begin{align}
		& \left\Vert \mathbf{x}_{t}\right\Vert \leq\left\Vert \prod_{i=0}^{t-1}\left(\widetilde{\mathbf{A}}+\mathbf{\Delta}_{i}\right)\mathbf{x}_{0}\right\Vert +\left\Vert \sum_{i=0}^{t-1}\prod_{j=i+1}^{t-1}\left(\mathbf{1}_{j<t}\left(\widetilde{\mathbf{A}}+\mathbf{\Delta}_{j}\right)+\mathbf{1}_{j=t}\mathbf{I}\right)\mathbf{B}\mathbf{M}\left[\mathbf{w}\right]_{i-1}^{H}\right\Vert \label{eq: norm x-t bound}\\
		& +\left\Vert \sum_{i=0}^{t-1}\prod_{j=i+1}^{t-1}\left(\mathbf{1}_{j<t}\left(\widetilde{\mathbf{A}}+\mathbf{\Delta}_{j}\right)+\mathbf{1}_{j=t}\mathbf{I}\right)\mathbf{w}_{i}\right\Vert \leq x_{0}\kappa^{2}\prod_{i=0}^{t-1}\left(1-\gamma+\kappa^{2}\left\Vert \mathbf{\Delta}_{i}\right\Vert \right)\nonumber \\
		& +\kappa^{2}W\left(\left\Vert \mathbf{B}\right\Vert HM+1\right)\sum_{i=0}^{t-1}\prod_{j=i+1}^{t-1}\left\Vert \mathbf{1}_{j<t}\left(1-\gamma+\kappa^{2}\left\Vert \mathbf{\Delta}_{j}\right\Vert \right)+\mathbf{1}_{j=t}\mathbf{I}\right\Vert ,\nonumber \\
		& \left\Vert \mathbf{u}_{t}\right\Vert \leq\kappa\left\Vert \mathbf{x}_{t}\right\Vert +MHW.\label{eq:norm u-t bound}
	\end{align}
	
	Lemma \ref{Lemma: gradient ct bound} follows by noting that $\left\Vert \mathbf{\Delta}_{t}\right\Vert \leq\xi_{t},\forall0\leq t\leq T,\forall\mathbf{M}\in\mathcal{\mathbb{M}},$
	and whenever $\left\Vert \mathbf{x}_{t}\right\Vert ,\left\Vert \mathbf{u}_{t}\right\Vert \leq D$,
	we have $\left\Vert \nabla_{\mathbf{x}}c_{t}\left(\mathbf{x}_{t}^{\left(\mathbf{M}\right)},\mathbf{u}_{t}^{\left(\mathbf{M}\right)}\right)\right\Vert ,\left\Vert \nabla_{\mathbf{u}}c_{t}\left(\mathbf{x}_{t}^{\left(\mathbf{M}\right)},\mathbf{u}_{t}^{\left(\mathbf{M}\right)}\right)\right\Vert \leq GD.$ 
\end{proof}

We next prove that the variance of the gradient $\nabla_{\mathbf{M}}c_{t}\left(\mathbf{x}_{t}^{\left(\mathbf{M}\right)},\mathbf{u}_{t}^{\left(\mathbf{M}\right)}\right)$
is bounded.
\begin{lemma} \label{lemma Gradient-Variance per stage}
	There exists a $\vartheta_{t}>0,\forall1\leq t\leq T,$ such that
	\begin{align}
		& \mathbb{E}_{\mathbf{x}_{0},\left\{ \mathbf{\Delta}_{i}\sim\mathcal{D}_{i}\left(\mathbf{M}\right)\right\} _{0\leq i<t},\left\{ \mathbf{w}_{i}\right\} _{0\leq i<t}}\left[\left\Vert \nabla_{\mathbf{M}}c_{t}\left(\mathbf{x}_{t}^{\left(\mathbf{M}\right)},\mathbf{u}_{t}^{\left(\mathbf{M}\right)}\right)\right.\right\Vert \label{eq:var-grad-M}\\
		& -\left.\left.\mathbb{E}_{\mathbf{x}_{0},\left\{ \mathbf{\Delta}_{i}\sim\mathcal{D}_{i}\left(\mathbf{M}\right)\right\} _{0\leq i<t},\left\{ \mathbf{w}_{i}\right\} _{0\leq i<t}}\left[\nabla_{\mathbf{M}}c_{t}\left(\mathbf{x}_{t}^{\left(\mathbf{M}\right)},\mathbf{u}_{t}^{\left(\mathbf{M}\right)}\right)\right]\right\Vert _{F}^{2}\right]\leq\vartheta_{t}.\nonumber
	\end{align}
\end{lemma}

\begin{proof}
	Based on the expression of $\nabla_{\mathbf{M}}c_{t}\left(\mathbf{x}_{t}^{\left(\mathbf{M}\right)},\mathbf{u}_{t}^{\left(\mathbf{M}\right)}\right)$
	in (\ref{eq: grad M ct}), we have 
	
	\begin{align}
		& \left\Vert \nabla_{\mathbf{M}}c_{t}\left(\mathbf{x}_{t}^{\left(\mathbf{M}\right)},\mathbf{u}_{t}^{\left(\mathbf{M}\right)}\right)\right\Vert\label{eq:grad norm bound}\\
		&\leq\left\Vert \mathbf{B}\right\Vert HW\sum_{i=0}^{t-1}\left\Vert \prod_{j=i+1}^{t-1}\left(\mathbf{1}_{j<t}\left(\widetilde{\mathbf{A}}+\mathbf{\Delta}_{j}\right)+\mathbf{1}_{j=t}\mathbf{I}\right)\right\Vert \left\Vert \nabla_{\mathbf{x}}c_{t}\left(\mathbf{x}_{t}^{\left(\mathbf{M}\right)},\mathbf{u}_{t}^{\left(\mathbf{M}\right)}\right)\right\Vert \nonumber\\
		& +\kappa\left\Vert \mathbf{B}\right\Vert HW\sum_{i=0}^{t-1}\left\Vert \prod_{j=i+1}^{t-1}\left(\mathbf{1}_{j<t}\left(\widetilde{\mathbf{A}}+\mathbf{\Delta}_{j}\right)+\mathbf{1}_{j=t}\mathbf{I}\right)\right\Vert \left\Vert \nabla_{\mathbf{u}}c_{t}\left(\mathbf{x}_{t}^{\left(\mathbf{M}\right)},\mathbf{u}_{t}^{\left(\mathbf{M}\right)}\right)\right\Vert \nonumber \\
		& +HW\left\Vert \nabla_{\mathbf{u}}c_{t}\left(\mathbf{x}_{t}^{\left(\mathbf{M}\right)},\mathbf{u}_{t}^{\left(\mathbf{M}\right)}\right)\right\Vert \nonumber \\
		& \leq\kappa^{2}\left\Vert \mathbf{B}\right\Vert HW\beta_{t}\left\Vert \nabla_{\mathbf{x}}c_{t}\left(\mathbf{x}_{t}^{\left(\mathbf{M}\right)},\mathbf{u}_{t}^{\left(\mathbf{M}\right)}\right)\right\Vert +\left(\kappa^{3}\left\Vert \mathbf{B}\right\Vert \beta_{t}+1\right)HW\left\Vert \nabla_{\mathbf{u}}c_{t}\left(\mathbf{x}_{t}^{\left(\mathbf{M}\right)},\mathbf{u}_{t}^{\left(\mathbf{M}\right)}\right)\right\Vert \nonumber \\
		& \overset{\left(\ref{eq:grad norm bound}.a\right)}{\leq}\left(\kappa\left\Vert \mathbf{B}\right\Vert HW\beta_{t}+\kappa^{3}\left\Vert \mathbf{B}\right\Vert \beta_{t}+1\right)\left(x_{0}\kappa^{3}G\alpha_{t}+\kappa^{3}GW\left(\left\Vert \mathbf{B}\right\Vert HM+1\right)\beta_{t}\right)\nonumber \\
		& +\left(\kappa^{3}\left\Vert \mathbf{B}\right\Vert \beta_{t}+1\right)GMHW,\nonumber 
	\end{align}
	where inequality $\left(\ref{eq:grad norm bound}.a\right)$ holds
	becasue of the gradient norm bounds in Lemma \ref{Lemma: gradient ct bound}.
	
	Further note that 
	\begin{align}
		& \mathbb{E}\left[\left\Vert \nabla_{\mathbf{M}}c_{t}\left(\mathbf{x}_{t}^{\left(\mathbf{M}\right)},\mathbf{u}_{t}^{\left(\mathbf{M}\right)}\right)-\mathbb{E}\left[\nabla_{\mathbf{M}}c_{t}\left(\mathbf{x}_{t}^{\left(\mathbf{M}\right)},\mathbf{u}_{t}^{\left(\mathbf{M}\right)}\right)\right]\right\Vert _{F}^{2}\right]\leq\mathbb{E}\left[\left\Vert \nabla_{\mathbf{M}}c_{t}\left(\mathbf{x}_{t}^{\left(\mathbf{M}\right)},\mathbf{u}_{t}^{\left(\mathbf{M}\right)}\right)\right\Vert _{F}^{2}\right]\label{eq: grad var bound}\\
		& \leq d_{x}^{2}H\mathbb{E}\left[\left\Vert \nabla_{\mathbf{M}}c_{t}\left(\mathbf{x}_{t}^{\left(\mathbf{M}\right)},\mathbf{u}_{t}^{\left(\mathbf{M}\right)}\right)\right\Vert ^{2}\right]\nonumber 
	\end{align}
	
	The desired result follows by substituting (\ref{eq:grad norm bound})
	into (\ref{eq: grad var bound}) and letting 
	\begin{align}
		& \vartheta_{t}=d_{x}^{2}H\left(\left(\kappa\left\Vert \mathbf{B}\right\Vert HW\beta_{t}+\kappa^{3}\left\Vert \mathbf{B}\right\Vert \beta_{t}+1\right)\left(x_{0}\kappa^{3}G\alpha_{t}+\kappa^{3}GW\left(\left\Vert \mathbf{B}\right\Vert HM+1\right)\beta_{t}\right)\right.\\
		& \left.\left.+\left(\kappa^{3}\left\Vert \mathbf{B}\right\Vert \beta_{t}+1\right)GMHW\right)\right).\nonumber
	\end{align}
\end{proof}

The next lemma characterizes the difference between the gradients
under any two different polices $\mathbf{M}$ and $\mathbf{M}^{\prime}$.\medskip{}

\begin{lemma}\label{Lemma: per stage gradient smooth}For
	any given two policies $\mathbf{M},\mathbf{M}^{\prime}\in\mathcal{\mathbb{M}}$,
	the differences between the gradients satify, $\forall1\leq t\leq T,$
	\begin{align}
		& \left\Vert \nabla_{\mathbf{x}}c_{t}\left(\mathbf{x}_{t}^{\left(\mathbf{M}\right)},\mathbf{u}_{t}^{\left(\mathbf{M}\right)}\right)-\nabla_{\mathbf{x}}c_{t}\left(\mathbf{x}_{t}^{\left(\mathbf{M}^{\prime}\right)},\mathbf{u}_{t}^{\left(\mathbf{M}^{\prime}\right)}\right)\right\Vert \nonumber\\
		&+\left\Vert \nabla_{\mathbf{u}}c_{t}\left(\mathbf{x}_{t}^{\left(\mathbf{M}\right)},\mathbf{u}_{t}^{\left(\mathbf{M}\right)}\right)-\nabla_{\mathbf{u}}c_{t}\left(\mathbf{x}_{t}^{\left(\mathbf{M}^{\prime}\right)},\mathbf{u}_{t}^{\left(\mathbf{M}^{\prime}\right)}\right)\right\Vert \\
		& \leq\varsigma\left(\left(1+\kappa\right)\kappa^{2}HW\beta_{t}+HW\right)\left\Vert \text{\ensuremath{\mathbf{M}}}-\mathbf{M}^{\prime}\right\Vert \nonumber \\
		& +\varsigma\left(1+\kappa\right)\kappa^{2}\left(1-\gamma\right)^{-1}\left(x_{0}\alpha_{t}+\left(\left\Vert \mathbf{B}\right\Vert HM+1\right)W\beta_{t}\right)\sum_{i=0}^{t-1}\left\Vert \mathbf{\Delta}_{i}-\mathbf{\Delta}_{i}^{\prime}\right\Vert ,\nonumber 
	\end{align}
	where $\mathbf{\Delta}_{i}$ and $\mathbf{\Delta}_{i}^{\prime},\forall0\leq i<t,$
	denotes the policy-dependent state transition perturbations under
	the policy $\mathbf{M}$ and $\mathbf{M}^{\prime},$ respectively.
\end{lemma}

\begin{proof}
	Recall the smoothness of the cost function $c_{t}\left(\mathbf{x},\mathbf{u}\right)$
	in  {A \ref{A5}}, i.e.,
	\begin{align}
		& \left\Vert \nabla_{\mathbf{x}}c_{t}\left(\mathbf{x}_{1},\mathbf{u}_{1}\right)-\nabla_{\mathbf{x}}c_{t}\left(\mathbf{x}_{2},\mathbf{u}_{2}\right)\right\Vert +\left\Vert \nabla_{\mathbf{u}}c_{t}\left(\mathbf{x}_{1},\mathbf{u}_{1}\right)-\nabla_{\mathbf{u}}c_{t}\left(\mathbf{x}_{2},\mathbf{u}_{2}\right)\right\Vert \label{eq:grad ct-1}\\
		& \leq\varsigma\left(\left\Vert \mathbf{x}_{1}-\mathbf{x}_{2}\right\Vert +\left\Vert \mathbf{u}_{1}-\mathbf{u}_{2}\right\Vert \right).\nonumber 
	\end{align}
	To prove Lemma \ref{Lemma: per stage gradient smooth}, it suffices
	to quantify $\left\Vert \mathbf{x}_{t}^{\left(\mathbf{M}\right)}-\mathbf{x}_{t}^{\left(\mathbf{M}^{\prime}\right)}\right\Vert $
	and $\left\Vert \mathbf{u}_{t}^{\left(\mathbf{M}\right)}-\mathbf{u}_{t}^{\left(\mathbf{M}^{\prime}\right)}\right\Vert $. 
	
	Based on the evolution of the system state $\mathbf{x}_{t}^{\left(\mathbf{M}\right)}$,
	we have 
	\begin{align}
		& \mathbf{x}_{t}^{\left(\mathbf{M}\right)}-\mathbf{x}_{t}^{\left(\mathbf{M}^{\prime}\right)}=E_{1}+E_{2}+E_{3},\label{eq: state diff M M'}
	\end{align}
	where
	\begin{align}
		E_{1}= & \left(\prod_{i=0}^{t-1}\left(\widetilde{\mathbf{A}}+\mathbf{\Delta}_{i}\right)-\prod_{i=0}^{t-1}\left(\widetilde{\mathbf{A}}+\mathbf{\Delta}_{i}^{\prime}\right)\right)\mathbf{x}_{0},\\
		E_{2}= & \sum_{i=0}^{t-1}\prod_{j=i+1}^{t-1}\left(\mathbf{1}_{j<t}\left(\widetilde{\mathbf{A}}+\mathbf{\Delta}_{j}\right)+\mathbf{1}_{j=t}\mathbf{I}\right)\mathbf{B}\mathbf{M}\left[\mathbf{w}\right]_{i-1}^{H}\\
		& -\sum_{i=0}^{t-1}\prod_{j=i+1}^{t-1}\left(\mathbf{1}_{j<t}\left(\widetilde{\mathbf{A}}+\mathbf{\Delta}_{j}^{\prime}\right)+\mathbf{1}_{j=t}\mathbf{I}\right)\mathbf{B}\mathbf{M}^{\prime}\left[\mathbf{w}\right]_{i-1}^{H},\nonumber \\
		E_{3}= & \sum_{i=0}^{t-1}\left(\prod_{j=i+1}^{t-1}\left(\mathbf{1}_{j<t}\left(\widetilde{\mathbf{A}}+\mathbf{\Delta}_{j}\right)+\mathbf{1}_{j=t}\mathbf{I}\right)-\prod_{j=i+1}^{t-1}\left(\mathbf{1}_{j<t}\left(\widetilde{\mathbf{A}}+\mathbf{\Delta}_{j}^{\prime}\right)+\mathbf{1}_{j=t}\mathbf{I}\right)\right)\mathbf{w}_{i}.
	\end{align}
	
	We next analyze each term in (\ref{eq: state diff M M'}) one by one.
	Specifically,
	\begin{align}
		\left\Vert E_{1}\right\Vert  & \leq x_{0}\left\Vert \prod_{i=0}^{t-1}\left(\widetilde{\mathbf{A}}+\mathbf{\Delta}_{i}\right)-\prod_{i=0}^{t-1}\left(\widetilde{\mathbf{A}}+\mathbf{\Delta}_{i}^{\prime}\right)\right\Vert \label{eq:F1 bound}\\
		& \leq x_{0}\left\Vert \left(\mathbf{\Delta}_{0}-\mathbf{\Delta}_{0}^{\prime}\right)\prod_{i=1}^{t-1}\left(\widetilde{\mathbf{A}}+\mathbf{\Delta}_{i}\right)\right.+\left(\widetilde{\mathbf{A}}+\mathbf{\Delta}_{0}^{\prime}\right)\left(\mathbf{\Delta}_{1}-\mathbf{\Delta}_{1}^{\prime}\right)\prod_{i=2}^{t-1}\left(\widetilde{\mathbf{A}}+\mathbf{\Delta}_{i}\right)\nonumber \\
		& +\cdots+\left.\prod_{i=0}^{t-2}\left(\widetilde{\mathbf{A}}+\mathbf{\Delta}_{i}^{\prime}\right)\left(\mathbf{\Delta}_{t-1}-\mathbf{\Delta}_{t-1}^{\prime}\right)\right\Vert \nonumber \\
		& \leq x_{0}\kappa^{2}\left(1-\gamma\right)^{-1}\prod_{i=0}^{t-1}\left(1-\gamma+\kappa^{2}\xi_{i}\right)\sum_{i=0}^{t-1}\left\Vert \mathbf{\Delta}_{i}-\mathbf{\Delta}_{i}^{\prime}\right\Vert \nonumber \\
		& =x_{0}\kappa^{2}\left(1-\gamma\right)^{-1}\alpha_{t}\sum_{i=0}^{t-1}\left\Vert \mathbf{\Delta}_{i}-\mathbf{\Delta}_{i}^{\prime}\right\Vert .\nonumber 
	\end{align}
	
	For the term $E_{2}$, we use the similar approach for analyzing $E_{1}$
	in (\ref{eq:delta-grad-M-part1}) and obtain the following bound
	\begin{align}
		& \left\Vert E_{2}\right\Vert \label{eq: F2 bound}\\
		& \leq\left\Vert \sum_{i=0}^{t-1}\prod_{j=i+1}^{t-1}\left(\mathbf{1}_{j<t}\left(\widetilde{\mathbf{A}}+\mathbf{\Delta}_{j}\right)+\mathbf{1}_{j=t}\mathbf{I}\right)\mathbf{B}\left(\mathbf{M}-\mathbf{M}^{\prime}\right)\left[\mathbf{w}\right]_{i-1}^{H}\right.\nonumber \\
		& +\left.\sum_{i=0}^{t-1}\left(\prod_{j=i+1}^{t-1}\left(\mathbf{1}_{j<t}\left(\widetilde{\mathbf{A}}+\mathbf{\Delta}_{j}\right)+\mathbf{1}_{j=t}\mathbf{I}\right)-\prod_{j=i+1}^{t-1}\left(\mathbf{1}_{j<t}\left(\widetilde{\mathbf{A}}+\mathbf{\Delta}_{j}^{\prime}\right)+\mathbf{1}_{j=t}\mathbf{I}\right)\right)\mathbf{B}\mathbf{M}^{\prime}\left[\mathbf{w}\right]_{i-1}^{H}\right\Vert \nonumber \\
		& \leq\kappa^{2}HW\beta_{t}\left\Vert \text{\ensuremath{\mathbf{M}}}-\mathbf{M}^{\prime}\right\Vert \nonumber \\
		& +\left\Vert \mathbf{B}\right\Vert MHW\sum_{i=0}^{t-1}\left\Vert \prod_{j=i+1}^{t-1}\left(\mathbf{1}_{j<t}\left(\widetilde{\mathbf{A}}+\mathbf{\Delta}_{j}\right)+\mathbf{1}_{j=t}\mathbf{I}\right)-\prod_{j=i+1}^{t-1}\left(\mathbf{1}_{j<t}\left(\widetilde{\mathbf{A}}+\mathbf{\Delta}_{j}^{\prime}\right)+\mathbf{1}_{j=t}\mathbf{I}\right)\right\Vert \nonumber \\
		& \overset{\left(\ref{eq: F2 bound}.a\right)}{\leq}\kappa^{2}HW\beta_{t}\left\Vert \text{\ensuremath{\mathbf{M}}}-\mathbf{M}^{\prime}\right\Vert +\kappa^{2}\left(1-\gamma\right)^{-1}\left\Vert \mathbf{B}\right\Vert MHW\beta_{t}\sum_{i=0}^{t-1}\left\Vert \mathbf{\Delta}_{i}-\mathbf{\Delta}_{i}^{\prime}\right\Vert ,\nonumber 
	\end{align}
	where inequality $\left(\ref{eq: F2 bound}.a\right)$ is obtained
	by using (\ref{eq:delta-grad-M-part 1-(2)}).
	
	Finally, still using (\ref{eq:delta-grad-M-part 1-(2)}), an upper
	bound of $\left\Vert E_{3}\right\Vert $ is obtained as follows
	\begin{align}
		\left\Vert E_{3}\right\Vert \leq & W\sum_{i=0}^{t-1}\left\Vert \prod_{j=i+1}^{t-1}\left(\mathbf{1}_{j<t}\left(\widetilde{\mathbf{A}}+\mathbf{\Delta}_{j}\right)+\mathbf{1}_{j=t}\mathbf{I}\right)-\prod_{j=i+1}^{t-1}\left(\mathbf{1}_{j<t}\left(\widetilde{\mathbf{A}}+\mathbf{\Delta}_{j}^{\prime}\right)+\mathbf{1}_{j=t}\mathbf{I}\right)\right\Vert \label{eq:F3 bound}\\
		\leq & \kappa^{2}\left(1-\gamma\right)^{-1}W\beta_{t}\sum_{i=0}^{t-1}\left\Vert \mathbf{\Delta}_{i}-\mathbf{\Delta}_{i}^{\prime}\right\Vert .\nonumber 
	\end{align}
	
	It follow directly that
	\begin{align}
		& \left\Vert \mathbf{x}_{t}^{\left(\mathbf{M}\right)}-\mathbf{x}_{t}^{\left(\mathbf{M}^{\prime}\right)}\right\Vert \leq\left\Vert E_{1}\right\Vert +\left\Vert E_{2}\right\Vert +\left\Vert E_{3}\right\Vert \\
		& \leq\kappa^{2}HW\beta_{t}\left\Vert \text{\ensuremath{\mathbf{M}}}-\mathbf{M}^{\prime}\right\Vert +\kappa^{2}\left(1-\gamma\right)^{-1}\left(x_{0}\alpha_{t}+\left(\left\Vert \mathbf{B}\right\Vert HM+1\right)W\beta_{t}\right)\sum_{i=0}^{t-1}\left\Vert \mathbf{\Delta}_{i}-\mathbf{\Delta}_{i}^{\prime}\right\Vert .\nonumber 
	\end{align}
	
	We next observe the relationship
	\begin{align}
		& \left\Vert \mathbf{u}_{t}^{\left(\mathbf{M}\right)}-\mathbf{u}_{t}^{\left(\mathbf{M}^{\prime}\right)}\right\Vert =\left\Vert -\mathbf{K}\mathbf{x}_{t}^{\left(\mathbf{M}\right)}+\mathbf{K}\mathbf{x}_{t}^{\left(\mathbf{M}^{\prime}\right)}+\mathbf{M}\left[\mathbf{w}\right]_{t-1}^{H}-\mathbf{M}^{\prime}\left[\mathbf{w}\right]_{t-1}^{H}\right\Vert \\
		& \leq\kappa\left\Vert \mathbf{x}_{t}^{\left(\mathbf{M}\right)}-\mathbf{x}_{t}^{\left(\mathbf{M}^{\prime}\right)}\right\Vert +HW\left\Vert \text{\ensuremath{\mathbf{M}}}-\mathbf{M}^{\prime}\right\Vert .\nonumber 
	\end{align}
	Therefore, it follows that
	\begin{align}
		& \left\Vert \mathbf{x}_{t}^{\left(\mathbf{M}\right)}-\mathbf{x}_{t}^{\left(\mathbf{M}^{\prime}\right)}\right\Vert +\left\Vert \mathbf{u}_{t}^{\left(\mathbf{M}\right)}-\mathbf{u}_{t}^{\left(\mathbf{M}^{\prime}\right)}\right\Vert \leq\left(1+\kappa\right)\left\Vert \mathbf{x}_{t}^{\left(\mathbf{M}\right)}-\mathbf{x}_{t}^{\left(\mathbf{M}^{\prime}\right)}\right\Vert +HW\left\Vert \text{\ensuremath{\mathbf{M}}}-\mathbf{M}^{\prime}\right\Vert \label{eq: xm-xm' + um-um' bound}\\
		& \leq\left(\left(1+\kappa\right)\kappa^{2}HW\beta_{t}+HW\right)\left\Vert \text{\ensuremath{\mathbf{M}}}-\mathbf{M}^{\prime}\right\Vert \nonumber \\
		& +\left(1+\kappa\right)\kappa^{2}\left(1-\gamma\right)^{-1}\left(x_{0}\alpha_{t}+\left(\left\Vert \mathbf{B}\right\Vert HM+1\right)W\beta_{t}\right)\sum_{i=0}^{t-1}\left\Vert \mathbf{\Delta}_{i}-\mathbf{\Delta}_{i}^{\prime}\right\Vert .\nonumber 
	\end{align}
	
	Combining (\ref{eq: xm-xm' + um-um' bound}) and (\ref{eq:grad ct-1}),
	we obatin the desired result.\medskip{}
\end{proof}

\textbf{Smoothness}: We analyze the smoothness of the per stage cost
and the per stage expected cost in the following Lemma \ref{Lemma: Smoothness per Stage Cost}
and \ref{Lemma: Smoothness Average per Stage Cost}, respectively.

\begin{lemma}
	 \label{Lemma: Smoothness per Stage Cost}
	The per stage cost function $c_{t}\left(\mathbf{M};\left\{ \mathbf{\Delta}_{i}\right\} _{0\leq i<t}\right),\forall1\leq t\leq T,$
	is smooth in the sense that
	\begin{align}
		& \left\Vert \nabla_{\mathbf{M}}c_{t}\left(\mathbf{M};\left\{ \mathbf{\Delta}_{i}\right\} _{0\leq i<t}\right)-\nabla_{\mathbf{M}^{\prime}}c_{t}\left(\mathbf{M}^{\prime};\left\{ \mathbf{\Delta}_{i}^{\prime}\right\} _{0\leq i<t}\right)\right\Vert _{F}\\
		&\leq\lambda_{t}\left\Vert \mathbf{M}-\mathbf{M}^{\prime}\right\Vert _{F}+\nu_{t}\sum_{i=0}^{t-1}\left\Vert \mathbf{\Delta}_{i}-\mathbf{\Delta}_{i}^{\prime}\right\Vert _{F},\nonumber
	\end{align}
	where 
	\begin{align}
		\lambda_{t}= & d_{x}\sqrt{H}\varsigma H^{2}W^{2}\left(1+\left\Vert \mathbf{B}\right\Vert \left(\kappa^{2}+\kappa^{3}\right)\right)\left(\left(\kappa^{2}+\kappa^{3}\right)\beta_{t}+1\right),\\
		\nu_{t}= & d_{x}\sqrt{H}\left(\varsigma HW\left(1+\left\Vert \mathbf{B}\right\Vert \left(\kappa^{2}+\kappa^{3}\right)\right)\left(\kappa^{2}+\kappa^{3}\right)\left(1-\gamma\right)^{-1}+G\left(1-\gamma\right)^{-1}\right.\\
		& \cdot\left.\left(\kappa^{4}+\kappa^{5}\right)HW\left\Vert \mathbf{B}\right\Vert \beta_{t}\right)\left(x_{0}\alpha_{t}+\left(\left\Vert \mathbf{B}\right\Vert HM+1\right)W\beta_{t}\right),\nonumber 
	\end{align}
	and $\mathbf{\Delta}_{i}$ and $\mathbf{\Delta}_{i}^{\prime},\forall0\leq i<t,$
	denotes the policy-dependent state transition perturbations under
	the policy $\mathbf{M}$ and $\mathbf{M}^{\prime},$ respectively.
\end{lemma}

\begin{proof}
	We prove this lemma based on the smoothness of the cost function $c_{t}\left(\mathbf{x},\mathbf{u}\right)$
	in  {A \ref{A5}}, i.e., 
	\begin{align}
		& \left\Vert \nabla_{\mathbf{x}}c_{t}\left(\mathbf{x}_{1},\mathbf{u}_{1}\right)-\nabla_{\mathbf{x}}c_{t}\left(\mathbf{x}_{2},\mathbf{u}_{1}\right)\right\Vert +\left\Vert \nabla_{\mathbf{u}}c_{t}\left(\mathbf{x}_{1},\mathbf{u}_{1}\right)-\nabla_{\mathbf{u}}c_{t}\left(\mathbf{x}_{1},\mathbf{u}_{2}\right)\right\Vert \nonumber\\
		&\leq\varsigma\left(\left\Vert \mathbf{x}_{1}-\mathbf{x}_{2}\right\Vert +\left\Vert \mathbf{u}_{1}-\mathbf{u}_{2}\right\Vert \right).\label{eq:grad ct}
	\end{align}
	
	Based on the expression of $\nabla_{\mathbf{M}}c_{t}\left(\mathbf{M};\left\{ \mathbf{\Delta}_{i}\right\} _{0\leq i<t}\right)$
	in (\ref{eq: grad M ct}), it follows that 
	\begin{align}
		& \nabla_{\mathbf{M}}c_{t}\left(\mathbf{M};\left\{ \mathbf{\Delta}_{i}\right\} _{0\leq i<t}\right)-\nabla_{\mathbf{M}^{\prime}}c_{t}\left(\mathbf{M}^{\prime};\left\{ \mathbf{\Delta}_{i}^{\prime}\right\} _{0\leq i<t}\right)\leq F_{1}+F_{2}+F_{3},\label{eq:delta-gradient M}
	\end{align}
	where 
	\begin{align}
		F_{1}= & \sum_{i=0}^{t-1}\left(\left(\prod_{j=i+1}^{t-1}\left(\mathbf{1}_{j<t}\left(\widetilde{\mathbf{A}}+\mathbf{\Delta}_{j}\right)+\mathbf{1}_{j=t}\mathbf{I}\right)\mathbf{B}\right)^{\top}\nabla_{\mathbf{x}}c_{t}\left(\mathbf{x}_{t}^{\left(\mathbf{M}\right)},\mathbf{u}_{t}^{\left(\mathbf{M}\right)}\right)\right.\\
		& -\left.\left(\prod_{j=i+1}^{t-1}\left(\mathbf{1}_{j<t}\left(\widetilde{\mathbf{A}}+\mathbf{\Delta}_{j}^{\prime}\right)+\mathbf{1}_{j=t}\mathbf{I}\right)\mathbf{B}\right)^{\top}\nabla_{\mathbf{x}}c_{t}\left(\mathbf{x}_{t}^{\left(\mathbf{M}^{\prime}\right)},\mathbf{u}_{t}^{\left(\mathbf{M}^{\prime}\right)}\right)\right)\left(\left[\mathbf{w}\right]_{i-1}^{H}\right)^{\top},\nonumber \\
		F_{2}= & \sum_{i=0}^{t-1}\left(\left(\mathbf{K}\prod_{j=i+1}^{t-1}\left(\mathbf{1}_{j<t}\left(\widetilde{\mathbf{A}}+\mathbf{\Delta}_{j}\right)+\mathbf{1}_{j=t}\mathbf{I}\right)\mathbf{B}\right)^{\top}\nabla_{\mathbf{x}}c_{t}\left(\mathbf{x}_{t}^{\left(\mathbf{M}\right)},\mathbf{u}_{t}^{\left(\mathbf{M}\right)}\right)\right.\\
		& -\left.\left(\mathbf{K}\prod_{j=i+1}^{t-1}\left(\mathbf{1}_{j<t}\left(\widetilde{\mathbf{A}}+\mathbf{\Delta}_{j}^{\prime}\right)+\mathbf{1}_{j=t}\mathbf{I}\right)\mathbf{B}\right)^{\top}\nabla_{\mathbf{x}}c_{t}\left(\mathbf{x}_{t}^{\left(\mathbf{M}^{\prime}\right)},\mathbf{u}_{t}^{\left(\mathbf{M}^{\prime}\right)}\right)\right)\left(\left[\mathbf{w}\right]_{i-1}^{H}\right)^{\top},\nonumber \\
		F_{3}= & \left(\nabla_{\mathbf{u}}c_{t}\left(\mathbf{x}_{t}^{\left(\mathbf{M}\right)},\mathbf{u}_{t}^{\left(\mathbf{M}\right)}\right)-\nabla_{\mathbf{u}}c_{t}\left(\mathbf{x}_{t}^{\left(\mathbf{M}^{\prime}\right)},\mathbf{u}_{t}^{\left(\mathbf{M}^{\prime}\right)}\right)\right)\left(\left[\mathbf{w}\right]_{t-1}^{H}\right)^{\top}.
	\end{align}
	
	We next analyze each term in (\ref{eq:delta-gradient M}) one by one.
	Specifically,
	\begin{align}
		& \left\Vert F_{1}\right\Vert =\left\Vert \sum_{i=0}^{t-1}\left(\left(\prod_{j=i+1}^{t-1}\left(\mathbf{1}_{j<t}\left(\widetilde{\mathbf{A}}+\mathbf{\Delta}_{j}\right)+\mathbf{1}_{j=t}\mathbf{I}\right)\mathbf{B}\right)^{\top}\nabla_{\mathbf{x}}c_{t}\left(\mathbf{x}_{t}^{\left(\mathbf{M}\right)},\mathbf{u}_{t}^{\left(\mathbf{M}\right)}\right)\right.\right.\label{eq:delta-grad-M-part1}\\
		& -\left(\prod_{j=i+1}^{t-1}\left(\mathbf{1}_{j<t}\left(\widetilde{\mathbf{A}}+\mathbf{\Delta}_{j}\right)+\mathbf{1}_{j=t}\mathbf{I}\right)\mathbf{B}\right)^{\top}\nabla_{\mathbf{x}}c_{t}\left(\mathbf{x}_{t}^{\left(\mathbf{M}^{\prime}\right)},\mathbf{u}_{t}^{\left(\mathbf{M}^{\prime}\right)}\right)\nonumber \\
		& +\left(\prod_{j=i+1}^{t-1}\left(\mathbf{1}_{j<t}\left(\widetilde{\mathbf{A}}+\mathbf{\Delta}_{j}\right)+\mathbf{1}_{j=t}\mathbf{I}\right)\mathbf{B}\right)^{\top}\nabla_{\mathbf{x}}c_{t}\left(\mathbf{x}_{t}^{\left(\mathbf{M}^{\prime}\right)},\mathbf{u}_{t}^{\left(\mathbf{M}^{\prime}\right)}\right)\nonumber \\
		& -\left.\left.\left(\prod_{j=i+1}^{t-1}\left(\mathbf{1}_{j<t}\left(\widetilde{\mathbf{A}}+\mathbf{\Delta}_{j}^{\prime}\right)+\mathbf{1}_{j=t}\mathbf{I}\right)\mathbf{B}\right)^{\top}\nabla_{\mathbf{x}}c_{t}\left(\mathbf{x}_{t}^{\left(\mathbf{M}^{\prime}\right)},\mathbf{u}_{t}^{\left(\mathbf{M}^{\prime}\right)}\right)\right)\left(\left[\mathbf{w}\right]_{i-1}^{H}\right)^{\top}\right\Vert \nonumber \\
		& \overset{\left(\ref{eq:delta-grad-M-part1}.a\right)}{\leq}HW\left\Vert \mathbf{B}\right\Vert \kappa^{2}\beta_{t}\left\Vert \nabla_{\mathbf{x}}c_{t}\left(\mathbf{x}_{t}^{\left(\mathbf{M}\right)},\mathbf{u}_{t}^{\left(\mathbf{M}\right)}\right)-\nabla_{\mathbf{x}}c_{t}\left(\mathbf{x}_{t}^{\left(\mathbf{M}^{\prime}\right)},\mathbf{u}_{t}^{\left(\mathbf{M}^{\prime}\right)}\right)\right\Vert \nonumber \\
		& +HW\left\Vert \mathbf{B}\right\Vert \kappa^{2}\left(x_{0}G\alpha_{t}+GW\left(\left\Vert \mathbf{B}\right\Vert HM+1\right)\beta_{t}\right)\nonumber\\
		& \cdot\sum_{i=0}^{t-1}\left\Vert \prod_{j=i+1}^{t-1}\left(\mathbf{1}_{j<t}\left(\widetilde{\mathbf{A}}+\mathbf{\Delta}_{j}\right)+\mathbf{1}_{j=t}\mathbf{I}\right)-\prod_{j=i+1}^{t-1}\left(\mathbf{1}_{j<t}\left(\widetilde{\mathbf{A}}+\mathbf{\Delta}_{j}^{\prime}\right)+\mathbf{1}_{j=t}\mathbf{I}\right)\right\Vert ,\nonumber 
	\end{align}
	where inequality $\left(\ref{eq:delta-grad-M-part1}.a\right)$ holds
	due to the definition of $\beta_{t}$ and the upper bound of $\left\Vert \nabla_{\mathbf{x}}c_{t}\left(\mathbf{x}_{t}^{\left(\mathbf{M}^{\prime}\right)},\mathbf{u}_{t}^{\left(\mathbf{M}^{\prime}\right)}\right)\right\Vert $
	provided in Lemma \ref{Lemma: gradient ct bound}.
	
	For the last term in (\ref{eq:delta-grad-M-part1}), we first observe
	that 
	\begin{align}
		& \left\Vert \prod_{j=i+1}^{t-1}\left(\mathbf{1}_{j<t}\left(\widetilde{\mathbf{A}}+\mathbf{\Delta}_{j}^{\prime}\right)+\mathbf{1}_{j=t}\mathbf{I}\right)-\prod_{j=i+1}^{t-1}\left(\mathbf{1}_{j<t}\left(\widetilde{\mathbf{A}}+\mathbf{\Delta}_{j}\right)+\mathbf{1}_{j=t}\mathbf{I}\right)\right\Vert \label{eq:delta-grad-M-part 1-(1)}\\
		& =\left\Vert \left(\mathbf{\Delta}_{i+1}^{\prime}-\mathbf{\Delta}_{i+1}\right)\prod_{j=i+2}^{t-1}\left(\mathbf{1}_{j<t}\left(\widetilde{\mathbf{A}}+\mathbf{\Delta}_{j}^{\prime}\right)+\mathbf{1}_{j=t}\mathbf{I}\right)\right.\nonumber \\
		& +\left(\widetilde{\mathbf{A}}+\mathbf{\Delta}_{i+1}\right)\left(\mathbf{\Delta}_{i+2}^{\prime}-\mathbf{\Delta}_{i+2}\right)\prod_{j=i+3}^{t-1}\left(\mathbf{1}_{j<t}\left(\widetilde{\mathbf{A}}+\mathbf{\Delta}_{j}^{\prime}\right)+\mathbf{1}_{j=t}\mathbf{I}\right)+\cdots\nonumber \\
		& +\prod_{k=i+1}^{i+1+l}\left(\mathbf{1}_{k<t}\left(\widetilde{\mathbf{A}}+\mathbf{\Delta}_{k}\right)+\mathbf{1}_{k=t}\mathbf{I}\right)\left(\mathbf{\Delta}_{i+1+l+1}^{\prime}-\mathbf{\Delta}_{i+1+l+1}\right)\prod_{j=i+1+l+2}^{t-1}\left(\mathbf{1}_{j<t}\left(\widetilde{\mathbf{A}}+\mathbf{\Delta}_{j}^{\prime}\right)+\mathbf{1}_{j=t}\mathbf{I}\right)\nonumber \\
		& +\cdots+\left.\prod_{k=i+1}^{t-2}\left(\mathbf{1}_{k<t}\left(\widetilde{\mathbf{A}}+\mathbf{\Delta}_{k}\right)+\mathbf{1}_{k=t}\mathbf{I}\right)\left(\mathbf{\Delta}_{t-1}^{\prime}-\mathbf{\Delta}_{t-1}\right)\right\Vert \nonumber \\
		& \leq\left(\prod_{j=i+1}^{t-1}\left(\mathbf{1}_{j<t}\left(1-\gamma+\kappa^{2}\xi_{j}\right)+\mathbf{1}_{j=t}\right)\right)\left(\kappa^{2}\left\Vert \mathbf{\Delta}_{i+1}^{\prime}-\mathbf{\Delta}_{i+1}\right\Vert \left(1-\gamma+\kappa^{2}\xi_{i+1}\right)^{-1}\right.\nonumber \\
		& +\cdots+\kappa^{2}\left\Vert \left(\mathbf{\Delta}_{i+1+l+1}^{\prime}-\mathbf{\Delta}_{i+1+l+1}\right)\right\Vert \left(1-\gamma+\kappa^{2}\xi_{i+1+l+1}\right)^{-1}+\cdots\nonumber \\
		& +\left.\kappa^{2}\left\Vert \mathbf{\Delta}_{t-1}^{\prime}-\mathbf{\Delta}_{t-1}\right\Vert \left(1-\gamma+\kappa^{2}\xi_{t-1}\right)^{-1}\right)\nonumber \\
		& \leq\left(\prod_{j=i+1}^{t-1}\left(\mathbf{1}_{j<t}\left(1-\gamma+\kappa^{2}\xi_{j}\right)+\mathbf{1}_{j=t}\right)\right)\sum_{k=i+1}^{t-1}\kappa^{2}\left(1-\gamma\right)^{-1}\left\Vert \mathbf{\Delta}_{k}^{\prime}-\mathbf{\Delta}_{k}\right\Vert \nonumber \\
		& \leq\left(\prod_{j=i+1}^{t-1}\left(\mathbf{1}_{j<t}\left(1-\gamma+\kappa^{2}\xi_{j}\right)+\mathbf{1}_{j=t}\right)\right)\sum_{k=0}^{t-1}\kappa^{2}\left(1-\gamma\right)^{-1}\left\Vert \mathbf{\Delta}_{k}^{\prime}-\mathbf{\Delta}_{k}\right\Vert .\nonumber 
	\end{align}

As a result,
\begin{align}
	& \sum_{i=0}^{t-1}\left\Vert \prod_{j=i+1}^{t-1}\left(\mathbf{1}_{j<t}\left(\widetilde{\mathbf{A}}+\mathbf{\Delta}_{j}^{\prime}\right)+\mathbf{1}_{j=t}\mathbf{I}\right)-\prod_{j=i+1}^{t-1}\left(\mathbf{1}_{j<t}\left(\widetilde{\mathbf{A}}+\mathbf{\Delta}_{j}\right)+\mathbf{1}_{j=t}\mathbf{I}\right)\right\Vert \label{eq:delta-grad-M-part 1-(2)}\\
	& \leq\sum_{i=0}^{t-1}\prod_{j=i+1}^{t-1}\left(\mathbf{1}_{j<t}\left(1-\gamma+\kappa^{2}\xi_{j}\right)+\mathbf{1}_{j=t}\right)\left(\sum_{k=0}^{t-1}\kappa^{2}\left(1-\gamma\right)^{-1}\left\Vert \mathbf{\Delta}_{k}^{\prime}-\mathbf{\Delta}_{k}\right\Vert \right)\nonumber \\
	& =\kappa^{2}\left(1-\gamma\right)^{-1}\beta_{t}\sum_{i=0}^{t-1}\left\Vert \mathbf{\Delta}_{i}^{\prime}-\mathbf{\Delta}_{i}\right\Vert .\nonumber 
\end{align}

Substitue (\ref{eq:delta-grad-M-part 1-(2)}) into (\ref{eq:delta-grad-M-part1}),
it follows
\begin{align}
	\left\Vert E_{1}\right\Vert \leq & HW\left\Vert \mathbf{B}\right\Vert \kappa^{2}\beta_{t}\left(\left\Vert \nabla_{\mathbf{x}}c_{t}\left(\mathbf{x}_{t}^{\left(\mathbf{M}\right)},\mathbf{u}_{t}^{\left(\mathbf{M}\right)}\right)-\nabla_{\mathbf{x}}c_{t}\left(\mathbf{x}_{t}^{\left(\mathbf{M}^{\prime}\right)},\mathbf{u}_{t}^{\left(\mathbf{M}^{\prime}\right)}\right)\right\Vert \right.\label{eq:delta-grad-M-part 1 full}\\
	& +\left.\left(x_{0}G\alpha_{t}+GW\left(\left\Vert \mathbf{B}\right\Vert HM+1\right)\beta_{t}\right)\kappa^{2}\left(1-\gamma\right)^{-1}\sum_{i=0}^{t-1}\left\Vert \mathbf{\Delta}_{i}^{\prime}-\mathbf{\Delta}_{i}\right\Vert \right)\nonumber 
\end{align}

The upper bounds for $\left\Vert E_{2}\right\Vert $ and $\left\Vert E_{3}\right\Vert $
can be readily obtained as follows 
\begin{align}
	& \left\Vert E_{2}\right\Vert \leq\kappa HW\left\Vert \mathbf{B}\right\Vert \kappa^{2}\beta_{t}\left(\left\Vert \nabla_{\mathbf{x}}c_{t}\left(\mathbf{x}_{t}^{\left(\mathbf{M}\right)},\mathbf{u}_{t}^{\left(\mathbf{M}\right)}\right)-\nabla_{\mathbf{x}}c_{t}\left(\mathbf{x}_{t}^{\left(\mathbf{M}^{\prime}\right)},\mathbf{u}_{t}^{\left(\mathbf{M}^{\prime}\right)}\right)\right\Vert \right.\label{eq:delta-grad-M-part2}\\
	& +\left.\left(x_{0}G\alpha_{t}+GW\left(\left\Vert \mathbf{B}\right\Vert HM+1\right)\beta_{t}\right)\kappa^{2}\left(1-\gamma\right)^{-1}\sum_{i=0}^{t-1}\left\Vert \mathbf{\Delta}_{i}^{\prime}-\mathbf{\Delta}_{i}\right\Vert \right)\nonumber 
\end{align}
and
\begin{align}
	& \left\Vert E_{3}\right\Vert \leq HW\left\Vert \nabla_{\mathbf{u}}c_{t}\left(\mathbf{x}_{t}^{\left(\mathbf{M}\right)},\mathbf{u}_{t}^{\left(\mathbf{M}\right)}\right)-\nabla_{\mathbf{u}}c_{t}\left(\mathbf{x}_{t}^{\left(\mathbf{M}^{\prime}\right)},\mathbf{u}_{t}^{\left(\mathbf{M}^{\prime}\right)}\right)\right\Vert .\label{eq:delta-grad-M-part3}
\end{align}

Substitute (\ref{eq:delta-grad-M-part 1 full}), (\ref{eq:delta-grad-M-part2})
and (\ref{eq:delta-grad-M-part3}) into (\ref{eq:delta-gradient M}),
we have 
\begin{align}
	& \left\Vert \nabla_{\mathbf{M}}c_{t}\left(\mathbf{M};\left\{ \mathbf{\Delta}_{i}\right\} _{0\leq i<t}\right)-\nabla_{\mathbf{M}^{\prime}}c_{t}\left(\mathbf{M}^{\prime};\left\{ \mathbf{\Delta}_{i}^{\prime}\right\} _{0\leq i<t}\right)\right\Vert \label{eq:delta grad-M-final}\\
	& \leq\left(1+\left\Vert \mathbf{B}\right\Vert \left(\kappa^{2}+\kappa^{3}\right)\right)HW\left(\left\Vert \nabla_{\mathbf{x}}c_{t}\left(\mathbf{x}_{t}^{\left(\mathbf{M}\right)},\mathbf{u}_{t}^{\left(\mathbf{M}\right)}\right)-\nabla_{\mathbf{x}}c_{t}\left(\mathbf{x}_{t}^{\left(\mathbf{M}^{\prime}\right)},\mathbf{u}_{t}^{\left(\mathbf{M}^{\prime}\right)}\right)\right\Vert \right.\nonumber \\
	& +\left.\left\Vert \nabla_{\mathbf{u}}c_{t}\left(\mathbf{x}_{t},\mathbf{u}_{t}^{\left(\mathbf{M}\right)}\right)-\nabla_{\mathbf{u}}c_{t}\left(\mathbf{x}_{t},\mathbf{u}_{t}^{\left(\mathbf{M}^{\prime}\right)}\right)\right\Vert \right)+\widetilde{\beta}_{t}\sum_{i=0}^{t-1}\left\Vert \mathbf{\Delta}_{i}^{\prime}-\mathbf{\Delta}_{i}\right\Vert ,\nonumber 
\end{align}
where $\widetilde{\beta}_{t}=\left(1-\gamma\right)^{-1}\left(\kappa^{4}+\kappa^{5}\right)HW\left\Vert \mathbf{B}\right\Vert \beta_{t}\left(x_{0}G\alpha_{t}+GW\left(\left\Vert \mathbf{B}\right\Vert HM+1\right)\beta_{t}\right).$

Applying Lemma \ref{Lemma: per stage gradient smooth} to (\ref{eq:delta grad-M-final})
leads to the desired result.\medskip{}
\end{proof}

Before we proceed to analyze the gradient difference of the expected
per stage cost, we cite the following tool lemma in \cite{perdomo2020performative}.

\begin{lemma}
	(Kantorovich-Rubinstein (Lemma D.3 in \cite{perdomo2020performative})) \label{Lemma: Kantorovich-Rubinstein } A
	distribution map $\mathcal{D}\left(\cdot\right)$ is $\varepsilon$-sensitive
	if and only if for all $\mathbf{M},$$\mathbf{M}^{\prime}\in\mathcal{\mathbb{M}}:$
	\begin{align}
		& \sup\left\{ \mathbb{E}_{\mathbf{A}\sim\mathcal{D}\left(\mathbf{M}\right)}\left[g\left(\mathbf{A}\right)\right]-\mathbb{E}_{\mathbf{A}^{\prime}\sim\mathcal{D}\left(\mathbf{M}^{\prime}\right)}\left[g\left(\mathbf{A}^{\prime}\right)\right]:g:\mathbb{R}^{d_{x}\times d_{x}}\rightarrow\mathbb{R},g\ 1-Lipschitz\right\}\nonumber\\
		& \leq\varepsilon\left\Vert \mathbf{M}-\mathbf{M}^{\prime}\right\Vert _{F}.
	\end{align}
\end{lemma}

The smoothness regarding the expected per stage cost is summarized
below.
\begin{lemma}
	\label{Lemma: Smoothness Average per Stage Cost}
	For any $\mathbf{M},\mathbf{M}^{\prime},\mathbf{M}_{1},\mathbf{M}_{2}\in\mathcal{\mathbb{M}}$
	and $\forall1\leq t\leq T,$ the following inequality holds
	\begin{align}
		& \left\Vert \mathbb{E}_{\mathbf{x}_{0},\left\{ \mathbf{\Delta}_{i}\sim\mathcal{D}_{i}\left(\mathbf{M}_{1}\right)\right\} _{0\leq i<t},\left\{ \mathbf{w}_{i}\right\} _{0\leq i<t}}\left[\nabla_{\mathbf{M}}c_{t}\left(\mathbf{M};\left\{ \mathbf{\Delta}_{i}\right\} _{0\leq i<t}\right)\right]\right.\label{eq:average grad diff2}\\
		& \left.-\mathbb{E}_{\mathbf{x}_{0},\left\{ \mathbf{\Delta}_{i}^{\prime}\sim\mathcal{D}_{i}\left(\mathbf{M}_{2}\right)\right\} _{0\leq i<t},\left\{ \mathbf{w}_{i}\right\} _{0\leq i<t}}\left[\nabla_{\mathbf{M}^{\prime}}c_{t}\left(\mathbf{M}^{\prime};\left\{ \mathbf{\Delta}_{i}^{\prime}\right\} _{0\leq i<t}\right)\right]\right\Vert _{F}\nonumber \\
		& \leq\lambda_{t}\left\Vert \mathbf{M}-\mathbf{M}^{\prime}\right\Vert _{F}+\nu_{t}\sum_{i=0}^{t-1}\varepsilon_{i}\left\Vert \mathbf{M}_{1}-\mathbf{M}_{2}\right\Vert _{F}.\nonumber 
	\end{align}
\end{lemma}

\begin{proof}
	The norm of the gradient difference in (\ref{eq:average grad diff2})
	can be expanded as 
	\begin{align}
		& \left\Vert \mathbb{E}\left[\nabla_{\mathbf{M}}c_{t}\left(\mathbf{M};\left\{ \mathbf{\Delta}_{i}\right\} _{0\leq i<t}\right)\right]-\nabla_{\mathbf{M}^{\prime}}c_{t}\left(\mathbf{M}^{\prime};\left\{ \mathbf{\Delta}_{i}^{\prime}\right\} _{0\leq i<t}\right)\right\Vert _{F}\label{eq: expanded average grad diff norm}\\
		& \leq\left\Vert \mathbb{E}\left[\nabla_{\mathbf{M}}c_{t}\left(\mathbf{M};\left\{ \mathbf{\Delta}_{i}\right\} _{0\leq i<t}\right)-\nabla_{\mathbf{M}^{\prime}}c_{t}\left(\mathbf{M}^{\prime},\left\{ \mathbf{\Delta}_{i}\right\} _{0\leq i<t}\right)\right]\right\Vert _{F}\nonumber \\
		& +\left\Vert \mathbb{E}\left[\nabla_{\mathbf{M}^{\prime}}c_{t}\left(\mathbf{M}^{\prime},\left\{ \mathbf{\Delta}_{i}\right\} _{0\leq i<t}\right)-\nabla_{\mathbf{M}^{\prime}}c_{t}\left(\mathbf{M}^{\prime},\mathbf{\Delta}_{1}^{\prime},\left\{ \mathbf{\Delta}_{i}\right\} _{1\leq i<t}\right)\right]\right\Vert _{F}+\cdots\nonumber \\
		& +\left\Vert \mathbb{E}\left[\nabla_{\mathbf{\mathbf{M}^{\prime}}}c_{t}\left(\mathbf{M}^{\prime},\left\{ \mathbf{\Delta}_{j}^{\prime}\right\} _{0\leq j\leq k-1},\left\{ \mathbf{\Delta}_{i}\right\} _{k\leq i<t}\right)-\nabla_{\mathbf{M}^{\prime}}c_{t}\left(\mathbf{M}^{\prime},\left\{ \mathbf{\Delta}_{j}^{\prime}\right\} _{0\leq j\leq k},\left\{ \mathbf{\Delta}_{i}\right\} _{k+1\leq i<t}\right)\right]\right\Vert _{F}\nonumber \\
		& +\cdots+\left\Vert \mathbb{E}\left[\nabla_{\mathbf{M}^{\prime}}c_{t}\left(\mathbf{M}^{\prime},\left\{ \mathbf{\Delta}_{i}^{\prime}\right\} _{0\leq i<t-1},\mathbf{\Delta}_{t-1}\right)-\nabla_{\mathbf{M}^{\prime}}c_{t}\left(\mathbf{M}^{\prime},\left\{ \mathbf{\Delta}_{i}^{\prime}\right\} _{0\leq i<t-1}\right)\right]\right\Vert _{F},\nonumber 
	\end{align}
	where we drop the subscript in $\mathbb{E}\left[\cdot\right]$ since
	the associated randomness is clear from the context.
	
	We next analyze a general term in R.H.S. of inequality (\ref{eq: expanded average grad diff norm}).
	Specifically, 
	\begin{align}
		& \left\Vert \mathbb{E}\left[\nabla_{\mathbf{\mathbf{M}^{\prime}}}c_{t}\left(\mathbf{M}^{\prime},\left\{ \mathbf{\Delta}_{j}^{\prime}\right\} _{0\leq j\leq k-1},\left\{ \mathbf{\Delta}_{i}\right\} _{k\leq i<t}\right)-\nabla_{\mathbf{M}^{\prime}}c_{t}\left(\mathbf{M}^{\prime},\left\{ \mathbf{\Delta}_{j}^{\prime}\right\} _{0\leq j\leq k},\left\{ \mathbf{\Delta}_{i}\right\} _{k+1\leq i<t}\right)\right]\right\Vert _{F}\\
		& =\left\Vert \mathbb{E}\left[\mathbb{E}\left[\nabla_{\mathbf{\mathbf{M}^{\prime}}}c_{t}\left(\mathbf{M}^{\prime},\left\{ \mathbf{\Delta}_{j}^{\prime}\right\} _{0\leq j\leq k-1},\mathbf{\Delta}_{k},\left\{ \mathbf{\Delta}_{i}\right\} _{k+1\leq i<t}\right)\right.\right.\right.\nonumber \\
		& \left.\left.\left.\left.-\nabla_{\mathbf{M}^{\prime}}c_{t}\left(\mathbf{M}^{\prime},\left\{ \mathbf{\Delta}_{j}^{\prime}\right\} _{0\leq j\leq k-1},\mathbf{\Delta}_{k}^{\prime},\left\{ \mathbf{\Delta}_{i}\right\} _{k+1\leq i<t}\right)\right|\left\{ \mathbf{\Delta}_{j}^{\prime}\right\} _{0\leq j\leq k-1},\left\{ \mathbf{\Delta}_{i}\right\} _{k+1\leq i<t}\right]\right]\right\Vert _{F}\nonumber \\
		& \leq\mathbb{E}\left[\left\Vert \mathbb{E}\left[\nabla_{\mathbf{\mathbf{M}^{\prime}}}c_{t}\left(\mathbf{M}^{\prime},\left\{ \mathbf{\Delta}_{j}^{\prime}\right\} _{0\leq j\leq k-1},\mathbf{\Delta}_{k},\left\{ \mathbf{\Delta}_{i}\right\} _{k+1\leq i<t}\right)\right.\right.\right.\nonumber \\
		& \left.\left.\left.\left.-\nabla_{\mathbf{M}^{\prime}}c_{t}\left(\mathbf{M}^{\prime},\left\{ \mathbf{\Delta}_{j}^{\prime}\right\} _{0\leq j\leq k-1},\mathbf{\Delta}_{k}^{\prime},\left\{ \mathbf{\Delta}_{i}\right\} _{k+1\leq i<t}\right)\right|\left\{ \mathbf{\Delta}_{j}^{\prime}\right\} _{0\leq j\leq k-1},\left\{ \mathbf{\Delta}_{i}\right\} _{k+1\leq i<t}\right]\right\Vert _{F}\right].\nonumber 
	\end{align}
	
	Given $\left\{ \left\{ \mathbf{\Delta}_{j}^{\prime}\right\} _{0\leq j\leq k-1},\left\{ \mathbf{\Delta}_{i}\right\} _{k+1\leq i<t}\right\} ,$
	for notation conciseness, we abbreviate 
	
	$\nabla_{\mathbf{\mathbf{M}^{\prime}}}c_{t}\left(\mathbf{M}^{\prime},\left\{ \mathbf{\Delta}_{j}^{\prime}\right\} _{0\leq j\leq k-1},\mathbf{\Delta}_{k},\left\{ \mathbf{\Delta}_{i}\right\} _{k+1\leq i<t}\right)$
	and $\nabla_{\mathbf{\mathbf{M}^{\prime}}}c_{t}\left(\mathbf{M}^{\prime},\left\{ \mathbf{\Delta}_{j}^{\prime}\right\} _{0\leq j\leq k-1},\mathbf{\Delta}_{k}^{\prime},\left\{ \mathbf{\Delta}_{i}\right\} _{k+1\leq i<t}\right)$
	as $\nabla_{\mathbf{\mathbf{M}^{\prime}}}c_{t}\left(\mathbf{M}^{\prime},\mathbf{\Delta}_{k}\right)$
	and $\nabla_{\mathbf{\mathbf{M}^{\prime}}}c_{t}\left(\mathbf{M}^{\prime},\mathbf{\Delta}_{k}^{\prime}\right),$
	respectively. It follows that 
	\begin{align}
		& \left\Vert \mathbb{E}\left[\nabla_{\mathbf{\mathbf{M}^{\prime}}}c_{t}\left(\mathbf{M}^{\prime},\mathbf{\Delta}_{k}\right)-\nabla_{\mathbf{\mathbf{M}^{\prime}}}c_{t}\left(\mathbf{M}^{\prime},\mathbf{\Delta}_{k}^{\prime}\right)\right]\right\Vert _{F}^{2}\label{eq:Smooth Average per stage cost -1}\\
		& =\mathrm{Tr}\left(\left(\mathbb{E}\left[\nabla_{\mathbf{\mathbf{M}^{\prime}}}c_{t}\left(\mathbf{M}^{\prime},\mathbf{\Delta}_{k}\right)\right]-\mathbb{E}\left[\nabla_{\mathbf{\mathbf{M}^{\prime}}}c_{t}\left(\mathbf{M}^{\prime},\mathbf{\Delta}_{k}^{\prime}\right)\right]\right)^{\top}\left(\mathbb{E}\left[\nabla_{\mathbf{\mathbf{M}^{\prime}}}c_{t}\left(\mathbf{M}^{\prime},\mathbf{\Delta}_{k}\right)\right]-\mathbb{E}\left[\nabla_{\mathbf{\mathbf{M}^{\prime}}}c_{t}\left(\mathbf{M}^{\prime},\mathbf{\Delta}_{k}^{\prime}\right)\right]\right)\right)\nonumber \\
		& =\left\Vert \mathbb{E}\left[\nabla_{\mathbf{\mathbf{M}^{\prime}}}c_{t}\left(\mathbf{M}^{\prime},\mathbf{\Delta}_{k}\right)-\nabla_{\mathbf{\mathbf{M}^{\prime}}}c_{t}\left(\mathbf{M}^{\prime},\mathbf{\Delta}_{k}^{\prime}\right)\right]\right\Vert _{F}\mathrm{Tr}\left(V^{\top}\mathbb{E}\left[\nabla_{\mathbf{\mathbf{M}^{\prime}}}c_{t}\left(\mathbf{M}^{\prime},\mathbf{\Delta}_{k}\right)\right]-V^{\top}\mathbb{E}\left[\nabla_{\mathbf{\mathbf{M}^{\prime}}}c_{t}\left(\mathbf{M}^{\prime},\mathbf{\Delta}_{k}^{\prime}\right)\right]\right),\nonumber 
	\end{align}
	where 
	\begin{align}
		V & =\frac{\mathbb{E}\left[\nabla_{\mathbf{\mathbf{M}^{\prime}}}c_{t}\left(\mathbf{M}^{\prime},\mathbf{\Delta}_{k}\right)\right]-\mathbb{E}\left[\nabla_{\mathbf{\mathbf{M}^{\prime}}}c_{t}\left(\mathbf{M}^{\prime},\mathbf{\Delta}_{k}^{\prime}\right)\right]}{\left\Vert \mathbb{E}\left[\nabla_{\mathbf{\mathbf{M}^{\prime}}}c_{t}\left(\mathbf{M}^{\prime},\mathbf{\Delta}_{k}\right)-\nabla_{\mathbf{\mathbf{M}^{\prime}}}c_{t}\left(\mathbf{M}^{\prime},\mathbf{\Delta}_{k}^{\prime}\right)\right]\right\Vert _{F}},\ \left\Vert V\right\Vert _{F}=1.
	\end{align}
	
	Based on Lemma \ref{Lemma: Smoothness per Stage Cost}, given $\left\{ \left\{ \mathbf{\Delta}_{j}^{\prime}\right\} _{0\leq j\leq k-1},\left\{ \mathbf{\Delta}_{i}\right\} _{k+1\leq i<t}\right\} ,$
	the conditional gradient $V^{\top}\mathbb{E}\left[\nabla_{\mathbf{\mathbf{M}^{\prime}}}c_{t}\left(\mathbf{M}^{\prime},\mathbf{\Delta}_{k}\right)\right]$
	is $\nu_{t}$-Lipschitz in $\mathbf{\Delta}_{k}.$ Further note that
	$\mathbf{\Delta}_{k}\sim\mathcal{D}_{k}\left(\mathbf{M}_{1}\right)$
	and $\mathbf{\Delta}_{k}^{\prime}\sim\mathcal{D}_{k}\left(\mathbf{M}_{2}\right)$,
	which are both $\varepsilon_{k}$-sensitive. Applying Lemma \ref{Lemma: Kantorovich-Rubinstein },
	it follows that 
	\begin{align}
		& \mathrm{Tr}\left(V^{\top}\mathbb{E}\left[\nabla_{\mathbf{\mathbf{M}^{\prime}}}c_{t}\left(\mathbf{M}^{\prime},\mathbf{\Delta}_{k}\right)\right]-V^{\top}\mathbb{E}\left[\nabla_{\mathbf{\mathbf{M}^{\prime}}}c_{t}\left(\mathbf{M}^{\prime},\mathbf{\Delta}_{k}^{\prime}\right)\right]\right)\leq\nu_{t}\varepsilon_{k}\left\Vert \mathbf{M}_{1}-\mathbf{M}_{2}\right\Vert _{F}.\label{eq:Smooth Average per stage cost -2}
	\end{align}
	
	Substitute (\ref{eq:Smooth Average per stage cost -2}) into (\ref{eq:Smooth Average per stage cost -1}),
	it follows that 
	\begin{align}
		& \left\Vert \mathbb{E}\left[\nabla_{\mathbf{\mathbf{M}^{\prime}}}c_{t}\left(\mathbf{M}^{\prime},\mathbf{\Delta}_{k}\right)-\nabla_{\mathbf{\mathbf{M}^{\prime}}}c_{t}\left(\mathbf{M}^{\prime},\mathbf{\Delta}_{k}^{\prime}\right)\right]\right\Vert _{F}\leq\nu_{t}\varepsilon_{k}\left\Vert \mathbf{M}_{1}-\mathbf{M}_{2}\right\Vert _{F}.
	\end{align}
	
	Using the similar approach, we can obtain 
	\begin{align}
		& \left\Vert \mathbb{E}\left[\nabla_{\mathbf{M}}c_{t}\left(\mathbf{M};\left\{ \mathbf{\Delta}_{i}\right\} _{0\leq i<t}\right)-\nabla_{\mathbf{M}^{\prime}}c_{t}\left(\mathbf{M}^{\prime},\left\{ \mathbf{\Delta}_{i}\right\} _{0\leq i<t}\right)\right]\right\Vert \leq\lambda_{t}\left\Vert \mathbf{M}-\mathbf{M}^{\prime}\right\Vert _{F};\label{eq:expanded average grad diff norm 1}\\
		& \left\Vert \mathbb{E}\left[\nabla_{\mathbf{M}^{\prime}}c_{t}\left(\mathbf{M}^{\prime},\left\{ \mathbf{\Delta}_{i}\right\} _{0\leq i<t}\right)-\nabla_{\mathbf{M}^{\prime}}c_{t}\left(\mathbf{M}^{\prime},\mathbf{A}_{1}^{\prime},\left\{ \mathbf{\Delta}_{i}\right\} _{1\leq i<t}\right)\right]\right\Vert \leq\nu_{t}\varepsilon_{1}\left\Vert \mathbf{M}_{1}-\mathbf{M}_{2}\right\Vert _{F};\label{eq:expanded average grad diff norm 2}\\
		& \vdots\nonumber \\
		& \left\Vert \mathbb{E}\left[\nabla_{\mathbf{M}^{\prime}}c_{t}\left(\mathbf{M}^{\prime},\left\{ \mathbf{\Delta}_{i}^{\prime}\right\} _{0\leq i<t-1},\mathbf{\Delta}_{t-1}\right)-\nabla_{\mathbf{M}^{\prime}}c_{t}\left(\mathbf{M}^{\prime},\left\{ \mathbf{\Delta}_{i}^{\prime}\right\} _{0\leq i<t-1}\right)\right]\right\Vert \nonumber\\
		&\leq\nu_{t}\varepsilon_{t-1}\left\Vert \mathbf{M}_{1}-\mathbf{M}_{2}\right\Vert _{F}.\label{eq:expanded average grad diff norm 3}
	\end{align}
	
	Combining (\ref{eq: expanded average grad diff norm}), (\ref{eq:expanded average grad diff norm 1})-(\ref{eq:expanded average grad diff norm 3}),
	it follow inequality (\ref{eq:average grad diff2}).
\end{proof}

\section{Properties of the Total Cost Function}
We define a total cost function $J_{T}$ as
\begin{align}
	& J_{T}\left(\mathbf{M};\left\{ \mathbf{\Delta}_{t}\right\} _{0\leq t<T}\right)=\sum_{t=0}^{T}c_{t}\left(\mathbf{M};\left\{ \mathbf{\Delta}_{i}\right\} _{0\leq i<t}\right).
\end{align}
It is clear that 
\begin{align}
	& C_{T}\left(\mathbf{M};\mathbf{M}_{1}\right)=\mathbb{E}_{\mathbf{x}_{0},\left\{ \mathbf{\Delta}_{t}\sim\mathcal{D}_{t}\left(\mathbf{M}_{1}\right)\right\} _{0\leq i<T},\left\{ \mathbf{w}_{i}\right\} _{0\leq i<T}}J_{T}\left(\mathbf{M};\left\{ \mathbf{\Delta}_{t}\right\} _{0\leq t<T}\right).
\end{align}

We next analyze the gradient properties and the smoothness of the
total cost function $J_{T}\left(\mathbf{M};\left\{ \mathbf{\Delta}_{t}\right\} _{0\leq t<T}\right)$.
We again note that $\nabla J_{T}\left(\mathbf{M};\left\{ \mathbf{\Delta}_{t}\right\} _{0\leq t<T}\right)$
denotes the gradient taken w.r.t. the first argument $\mathbf{M}$.
\begin{lemma}
	\label{lemma Gradient-Variance total cost}The
	variance of the gradient of the total cost function satisfies
	\begin{align}
		& \mathbb{E}_{\mathbf{x}_{0},\left\{ \mathbf{\Delta}_{t}\sim\mathcal{D}_{t}\left(\mathbf{M}\right)\right\} _{0\leq i<T},\left\{ \mathbf{w}_{i}\right\} _{0\leq i<T}}\left[\left\Vert \nabla J_{T}\left(\mathbf{M};\left\{ \mathbf{\Delta}_{t}\right\} _{0\leq t<T}\right)-\nabla C_{T}\left(\mathbf{M};\mathbf{M}\right)\right\Vert _{F}^{2}\right]\leq T\sum_{t=1}^{T}\vartheta_{t}^{2}.\label{eq:var-grad-M total cost}
	\end{align}
\end{lemma}

\begin{proof}
	Note that $\nabla J_{0}\left(\mathbf{M};\mathbf{\Delta}_{0}\right)=\nabla_{\mathbf{M}}c_{0}\left(\mathbf{x}_{0},\mathbf{u}_{0}\right)=\nabla_{\mathbf{M}}c_{0}\left(\mathbf{x}_{0},-\mathbf{K}\mathbf{x}_{0}\right)=\mathbf{0}.$
	Inequality (\ref{eq:var-grad-M total cost}) then follows directly
	from the norm triangular inequality:
	\begin{align}
		& \mathbb{E}\left[\left\Vert \nabla J_{T}\left(\mathbf{M};\left\{ \mathbf{\Delta}_{t}\right\} _{0\leq t<T}\right)-\mathbb{E}\left[\nabla J_{T}\left(\mathbf{M};\left\{ \mathbf{\Delta}_{t}\right\} _{0\leq t<T}\right)\right]\right\Vert _{F}^{2}\right]\label{eq: var total cost}\\
		& \leq T\sum_{t=1}^{T}\mathbb{E}_{\mathbf{x}_{0},\left\{ \mathbf{\Delta}_{i}\sim\mathcal{D}_{i}\left(\mathbf{M}\right)\right\} _{0\leq i<t},\left\{ \mathbf{w}_{i}\right\} _{0\leq i<t}}\left[\left\Vert \nabla_{\mathbf{M}}c_{t}\left(\mathbf{x}_{t}^{\left(\mathbf{M}\right)},\mathbf{u}_{t}^{\left(\mathbf{M}\right)}\right)\right.\right\Vert \nonumber \\
		& -\left.\left.\mathbb{E}_{\mathbf{x}_{0},\left\{ \mathbf{\Delta}_{i}\sim\mathcal{D}_{i}\left(\mathbf{M}\right)\right\} _{0\leq i<t},\left\{ \mathbf{w}_{i}\right\} _{0\leq i<t}}\left[\nabla_{\mathbf{M}}c_{t}\left(\mathbf{x}_{t}^{\left(\mathbf{M}\right)},\mathbf{u}_{t}^{\left(\mathbf{M}\right)}\right)\right]\right\Vert _{F}^{2}\right]\nonumber \\
		& \overset{\left(\ref{eq: var total cost}.a\right)}{\leq}T\sum_{t=1}^{T}\vartheta_{t}^{2},\nonumber 
	\end{align}
	where inequality $\left(\ref{eq: var total cost}.a\right)$ holds
	becasue of Lemma \ref{lemma Gradient-Variance per stage}.
\end{proof}

\section{Proof of Lemma \ref{s_cvx_t}}
Denote $f_{t}\left(\mathbf{M};\mathbf{M}_{1}\right)=\mathbb{E}_{\mathbf{x}_{0},\left\{ \mathbf{\Delta}_{i}\sim\mathcal{D}_{i}\left(\mathbf{M}_{1}\right)\right\} _{0\leq i<t},\left\{ \mathbf{w}_{i}\right\} _{0\leq i<t}}\left[c_{t}\left(\mathbf{M};\left\{ \mathbf{\Delta}_{i}\right\} _{0\leq i<t}\right)\right]$
	as the per stage expected cost with the distribution of policy-dependent
	perturbation is shifted from $\mathbf{\Delta}_{i}\sim\mathcal{D}_{i}\left(\mathbf{M}\right)$
	to $\mathbf{\Delta}_{i}\sim\mathcal{D}_{i}\left(\mathbf{M}_{1}\right)$,
	$\forall0\leq i<t.$
	
	From the convexity of per stage expected cost in  {Lemma \ref{p_stage_a_cost}}, we know that, $\forall1\leq t_{1}<H,$ $f_{t_{1}}\left(\mathbf{M};\mathbf{M}_{1}\right)$
	is a convex function of $\mathbf{M}$. It follows that 
	\begin{align}
		f_{t_{1}}\left(\mathbf{M};\mathbf{M}_{1}\right)\geq & f_{t_{1}}\left(\mathbf{M}^{\prime};\mathbf{M}_{1}\right)+\mathrm{Tr}\left(\left(\nabla f_{t_{1}}\left(\mathbf{M}^{\prime};\mathbf{M}_{1}\right)\right)^{\top}\left(\mathbf{M}-\mathbf{M}^{\prime}\right)\right),\forall1\leq t_{1}<H,\label{eq: convex per stage 1<t<H}
	\end{align}
	where the gradient $\nabla f_{t_{1}}\left(\mathbf{M}^{\prime};\mathbf{M}_{1}\right)$
	is taken w.r.t. the first argument $\mathbf{M}^{\prime}$.
	
	From stong convexity of per stage expected cost in  {Lemma \ref{s_cvx}}, we know that, $\forall H\leq t_{2}\leq T,$ $f_{t_{2}}\left(\mathbf{M};\mathbf{M}_{1}\right)$
	is  {$\min\left\{ \frac{\mu\sigma^{2}}{2},\frac{\mu\gamma^{2}\sigma^{2}}{64\kappa^{10}}\right\} $}-strongly
	convex function of $\mathbf{M}$. Therefore, it follows
	\begin{align}
		f_{t_{2}}\left(\mathbf{M};\mathbf{M}_{1}\right)\geq & f_{t_{2}}\left(\mathbf{M}^{\prime};\mathbf{M}_{1}\right)+\mathrm{Tr}\left(\left(\nabla f_{t_{2}}\left(\mathbf{M}^{\prime};\mathbf{M}_{1}\right)\right)^{\top}\left(\mathbf{M}-\mathbf{M}^{\prime}\right)\right)\label{eq: scvx per stage H<t<T}\\
		& +{\min\left\{ \frac{\mu\sigma^{2}}{2},\frac{\mu\gamma^{2}\sigma^{2}}{64\kappa^{10}}\right\} }\left\Vert \mathbf{M}-\mathbf{M}^{\prime}\right\Vert _{F}^{2},\forall H\leq t_{2}\leq T.\nonumber 
	\end{align}
	
	Sum up (\ref{eq: convex per stage 1<t<H}) from $t_{1}=1$ to $t_{1}=H-1$
	and (\ref{eq: scvx per stage H<t<T}) from $t_{2}=H$ to $t_{2}=T$,
	we obatin 
	\begin{align}
		& \sum_{t_{1}=1}^{H-1}f_{t_{1}}\left(\mathbf{M};\mathbf{M}_{1}\right)+\sum_{t_{2}=H}^{T}f_{t_{1}}\left(\mathbf{M};\mathbf{M}_{1}\right)\geq\sum_{t_{1}=1}^{H-1}f_{t_{1}}\left(\mathbf{M}^{\prime};\mathbf{M}_{1}\right)+\sum_{t_{2}=H}^{T}f_{t_{1}}\left(\mathbf{M}^{\prime};\mathbf{M}_{1}\right)\\
		& +\mathrm{Tr}\left(\left(\nabla\left(\sum_{t_{1}=1}^{H-1}f_{t_{1}}\left(\mathbf{M}^{\prime};\mathbf{M}_{1}\right)+\sum_{t_{2}=H}^{T}f_{t_{1}}\left(\mathbf{M}^{\prime};\mathbf{M}_{1}\right)\right)\right)^{\top}\left(\mathbf{M}-\mathbf{M}^{\prime}\right)\right)\nonumber \\
		& +\frac{\widetilde{\mu}}{2}\left\Vert \mathbf{M}-\mathbf{M}^{\prime}\right\Vert _{F}^{2}.\nonumber 
	\end{align}
	
	Inequality (\ref{scvx_eq}) follows directly
	by noting that 
	\begin{align}
		& C_{T}\left(\mathbf{M};\mathbf{M}_{1}\right)=\sum_{t_{1}=1}^{H-1}f_{t_{1}}\left(\mathbf{M};\mathbf{M}_{1}\right)+\sum_{t_{2}=H}^{T}f_{t_{1}}\left(\mathbf{M};\mathbf{M}_{1}\right),\\
		& \nabla C_{T}\left(\mathbf{M}^{\prime};\mathbf{M}_{1}\right)=\nabla\left(\sum_{t_{1}=1}^{H-1}f_{t_{1}}\left(\mathbf{M}^{\prime};\mathbf{M}_{1}\right)+\sum_{t_{2}=H}^{T}f_{t_{1}}\left(\mathbf{M}^{\prime};\mathbf{M}_{1}\right)\right).
	\end{align}

\section{Proof of Lemma \ref{smooth}}
	Note that $\nabla C_{0}\left(\mathbf{M};\mathbf{M}_{1}\right)=\nabla_{\mathbf{M}}\mathbb{E}\left[c_{t}\left(\mathbf{x}_{0},\mathbf{u}_{0}\right)\right]=\nabla_{\mathbf{M}}\mathbb{E}\left[c_{t}\left(\mathbf{x}_{0},-\mathbf{K}\mathbf{x}_{0}\right)\right]=\mathbf{0}.$
	Therefore, inequality (\ref{eq:average grad diff}) follows
	directly by combining the norm triangular inequalityand the results
	in Lemma \ref{Lemma: Smoothness per Stage Cost}. 
	\begin{align}
		& \left\Vert \nabla C_{T}\left(\mathbf{M};\mathbf{M}_{1}\right)-\nabla C_{T}\left(\mathbf{M}^{\prime};\mathbf{M}_{2}\right)\right\Vert _{F}\\
		& \leq\sum_{t=1}^{T}\left\Vert \mathbb{E}_{\mathbf{x}_{0},\left\{ \mathbf{\Delta}_{i}\sim\mathcal{D}_{i}\left(\mathbf{M}_{1}\right)\right\} _{0\leq i<t},\left\{ \mathbf{w}_{i}\right\} _{0\leq i<t}}\left[\nabla_{\mathbf{M}}c_{t}\left(\mathbf{M};\left\{ \mathbf{\Delta}_{i}\right\} _{0\leq i<t}\right)\right]\right.\nonumber \\
		& \left.-\mathbb{E}_{\mathbf{x}_{0},\left\{ \mathbf{\Delta}_{i}^{\prime}\sim\mathcal{D}_{i}\left(\mathbf{M}_{2}\right)\right\} _{0\leq i<t},\left\{ \mathbf{w}_{i}\right\} _{0\leq i<t}}\left[\nabla_{\mathbf{M}^{\prime}}c_{t}\left(\mathbf{M}^{\prime};\left\{ \mathbf{\Delta}_{i}^{\prime}\right\} _{0\leq i<t}\right)\right]\right\Vert _{F}\nonumber \\
		& \leq\left(\sum_{t=1}^{T}\lambda_{t}\right)\left\Vert \mathbf{M}-\mathbf{M}^{\prime}\right\Vert _{F}+\sum_{t=1}^{T}\left(\nu_{t}\sum_{i=0}^{t-1}\varepsilon_{i}\right)\left\Vert \mathbf{M}_{1}-\mathbf{M}_{2}\right\Vert _{F}\nonumber \\
		& =\left(\sum_{t=1}^{T}\lambda_{t}\right)\left\Vert \mathbf{M}-\mathbf{M}^{\prime}\right\Vert _{F}+\sum_{t=0}^{T-1}\varepsilon_{t}\left(\sum_{i=t+1}^{T}\nu_{i}\right)\left\Vert \mathbf{M}_{1}-\mathbf{M}_{2}\right\Vert _{F}.\nonumber 
	\end{align}

\section{Proof of Lemma 	\ref{Lemma: Existence Uniqueness M^PS}}
	Fix any $\mathbf{M},\mathbf{M}^{\prime}\in\mathcal{\mathbb{M}},$
	the optimality condition to (\ref{eq:RRM}) implies that 
	\begin{align}
		& \nabla C_{T}\left(\Phi\left(\mathbf{M}\right);\mathbf{M}\right)=\mathbf{0},\nabla C_{T}\left(\Phi\left(\mathbf{M}^{\prime}\right);\mathbf{M}^{\prime}\right)=\mathbf{0},
	\end{align}
	where the gradients are again taken w.r.t. the first argument in the
	function $C_{T}\left(\cdot,\cdot\right)$. Observe the fact
	\begin{align}
		 0&=\mathrm{Tr}\left(\mathbf{0}^{T}\cdot\left(\Phi\left(\mathbf{M}\right)-\Phi\left(\mathbf{M}^{\prime}\right)\right)\right)\nonumber\\
		&=\mathrm{Tr}\left(\left(\nabla C_{T}\left(\Phi\left(\mathbf{M}\right);\mathbf{M}\right)-\nabla C_{T}\left(\Phi\left(\mathbf{M}^{\prime}\right);\mathbf{M}^{\prime}\right)\right)\cdot\left(\Phi\left(\mathbf{M}\right)-\Phi\left(\mathbf{M}^{\prime}\right)\right)\right).
	\end{align}
	
	Adding and subtracting the term $\nabla C_{T}\left(\Phi\left(\mathbf{M}\right);\mathbf{M}^{\prime}\right)$
	implies the equality
	\begin{align}
		& \mathrm{Tr}\left(\left(\nabla C_{T}\left(\Phi\left(\mathbf{M}\right);\mathbf{M}^{\prime}\right)-\nabla C_{T}\left(\Phi\left(\mathbf{M}\right);\mathbf{M}\right)\right)\cdot\left(\Phi\left(\mathbf{M}\right)-\Phi\left(\mathbf{M}^{\prime}\right)\right)\right)\label{eq: grad F_T diff}\\
		& =\mathrm{Tr}\left(\left(\nabla C_{T}\left(\Phi\left(\mathbf{M}\right);\mathbf{M}^{\prime}\right)-\nabla C_{T}\left(\Phi\left(\mathbf{M}^{\prime}\right);\mathbf{M}^{\prime}\right)\right)\cdot\left(\Phi\left(\mathbf{M}\right)-\Phi\left(\mathbf{M}^{\prime}\right)\right)\right).\nonumber 
	\end{align}
	
	Recall that $C_{T}\left(\Phi\left(\mathbf{M}\right);\mathbf{M}^{\prime}\right)$
	is $\widetilde{\mu}$-strongly convex w.r.t. $\Phi\left(\mathbf{M}\right)$,
	it follows 
	\begin{alignat}{1}
		& \mathrm{Tr}\left(\left(\nabla C_{T}\left(\Phi\left(\mathbf{M}\right);\mathbf{M}^{\prime}\right)-\nabla C_{T}\left(\Phi\left(\mathbf{M}^{\prime}\right);\mathbf{M}^{\prime}\right)\right)\cdot\left(\Phi\left(\mathbf{M}\right)-\Phi\left(\mathbf{M}^{\prime}\right)\right)\right)\label{eq:grad F_T diff 1}\\
		& \geq\widetilde{\mu}\left\Vert \Phi\left(\mathbf{M}\right)-\Phi\left(\mathbf{M}^{\prime}\right)\right\Vert _{F}^{2}.\nonumber 
	\end{alignat}
	
	Meanwhile, applying Lemma \ref{smooth}
	to the left hand side of (\ref{eq: grad F_T diff}) gives
	\begin{align}
		& \mathrm{Tr}\left(\left(\nabla C_{T}\left(\Phi\left(\mathbf{M}\right);\mathbf{M}^{\prime}\right)-\nabla C_{T}\left(\Phi\left(\mathbf{M}\right);\mathbf{M}\right)\right)\cdot\left(\Phi\left(\mathbf{M}\right)-\Phi\left(\mathbf{M}^{\prime}\right)\right)\right)\label{eq:grad F_T diff 2}\\
		& \leq\sum_{t=0}^{T-1}\left(\varepsilon_{t}\sum_{i=t+1}^{T}\nu_{i}\right)\left\Vert \mathbf{M}-\mathbf{M}^{\prime}\right\Vert _{F}\left\Vert \Phi\left(\mathbf{M}\right)-\Phi\left(\mathbf{M}^{\prime}\right)\right\Vert _{F}.\nonumber 
	\end{align}
	
	Combining (\ref{eq: grad F_T diff}), (\ref{eq:grad F_T diff 1})
	and (\ref{eq:grad F_T diff 2}), it follows
	\begin{align}\label{102}
		& \left\Vert \Phi\left(\mathbf{M}\right)-\Phi\left(\mathbf{M}^{\prime}\right)\right\Vert _{F}\leq\frac{\sum_{t=0}^{T-1}\left(\varepsilon_{t}\sum_{i=t+1}^{T}\nu_{i}\right)}{\widetilde{\mu}}\left\Vert \mathbf{M}-\mathbf{M}^{\prime}\right\Vert _{F}.
	\end{align}
	We note that $\mathbf{M}_{n+1}=\Phi\left(\mathbf{M}_{n}\right)$
	by the definition of (\ref{eq:RRM}), and $\mathbf{M}^{PS}=\Phi\left(\mathbf{M}^{PS}\right)$
	by the definition of performative stability. Applying \eqref{102} yields
	\begin{align}
		& \left\Vert \mathbf{M}_{n}-\mathbf{M}^{PS}\right\Vert _{F}=\left\Vert \Phi\left(\mathbf{M}_{n-1}\right)-\Phi\left(\mathbf{M}^{PS}\right)\right\Vert _{F}\leq\frac{\sum_{t=0}^{T-1}\left(\varepsilon_{t}\sum_{i=t+1}^{T}\nu_{i}\right)}{\widetilde{\mu}}\left\Vert \mathbf{M}_{l-1}-\mathbf{M}^{PS}\right\Vert _{F}\\
		& \leq\left(\frac{\sum_{t=0}^{T-1}\left(\varepsilon_{t}\sum_{i=t+1}^{T}\nu_{i}\right)}{\widetilde{\mu}}\right)^{n}\left\Vert \mathbf{M}_{0}-\mathbf{M}^{PS}\right\Vert _{F}.\nonumber 
	\end{align}
	Setting $\left(\frac{\sum_{t=0}^{T-1}\left(\varepsilon_{t}\sum_{i=t+1}^{T}\nu_{i}\right)}{\widetilde{\mu}}\right)^{n}\left\Vert \mathbf{M}_{0}-\mathbf{M}^{PS}\right\Vert _{F}$
	to be at most $\rho$ and solving for $n$ completes the proof.

\section{Convergence Analysis of Algorithm 1}
\subsection{Non-asymptomatic Convergence Analysis for Algorithm 1 with General Step Sizes}\label{G.1}
We have the following Lemma \ref{Thm: Convergence RSGD appendix_lem} regarding the convergence results of the proposed RSGD in Algorithm 1.

\begin{lem}
	\label{Thm: Convergence RSGD appendix_lem} Under A\ref{Assumption noises wt}-A\ref{A6}. Consider a sequence of non-negative step sizes $\left\{ \eta_{n},n\geq0\right\} $
	satisfy 
	\begin{align}\label{lem1_them}
		\sup_{n\geq0}\eta_{n} & \leq\mathrm{min}\left\{ \frac{\widetilde{\mu}-\sum_{t=0}^{T-1}\left(\varepsilon_{t}\sum_{i=t+1}^{T}\nu_{i}\right)}{2\left(\sum_{t=1}^{T}\lambda_{t}+\sum_{t=0}^{T-1}\left(\varepsilon_{t}\sum_{i=t+1}^{T}\nu_{i}\right)\right)^{2}},\frac{2}{\widetilde{\mu}-\sum_{t=0}^{T-1}\left(\varepsilon_{t}\sum_{i=t+1}^{T}\nu_{i}\right)}\right\} ,
	\end{align}
	and
   \begin{align}\label{lem2_them}
   	\frac{\eta_{n}}{\eta_{n+1}}\leq\left(1+\frac{1}{2}\left(\widetilde{\mu}-\sum_{t=0}^{T-1}\left(\varepsilon_{t}\sum_{i=t+1}^{T}\nu_{i}\right)\right)\eta_{n+1}\right),
   \end{align}
	for any $n\geq0$. Then, the iterates generated by RSGD admit the
	following bound for any $N\geq1$:
	\begin{align}
		\mathbb{E}\left[\left\Vert \mathbf{M}_{N}-\mathbf{M}^{PS}\right\Vert _{F}^{2}\right]\leq & \prod_{n=0}^{N-1}\left(1-\eta_{n}\left(\widetilde{\mu}-\sum_{t=0}^{T-1}\left(\varepsilon_{t}\sum_{i=t+1}^{T}\nu_{i}\right)\right)\right)\mathbb{E}\left[\left\Vert \mathbf{M}_{0}-\mathbf{M}^{PS}\right\Vert _{F}^{2}\right]\label{eq: RSGD error appendix_lem}\\
		& +\frac{4\eta_{N-1}T\sum_{t=1}^{T}\vartheta_{t}^{2}}{\widetilde{\mu}-\sum_{t=0}^{T-1}\left(\varepsilon_{t}\sum_{i=t+1}^{T}\nu_{i}\right)},\nonumber 
	\end{align}
	where $\vartheta_{t}=\kappa^{3}G\left(\left(HW+\kappa^{2}\right)\kappa\left\Vert \mathbf{B}\right\Vert \beta_{t}+1\right)\left(x_{0}\alpha_{t}+c_{3}\beta_{t}\right)+GHWM\left(\kappa^{3}\beta_{t}+1\right),\forall1\leq t\leq T.$
\end{lem}
\begin{proof}
	Since projecting onto a convex set can only bring two iterates closer
	together, it follows that 
	\begin{align}
		& \mathbb{E}\left[\left\Vert \mathbf{M}_{n+1}-\mathbf{M}^{PS}\right\Vert _{F}^{2}\right]\leq\mathbb{E}\left[\left\Vert \mathbf{M}_{n}-\eta_{n}\nabla J_{T}\left(\mathbf{M}_{n};\left\{ \mathbf{\Delta}_{t}\right\} _{0\leq t<T}\right)-\mathbf{M}^{PS}\right\Vert ^{2}\right]\label{eq: RSGD M diff}\\
		& \leq I_{1}+I_{2}+I_{3},\nonumber 
	\end{align}
	where 
	\begin{align}
		& I_{1}=\mathbb{E}\left[\left\Vert \mathbf{M}_{n}-\mathbf{M}^{PS}\right\Vert _{F}^{2}\right],\\
		& I_{2}=-2\eta_{n}\mathbb{E}\left[\mathrm{Tr}\left(\left(\mathbf{M}_{n}-\mathbf{M}^{PS}\right)^{\top}\nabla J_{T}\left(\mathbf{M}_{n};\left\{ \mathbf{\Delta}_{t}\right\} _{0\leq t<T}\right)\right)\right],\\
		& I_{3}=\eta_{n}^{2}\mathbb{E}\left[\left\Vert \nabla J_{T}\left(\mathbf{M}_{n};\left\{ \mathbf{\Delta}_{t}\right\} _{0\leq t<T}\right)\right\Vert _{F}^{2}\right],
	\end{align}
	and $\mathbf{\Delta}_{t}\sim\mathcal{D}_{t}\left(\mathbf{M}_{n}\right),\forall0\leq t<T.$
	
	We first analyze the term $I_{2}.$ Let $\mathbb{E}_{n}\left[\cdot\right]$
	be the conditional expectation on $\mathbf{M}_{n}$. We have 
	\begin{align}
		& \mathbb{E}_{n}\left[\mathrm{Tr}\left(\left(\mathbf{M}_{n}-\mathbf{M}^{PS}\right)^{\top}\nabla J_{T}\left(\mathbf{M}_{n};\left\{ \mathbf{\Delta}_{t}\right\} _{0\leq t<T}\right)\right)\right]=\mathrm{Tr}\left(\left(\mathbf{M}_{n}-\mathbf{M}^{PS}\right)^{\top}\nabla C_{T}\left(\mathbf{M}_{n};\mathbf{M}_{n}\right)\right)\label{eq:RSGD term 2}\\
		& =\mathrm{Tr}\left(\left(\mathbf{M}_{n}-\mathbf{M}^{PS}\right)^{\top}\left(\nabla C_{T}\left(\mathbf{M}_{n};\mathbf{M}_{n}\right)-\nabla C_{T}\left(\mathbf{M}_{n};\mathbf{M}^{PS}\right)\right)\right)\nonumber \\
		& +\mathrm{Tr}\left(\left(\mathbf{M}_{n}-\mathbf{M}^{PS}\right)^{\top}\left(\nabla C_{T}\left(\mathbf{M}_{n};\mathbf{M}^{PS}\right)-\nabla C_{T}\left(\mathbf{M}^{PS};\mathbf{M}^{PS}\right)\right)\right),\nonumber 
	\end{align}
	where we use the fact that $\nabla C_{T}\left(\mathbf{M}^{PS};\mathbf{M}^{PS}\right)=\mathbf{0}$
	in the last equality of (\ref{eq:RSGD term 2}).
	
	Applying the Cauchy-Schwarz inequality and Lemma \ref{smooth},
	we obtain
	\begin{align}
		& \mathrm{Tr}\left(\left(\mathbf{M}_{n}-\mathbf{M}^{PS}\right)^{\top}\left(\nabla C_{T}\left(\mathbf{M}_{n};\mathbf{M}_{n}\right)-\nabla C_{T}\left(\mathbf{M}_{n};\mathbf{M}^{PS}\right)\right)\right)\label{eq:RSGD term 2-1}\\
		& \geq-\left\Vert \mathbf{M}_{n}-\mathbf{M}^{PS}\right\Vert _{F}\left(\sum_{t=0}^{T-1}\left(\varepsilon_{t}\sum_{i=t+1}^{T}\nu_{i}\right)\left\Vert \mathbf{M}_{n}-\mathbf{M}^{PS}\right\Vert _{F}\right)\nonumber \\
		& =-\sum_{t=0}^{T-1}\left(\varepsilon_{t}\sum_{i=t+1}^{T}\nu_{i}\right)\left\Vert \mathbf{M}_{n}-\mathbf{M}^{PS}\right\Vert _{F}^{2}.\nonumber 
	\end{align}
	
	Meanwhile, based on strongly convexity of $C_{T}$ in Lemma \ref{s_cvx_t},
	we have 
	\begin{align}
		& \mathrm{Tr}\left(\left(\mathbf{M}_{n}-\mathbf{M}^{PS}\right)^{\top}\left(\nabla C_{T}\left(\mathbf{M}_{n};\mathbf{M}^{PS}\right)-\nabla C_{T}\left(\mathbf{M}^{PS};\mathbf{M}^{PS}\right)\right)\right)\geq\widetilde{\mu}\left\Vert \mathbf{M}_{n}-\mathbf{M}^{PS}\right\Vert _{F}^{2}.\label{eq:RSGD term 2-2}
	\end{align}
	
	Substitute (\ref{eq:RSGD term 2-1}) and (\ref{eq:RSGD term 2-2})
	into (\ref{eq:RSGD term 2}) and take the full expectation, it follows
	that $I_{2}$ satisfies
	\begin{align}
		& I_{2}\leq-2\eta_{n}\left(\widetilde{\mu}-\sum_{t=0}^{T-1}\left(\varepsilon_{t}\sum_{i=t+1}^{T}\nu_{i}\right)\right)\mathbb{E}\left[\left\Vert \mathbf{M}_{n}-\mathbf{M}^{PS}\right\Vert _{F}^{2}\right].\label{eq: I_2 bound}
	\end{align}
	
	We next analyze the term $I_{3}.$ We observe the following equivalent
	expression for $I_{3}$
	\begin{align}
		& \left\Vert \nabla J_{T}\left(\mathbf{M}_{n};\left\{ \mathbf{\Delta}_{t}\right\} _{0\leq t<T}\right)\right\Vert _{F}^{2}=\left\Vert \nabla J_{T}\left(\mathbf{M}_{n};\left\{ \mathbf{\Delta}_{t}\right\} _{0\leq t<T}\right)-\nabla C_{T}\left(\mathbf{M}_{n};\mathbf{M}_{n}\right)\right.\\
		& +\nabla C_{T}\left(\mathbf{M}_{n};\mathbf{M}_{n}\right)-\left.\nabla C_{T}\left(\mathbf{M}^{PS};\mathbf{M}^{PS}\right)\right\Vert _{F}^{2}.\nonumber 
	\end{align}
	Therefore, 
	\begin{align}
		& \mathbb{E}\left[\left\Vert \nabla J_{T}\left(\mathbf{M}_{n};\left\{ \mathbf{\Delta}_{t}\right\} _{0\leq t<T}\right)\right\Vert _{F}^{2}\right]\label{eq:RSDG term 3}\\
		& \leq2\mathbb{E}\left[\left\Vert \nabla J_{T}\left(\mathbf{M}_{l};\left\{ \mathbf{\Delta}_{t}\right\} _{0\leq t<T}\right)-\nabla C_{T}\left(\mathbf{M}_{n};\mathbf{M}_{n}\right)\right\Vert _{F}^{2}\right]\nonumber \\
		& +2\mathbb{E}\left[\left\Vert \nabla C_{T}\left(\mathbf{M}_{n};\mathbf{M}_{n}\right)-\nabla C_{T}\left(\mathbf{M}^{PS};\mathbf{M}^{PS}\right)\right\Vert _{F}^{2}\right].\nonumber 
	\end{align}
	
	According to Lemma \ref{lemma Gradient-Variance total cost}, an upper
	bound of the variance of $\nabla J_{T}\left(\mathbf{M}_{n};\left\{ \mathbf{\Delta}_{t}\right\} _{0\leq t<T}\right)$
	is given by (\ref{eq:var-grad-M total cost}). Moreover, based on
	Lemma \ref{smooth}, we have 
	\begin{align}
		& \mathbb{E}\left[\left\Vert \nabla C_{T}\left(\mathbf{M}_{n};\mathbf{M}_{n}\right)-\nabla C_{T}\left(\mathbf{M}^{PS};\mathbf{M}^{PS}\right)\right\Vert _{F}^{2}\right]\label{eq:RSGD term 3 - 1}\\
		& \leq\left(\sum_{t=1}^{T}\lambda_{t}+\sum_{t=0}^{T-1}\left(\varepsilon_{t}\sum_{i=t+1}^{T}\nu_{i}\right)\right)^{2}\mathbb{E}\left[\left\Vert \mathbf{M}_{n}-\mathbf{M}^{PS}\right\Vert _{F}^{2}\right].\nonumber 
	\end{align}
	
	Substitute (\ref{eq:var-grad-M total cost}) and (\ref{eq:RSGD term 3 - 1})
	back into (\ref{eq:RSDG term 3}), it follows
	\begin{align}
		& I_{3}\leq2\eta_{n}^{2}T\sum_{t=1}^{T}\vartheta_{t}^{2}+2\eta_{n}^{2}\left(\sum_{t=1}^{T}\lambda_{t}+\sum_{t=0}^{T-1}\left(\varepsilon_{t}\sum_{i=t+1}^{T}\nu_{i}\right)\right)^{2}\mathbb{E}\left[\left\Vert \mathbf{M}_{n}-\mathbf{M}^{PS}\right\Vert _{F}^{2}\right].\label{eq: I_3 bound}
	\end{align}
	
	Now subtitute (\ref{eq: I_2 bound}) and (\ref{eq: I_3 bound}) back
	into (\ref{eq: RSGD M diff}), we obtain
	\begin{align}
		& \mathbb{E}\left[\left\Vert \mathbf{M}_{n+1}-\mathbf{M}^{PS}\right\Vert _{F}^{2}\right]\leq\mathbb{E}\left[\left\Vert \mathbf{M}_{n}-\mathbf{M}^{PS}\right\Vert _{F}^{2}\right]-2\eta_{n}\left(\widetilde{\mu}-\sum_{t=0}^{T-1}\left(\varepsilon_{t}\sum_{i=t+1}^{T}\nu_{i}\right)\right)\label{eq: E=00005BM=00005D diff bound}\\
		& \cdot\mathbb{E}\left[\left\Vert \mathbf{M}_{n}-\mathbf{M}^{PS}\right\Vert _{F}^{2}\right]+2\eta_{n}^{2}\left(\sum_{t=1}^{T}\lambda_{t}+\sum_{t=0}^{T-1}\left(\varepsilon_{t}\sum_{i=t+1}^{T}\nu_{i}\right)\right)^{2}\mathbb{E}\left[\left\Vert \mathbf{M}_{n}-\mathbf{M}^{PS}\right\Vert _{F}^{2}\right]+2\eta_{n}^{2}T\sum_{t=1}^{T}\vartheta_{t}^{2}\nonumber \\
		& =\left(1-2\eta_{n}\left(\widetilde{\mu}-\sum_{t=0}^{T-1}\left(\varepsilon_{t}\sum_{i=t+1}^{T}\nu_{i}\right)\right)+2\eta_{n}^{2}\left(\sum_{t=1}^{T}\lambda_{t}+\sum_{t=0}^{T-1}\left(\varepsilon_{t}\sum_{i=t+1}^{T}\nu_{i}\right)\right)^{2}\right)\mathbb{E}\left[\left\Vert \mathbf{M}_{n}-\mathbf{M}^{PS}\right\Vert _{F}^{2}\right]\nonumber \\
		& +2\eta_{n}^{2}T\sum_{t=1}^{T}\vartheta_{t}^{2}.\nonumber 
	\end{align}
	
	Let $\eta_{n}$ be choosen such that $\eta_{n}\left(\widetilde{\mu}-\sum_{t=0}^{T-1}\left(\varepsilon_{t}\sum_{i=t+1}^{T}\nu_{i}\right)\right)\geq2\eta_{n}^{2}\left(\sum_{t=1}^{T}\lambda_{t}+\sum_{t=0}^{T-1}\left(\varepsilon_{t}\sum_{i=t+1}^{T}\nu_{i}\right)\right)^{2}$.
	Then inequality (\ref{eq: E=00005BM=00005D diff bound}) is reduced
	to
	\begin{align}
		\mathbb{E}\left[\left\Vert \mathbf{M}_{n+1}-\mathbf{M}^{PS}\right\Vert _{F}^{2}\right]\leq & \left(1-\eta_{n}\left(\widetilde{\mu}-\sum_{t=0}^{T-1}\left(\varepsilon_{t}\sum_{i=t+1}^{T}\nu_{i}\right)\right)\right)\mathbb{E}\left[\left\Vert \mathbf{M}_{n}-\mathbf{M}^{PS}\right\Vert _{F}^{2}\right]\label{eq:final bound 1}\\
		& +2\eta_{n}^{2}T\sum_{t=1}^{T}\vartheta_{t}^{2}.\nonumber 
	\end{align}
	
	Therefore, 
	\begin{align}
		\mathbb{E}\left[\left\Vert \mathbf{M}_{N}-\mathbf{M}^{PS}\right\Vert _{F}^{2}\right]\leq & \prod_{n=0}^{N-1}\left(1-\eta_{n}\left(\widetilde{\mu}-\sum_{t=0}^{T-1}\left(\varepsilon_{t}\sum_{i=t+1}^{T}\nu_{i}\right)\right)\right)\mathbb{E}\left[\left\Vert \mathbf{M}_{0}-\mathbf{M}^{PS}\right\Vert _{F}^{2}\right]\label{eq:final bound 1-1}\\
		& +\sum_{n=0}^{N-1}\prod_{i=n+1}^{N-1}\left(1-\eta_{i}\left(\widetilde{\mu}-\sum_{t=0}^{T-1}\left(\varepsilon_{t}\sum_{i=t+1}^{T}\nu_{i}\right)\right)\right)2\eta_{n}^{2}T\sum_{t=1}^{T}\vartheta_{t}^{2}.\nonumber 
	\end{align}
	
	Let $\eta_{0}<\frac{2}{\widetilde{\mu}-\sum_{t=0}^{T-1}\left(\varepsilon_{t}\sum_{i=t+1}^{T}\nu_{i}\right)}$.
	Based on Lemma 6 in \cite{li2209multi}, if $\frac{\eta_{n}}{\eta_{n+1}}\leq\left(1+\frac{1}{2}\left(\widetilde{\mu}-\sum_{t=0}^{T-1}\left(\varepsilon_{t}\sum_{i=t+1}^{T}\nu_{i}\right)\right)\eta_{n+1}\right)$
	for any $n\geq0$, then 
	\begin{align}
		& \sum_{n=0}^{N-1}\eta_{n}^{2}\prod_{i=n+1}^{N-1}\left(1-\eta_{i}\left(\widetilde{\mu}-\sum_{t=0}^{T-1}\left(\varepsilon_{t}\sum_{i=t+1}^{T}\nu_{i}\right)\right)\right)\leq\frac{2\eta_{N-1}}{\widetilde{\mu}-\sum_{t=0}^{T-1}\left(\varepsilon_{t}\sum_{i=t+1}^{T}\nu_{i}\right)},\forall N\geq1.\label{eq:sum of products}
	\end{align}
	Substitute (\ref{eq:sum of products}) into (\ref{eq:final bound 1-1}),
	we obtain the desired result.
\end{proof}

Based on the step size requirements in Lemma \ref{Thm: Convergence RSGD appendix_lem}, RSGD converges when the aggregated sensitivity $\sum_{t=0}^{T-1}\left(\varepsilon_t \sum_{i=t+1}^T \nu_i\right)$ is strictly below the threshold $\widetilde{\mu}$ defined in Lemma \ref{Lemma: Existence Uniqueness M^PS}, i.e., the same sufficient condition that guarantees the existence of $\mathbf{M}^{P S}$.

The first term on the R. H. S. of \eqref{eq: RSGD error appendix_lem} decays sub-geometrically and is scaled by the initial error $\mathbb{E}\left[\left\|\mathbf{M}_0-\mathbf{M}^{P S}\right\|_F^2\right]$. The second term is a fluctuation term that only depends on the variance of the stochastic gradient, which decays at a slower rate as $\mathcal{O}\left(\eta_{N-1}\right)$.

\subsection{Convergence Analysis of Algorithm 1 with Diminishing Step Sizes}
We select the diminishing step sizes $\eta_n=\frac{\phi_1}{n+\phi_2}, \forall n \geq 0$, where $\phi_{1}$  and $\phi_{2}$ satisfy \eqref{17} and \eqref{18} simultaneously. We can verify that condition \eqref{lem1_them} in Lemma \ref{Thm: Convergence RSGD appendix_lem} is reduced to \eqref{17} in Theorem \ref{Thm: Convergence RSGD}. Besides, condition \eqref{lem2_them} in Lemma \ref{Thm: Convergence RSGD appendix_lem} can be further simplifed as 
\begin{align}
		\frac{\eta_{n}}{\eta_{n+1}}=\frac{\frac{\phi_1}{n+\phi_2}}{\frac{\phi_1}{n+1+\phi_2}}=1+\frac{1}{n+\phi_2}\leq \left(1+\frac{1}{2}\left(\widetilde{\mu}-\sum_{t=0}^{T-1}\left(\varepsilon_{t}\sum_{i=t+1}^{T}\nu_{i}\right)\right)\frac{\phi_1}{n+1+\phi_2}\right),
\end{align}
which is equivalent to 
\begin{align}\label{lem3_thm}
	\frac{n+1+\phi_2}{n+\phi_2}\leq \frac{1}{2}\left(\widetilde{\mu}-\sum_{t=0}^{T-1}\left(\varepsilon_{t}\sum_{i=t+1}^{T}\nu_{i}\right)\right)\phi_1.
\end{align}
Note that $\forall n \geq 0,$ to guarantee \eqref{lem3_thm} holds, we only need to let
\begin{align}\label{lem4_thm}
	\frac{2}{\left(\widetilde{\mu}-\sum_{t=0}^{T-1}\left(\varepsilon_{t}\sum_{i=t+1}^{T}\nu_{i}\right)\right)}\leq \frac{\phi_{1}\phi_2}{1+\phi_{2}}=\frac{\phi_1}{1+\frac{1}{\phi_2}},
\end{align}
which coincides with condition \eqref{18} in Theorem \ref{Thm: Convergence RSGD}.

Lastly, not that $1+x\leq e^x$ for all $x>0$, we can verify the first term in R.H.S. of \eqref{eq: RSGD error appendix_lem} can be upper bounded as 
\begin{align}\label{lem5_thm}
	&\prod_{n=0}^{N-1}\left(1-\eta_{n}\left(\widetilde{\mu}-\sum_{t=0}^{T-1}\left(\varepsilon_{t}\sum_{i=t+1}^{T}\nu_{i}\right)\right)\right)\mathbb{E}\left[\left\Vert \mathbf{M}_{0}-\mathbf{M}^{PS}\right\Vert _{F}^{2}\right]\nonumber\\
	&\leq e^{ -\sum_{n=0}^{N-1}\eta_{n}\left(\widetilde{\mu}-\sum_{t=0}^{T-1}\left(\varepsilon_{t}\sum_{i=t+1}^{T}\nu_{i}\right)\right)}\mathbb{E}\left[\left\Vert \mathbf{M}_{0}-\mathbf{M}^{PS}\right\Vert _{F}^{2}\right]\nonumber\\
	&= e^{-\sum_{n=0}^{N-1}\frac{\phi_{1}}{n+\phi_2}\left(\widetilde{\mu}-\sum_{t=0}^{T-1}\left(\varepsilon_{t}\sum_{i=t+1}^{T}\nu_{i}\right)\right)}\mathbb{E}\left[\left\Vert \mathbf{M}_{0}-\mathbf{M}^{PS}\right\Vert _{F}^{2}\right]\nonumber\\
	&\leq e^{-\sum_{n=1}^{N}\frac{\phi_{1}}{n}\left(\widetilde{\mu}-\sum_{t=0}^{T-1}\left(\varepsilon_{t}\sum_{i=t+1}^{T}\nu_{i}\right)\right)}\mathbb{E}\left[\left\Vert \mathbf{M}_{0}-\mathbf{M}^{PS}\right\Vert _{F}^{2}\right].
\end{align}

Therefore, Theorem \ref{Thm: Convergence RSGD} follows by combining \eqref{lem4_thm} and \eqref{lem5_thm}. 

\section{Proof of Proposition \ref{Proposition 1: Almost Sure Stable}}
Since $\left\Vert \mathbf{A}_{t}\right\Vert \leq1-\gamma+\kappa^{2}\xi_{t}\leq\zeta<1,\forall0\leq t<T,$
it follows directly that $\alpha_{t}=\prod_{i=0}^{t-1}\left(1-\gamma+\kappa^{2}\xi_{i}\right)\leq\prod_{i=0}^{t-1}\zeta=\zeta^{t}$
and $\beta_{t}=\sum_{i=0}^{t-1}\prod_{j=i+1}^{t-1}\left(\mathbf{1}_{j<t}\left(1-\gamma+\kappa^{2}\xi_{j}\right)+\mathbf{1}_{j=t}\right)\leq\frac{1-\zeta^{t}}{1-\zeta}.$ 

Further calculation yeilds $\nu_{t}=\left(c_{1}+c_{2}\beta_{t}\right)\left(x_{0}\alpha_{t}+c_{3}\beta_{t}\right)\leq\left(c_{1}+c_{2}\frac{1}{1-\zeta}\right)\left(x_{0}+c_{3}\frac{1}{1-\zeta}\right).$ 

As a result, the sufficient condition {\eqref{eq:suff condition M^PS}} for the
existence and uniqueness of $\mathbf{M}^{PS}$ is reduced to 
\begin{align}
	& \left(c_{1}+c_{2}\frac{1}{1-\zeta}\right)\left(x_{0}+c_{3}\frac{1}{1-\zeta}\right)\sum_{t=0}^{T-1}\left(\varepsilon_{t}\left(T-t\right)\right)<\widetilde{\mu}.\label{eq:prop suff 1}
\end{align}
This is equivalent to $\left(c_{1}+c_{2}\frac{1}{1-\zeta}\right)\left(x_{0}+c_{3}\frac{1}{1-\zeta}\right)\sum_{t=0}^{T-1}\frac{T-t}{T-H+1}\varepsilon_{t}<\overline{\mu},$
where $\overline{\mu}=\min\left\{ \frac{\mu\sigma^2}{2},\frac{\mu\sigma^{2}\gamma^{2}}{64\kappa^{10}}\right\} $.
As a result, to guarantee condition (\ref{eq:prop suff 1}) holds,
it is sufficient to ensure that
\begin{align*}
	& \sum_{t=0}^{T-1}\varepsilon_{t}<\left(1-\frac{H}{T}\right)\left(c_{1}+c_{2}\frac{1}{1-\zeta}\right)^{-1}\left(x_{0}+c_{3}\frac{1}{1-\zeta}\right)^{-1}\overline{\mu}.
\end{align*}
Proposition \ref{Proposition 1: Almost Sure Stable} is thus proved by choosing $\phi=\left(c_{1}+c_{2}\frac{1}{1-\zeta}\right)^{-1}\left(x_{0}+c_{3}\frac{1}{1-\zeta}\right)^{-1}.$

\section{Proof of Proposition \ref{Proposition 2: Almost Sure Unstable}}
Since $\widetilde{\zeta}\leq\left\Vert \mathbf{A}_{t}\right\Vert \leq1-\gamma+\kappa^{2}\xi_{t},\forall0\leq t<T,$
for some a positive constant $\widetilde{\zeta}>1,$ it follows directly
that $\alpha_{t}=\prod_{i=0}^{t-1}\left(1-\gamma+\kappa^{2}\xi_{i}\right)\geq\alpha_{t}\geq\prod_{i=0}^{t-1}\widetilde{\zeta}=\widetilde{\zeta}^{t}$
and $\beta_{t}=\sum_{i=0}^{t-1}\prod_{j=i+1}^{t-1}\left(\mathbf{1}_{j<t}\left(1-\gamma+\kappa^{2}\xi_{j}\right)+\mathbf{1}_{j=t}\right)\geq\frac{\widetilde{\zeta}^{t}-1}{\widetilde{\zeta}-1}.$ 

Further calculation yeilds $\nu_{t}=\left(c_{1}+c_{2}\beta_{t}\right)\left(x_{0}\alpha_{t}+c_{3}\beta_{t}\right)\geq\left(c_{1}+c_{2}\widetilde{\zeta}^{t-1}\right)\left(x_{0}\widetilde{\zeta}^{t}+c_{3}\widetilde{\zeta}^{t-1}\right).$
Therefore, 
\begin{align}
	& \sum_{i=t+1}^{T}\nu_{i}\geq\left(c_{1}x_{0}+\left(c_{1}c_{3}+c_{2}c_{3}+c_{3}x_{0}\right)\widetilde{\zeta}^{-1}\right)\sum_{i=t+1}^{T}\widetilde{\zeta}^{i}\nonumber\\
	&=\left(c_{1}x_{0}+\left(c_{1}c_{3}+c_{2}c_{3}+c_{3}x_{0}\right)\widetilde{\zeta}^{-1}\right)\frac{\widetilde{\zeta}^{T-t}-1}{\widetilde{\zeta}-1}.
\end{align}

As a result, to guarantee the sufficient condition \eqref{eq:suff condition M^PS} can be satifised,
we must have 
\begin{align}
	& \sum_{t=0}^{T-1}\frac{\varepsilon_{t}}{T-H+1}\left(c_{1}x_{0}+\left(c_{1}c_{3}+c_{2}c_{3}+c_{3}x_{0}\right)\widetilde{\zeta}^{-1}\right)\frac{\widetilde{\zeta}^{T-t}-1}{\widetilde{\zeta}-1}<\overline{\mu}.\label{eq:prop suff 2}
\end{align}
This means that 
\begin{align}
	\varepsilon_{t}< & \frac{\left(T-H+1\right)\left(\widetilde{\zeta}-1\right)}{c_{1}x_{0}+\left(c_{1}c_{3}+c_{2}c_{3}+c_{3}x_{0}\right)\widetilde{\zeta}^{-1}}\cdot\frac{\overline{\mu}}{\widetilde{\zeta}^{T-t}-1}.
\end{align}
As a result, Proposition \ref{Proposition 2: Almost Sure Unstable} is proved by letting $\phi=\frac{\widetilde{\zeta}-1}{c_{1}x_{0}+\left(c_{1}c_{3}+c_{2}c_{3}+c_{3}x_{0}\right)\widetilde{\zeta}^{-1}}.$

\end{document}